\documentclass[a4paper]{amsart}

\usepackage{etex}

\usepackage{pstricks}
\usepackage{subfig}
\usepackage{graphicx}
\usepackage{color}
\usepackage{amssymb}
\usepackage[all]{xy}
\usepackage{amsmath}
\usepackage{epic}
\usepackage{url}
\usepackage{setspace}
\usepackage{booktabs}
\usepackage{longtable}
\usepackage{nicefrac}
\usepackage{pst-3d}
\usepackage{pst-grad}
\usepackage{pst-3dplot}
\usepackage{tikz}

\newtheorem{theorem}{Theorem}
\newtheorem{lemma}[theorem]{Lemma}
\newtheorem{proposition}[theorem]{Proposition}
\newtheorem{corollary}[theorem]{Corollary}

\theoremstyle{definition}
\newtheorem{definition}[theorem]{Definition}
\newtheorem{example}[theorem]{Example}

\newtheorem{remark}[theorem]{Remark}

\definecolor{cUeberschrift}{rgb}{0.2,0.70,0.23}

\newcommand{\cR}{\mathcal{R}}
\newcommand{\cD}{\mathcal{D}}
\newcommand{\cL}{\mathcal{L}}
\newcommand{\cO}{\mathcal{O}}

\newcommand{\cP}{\mathcal{P}}
\newcommand{\cV}{\mathcal{V}}
\newcommand{\cE}{\mathcal{E}}
\newcommand{\cF}{\mathcal{F}}
\newcommand{\cT}{\mathcal{T}}
\newcommand{\cX}{\mathcal{X}}
\newcommand{\cM}{\mathcal{M}}
\newcommand{\cm}{\mathfrak{m}}

\newcommand{\Cstar}{\mathbb{C}^*}

\newcommand{\CC}{\mathbb{C}}
\newcommand{\RR}{\mathbb{R}}
\newcommand{\QQ}{\mathbb{Q}}
\newcommand{\ZZ}{\mathbb{Z}}
\newcommand{\NN}{\mathbb{N}}
\newcommand{\PP}{\mathbb{P}}
\renewcommand{\AA}{\mathbb{A}}
\newcommand{\marked}{\mathcal{M}}

\DeclareMathOperator{\rank}{rank}
\DeclareMathOperator{\tail}{tail}

\DeclareMathOperator{\im}{im}

\DeclareMathOperator{\SL}{SL}

\DeclareMathOperator{\Hom}{Hom}

\DeclareMathOperator{\divisor}{div}
\renewcommand{\div}{\divisor}
\DeclareMathOperator{\DivQ}{{Div_\QQ}}
\DeclareMathOperator{\cadiv}{CaDiv}

\DeclareMathOperator{\cl}{Cl}
\DeclareMathOperator{\wdiv}{Div}

\DeclareMathOperator{\CaSF}{CaSF}
\DeclareMathOperator{\tcadiv}{T-CaDiv}
\DeclareMathOperator{\Cox}{Cox}

\DeclareMathOperator{\Eff}{Eff}
\DeclareMathOperator{\Nef}{Nef}

\DeclareMathOperator{\Loc}{Loc}
\DeclareMathOperator{\Sing}{Sing}
\newcommand{\Sym}{\textrm{Sym}^\bullet}

\DeclareMathOperator{\spec}{Spec}
\DeclareMathOperator{\proj}{Proj}
\DeclareMathOperator{\ord}{ord}
\DeclareMathOperator{\orb}{orb}
\DeclareMathOperator{\dist}{dist}
\DeclareMathOperator{\supp}{supp}

\DeclareMathOperator{\relint}{relint}

\DeclareMathOperator{\id}{id}
\DeclareMathOperator{\TV}{TV}
\DeclareMathOperator{\coeff}{coeff}
\DeclareMathOperator{\coef}{coef}
\DeclareMathOperator{\face}{face}

\DeclareMathOperator{\conv}{conv}

\DeclareMathOperator{\spann}{span}
\DeclareMathOperator{\open}{open}
\DeclareMathOperator{\closed}{closed}
\DeclareMathOperator{\vol}{vol}

\newcommand{\hstar}{h^*}

\newcommand{\f}{{\mathfrak{f}}}

\newcommand{\fan}{\mathcal{S}}

\newcommand{\bangle}[1]{\langle\, #1 \,\rangle}
\newcommand{\an}[1]{#1^{\textnormal{an}}}

\newcommand{\ovl}{\overline}
\newcommand{\wt}{\widetilde}

\newcommand{\sdeg}{\mathfrak{deg}}

\definecolor{cKlaus}{rgb}{0.1,0.0,0.9}

\definecolor{cLars}{rgb}{0.1,0.60,0.43}

\definecolor{cHendrik}{rgb}{0.8,0.4,0.3}

\definecolor{cNathan}{rgb}{0.0,1.0,0.0}

\definecolor{cRobert}{rgb}{0.9,0.08,0.5}

\newcommand{\normal}{\mathcal{N}}
\newcommand{\tM}{\wt{M}}       
\newcommand{\tN}{\wt{N}}       
\newcommand{\tT}{\wt{T}}       
\newcommand{\kP}{{\mathcal P}}
\newcommand{\kV}{{\mathcal V}} 
\newcommand{\kR}{{\mathcal R}}
\newcommand{\kVR}{{\kV\cup\kR}}
\newcommand{\kFQ}{\QQ^{\kVR}}    
\newcommand{\kpR}{\QQ^{\kR}_{\geq 0}}    
\newcommand{\kpVR}{\QQ^{\kVR}_{\geq 0}}    
\newcommand{\pFan}{\mathcal{S}} 
\newcommand{\pDiv}{\mathcal{D}} 

\newcommand{\Z}{\ZZ}
\newcommand{\Q}{\QQ}
\DeclareMathOperator{\kDiv}{Div}
\newcommand{\surj}{\rightarrow\hspace{-0.8em}\rightarrow}
\newcommand{\ko}{\overline}
\DeclareMathOperator{\algGr}{algGr} 

\newcommand{\gHom}{\mbox{\rm Hom}}

\newcommand{\kG}{\Gamma}

\newcommand{\kst}{\,|\;}

\newcommand{\kk}{\CC}
\DeclareMathOperator{\Pol}{Pol}
\DeclareMathOperator{\Spec}{Spec}

\DeclareMathOperator{\CaDiv}{CaDiv}
\DeclareMathOperator{\Div}{Div}

\DeclareMathOperator{\Cl}{Cl}
\DeclareMathOperator{\innt}{int}
\newcommand{\CO}{{\mathcal O}}
\newcommand{\dual}{^{\scriptscriptstyle\vee}}
\newcommand{\then}{\Rightarrow}
\DeclareMathOperator{\loc}{loc}
\renewcommand{\iff}{\Leftrightarrow}

\newcommand{\til}[1]{\widetilde{#1}}
\DeclareMathOperator{\codim}{codim}
\newcommand{\chQ}{/\hspace{-0.3em}/^{\mbox{\tiny ch}}} 

\newcommand{\keps}{\varepsilon}

\newcommand{\kbb}{{\scriptstyle \bullet}}
\newcommand{\rato}{-\hspace{-0.3em}\to}  

\DeclareMathOperator{\toric}{TV}
\DeclareMathOperator{\ttoric}{\widetilde{TV}}
\newcommand{\ptoric}{\PP}

\newcommand{\ul}{\underline}

\newcommand{\cU}{\mathcal{U}}
\DeclareMathOperator{\suppFunc}{sf}
\DeclareMathOperator{\convHull}{convHull}
\DeclareMathOperator{\project}{pr}

\newgray{gray1}{0.4}
\newgray{gray2}{0.5}
\newgray{gray3}{0.6}
\newgray{gray4}{0.7}
\newgray{gray5}{0.8}
\newgray{gray6}{0.9}
\newgray{gray7}{0.3}

\newcommand{\smat}[1]{\left(\begin{smallmatrix}#1\end{smallmatrix}\right)}
\newcommand{\pmat}[1]{\begin{pmatrix}#1\end{pmatrix}}

\newcommand{\polytopeTV}{%
\psset{unit=1cm}
\begin{pspicture}(-1.5,-.5)(3,2)%
\pspolygon[linecolor=white,fillstyle=solid,fillcolor=gray3](0,2)(0,1)(1,0)(2,0)(2,2)%
\psline[linewidth=1.5pt]{->}(0,1)(0,2)
\psline[linewidth=1.5pt]{->}(1,0)(2,0)
\psline[linewidth=1.5pt](0,1)(1,0)
\psline{->}(0,1.5)(0.5,1.5)
\psline{->}(1.5,0)(1.5,0.5)
\psline{->}(0.5,0.5)(1,1)
\end{pspicture}}

\newcommand{\normalPolytopeTV}{%
\psset{unit=1cm}
\begin{pspicture}(-2,-.5)(4,2)%
\pspolygon[linecolor=white,fillstyle=solid,fillcolor=gray4](0,2)(0,0)(2,2)%
\pspolygon[linecolor=white,fillstyle=solid,fillcolor=gray5](2,0)(0,0)(2,2)%
\psline[linewidth=1.5pt]{->}(0,0)(0,2)
\psline[linewidth=1.5pt]{->}(0,0)(2,2)
\psline[linewidth=1.5pt]{->}(0,0)(2,0)
\end{pspicture}}

\newcommand{\affineTB}{%
\psset{xunit=0.6cm,yunit=.6cm}
\begin{pspicture}(-1,-1)(4,3)%
\pspolygon[linecolor=white,fillstyle=solid,fillcolor=gray3](4,0)(0,0)(0,1)(2,3)(4,3)%
\psline[linewidth=1.5pt](2,3)(0,1)(0,0)(4,0)
\psline[linewidth=1.5pt]{->}(0,1)(2,3)
\psline[linewidth=1.5pt]{->}(0,0)(4,0)
\psline{->}(2,0)(2,.75)
\psline{->}(0,.5)(.75,.5)
\psline{->}(1,2)(1.5,1.5)
\qdisk(0,0){2pt}
\qdisk(0,1){2pt}
\rput[tr]{0}(-.2,-.2){$F_1$}
\rput[br]{0}(-.2,1.2){$F_2$}
\end{pspicture}}

\newcommand{\quasiFanAffineTB}{%
\psset{xunit=.6cm,yunit=.6cm}
\begin{pspicture}(-1,-2)(4,2)%
\pspolygon[linecolor=white,fillstyle=solid,fillcolor=gray4](0,2)(0,0)(4,0)(4,2)%
\pspolygon[linecolor=white,fillstyle=solid,fillcolor=gray5](4,0)(0,0)(2,-2)(4,-2)%
\psline[linewidth=1.5pt]{->}(0,0)(0,2)
\psline[linewidth=1.5pt]{->}(0,0)(4,0)
\psline[linewidth=1.5pt]{->}(0,0)(2,-2)
\qdisk(0,0){2pt}
\rput[bl]{0}(1,1){$\normal(F_1,\Delta)$}
\rput[bl]{0}(1.5,-1){$\normal(F_2,\Delta)$}
\end{pspicture}}

\newcommand{\cotangfanHirzebruchNull}{%
\psset{xunit=0.5cm,yunit=.4cm}
\begin{pspicture}(-3.2,-3.2)(3.2,3.2)%
\pspolygon[linewidth=.001pt,linecolor=white,fillstyle=solid,fillcolor=gray1](0,3)(0,1)(2,3)%
\pspolygon[linewidth=.001pt,linecolor=white,fillstyle=solid,fillcolor=gray3](2,3)(0,1)(0,0)(3,0)(3,3)%
\pspolygon[linewidth=.001pt,linecolor=white,fillstyle=solid,fillcolor=gray5](3,0)(0,0)(0,-1)(2,-3)(3,-3)%
\pspolygon[linewidth=.001pt,linecolor=white,fillstyle=solid,fillcolor=gray4](2,-3)(0,-1)(0,-3)%
\pspolygon[linewidth=.001pt,linecolor=white,fillstyle=solid,fillcolor=gray2](0,-3)(0,-1)(-3,-1)(-3,-3)%
\pspolygon[linewidth=.001pt,linecolor=white,fillstyle=solid,fillcolor=gray5](-3,-1)(0,-1)(0,0)(-3,3)%
\pspolygon[linewidth=.001pt,linecolor=white,fillstyle=solid,fillcolor=gray3](-3,3)(0,0)(0,1)(-1,3)%
\pspolygon[linewidth=.001pt,linecolor=white,fillstyle=solid,fillcolor=gray6](-1,3)(0,1)(0,3)%
\psset{linewidth=1pt}%

\psline{-}(0,3)(0,1)(2,3)%
\psline{-}(2,3)(0,1)(0,0)(3,0)%
\psline{-}(3,0)(0,0)(0,-1)(2,-3)%
\psline{-}(2,-3)(0,-1)(0,-3)%
\psline{-}(0,-3)(0,-1)(-3,-1)%
\psline{-}(-3,-1)(0,-1)(0,0)(-2,2)%
\psline{-}(-3,3)(0,0)(0,1)(-1,3)%
\psline{-}(-1,3)(0,1)(0,3)%
\end{pspicture}}

\newcommand{\cotangfanHirzebruchInfty}{%
\psset{xunit=0.5cm,yunit=.4cm}
\begin{pspicture}(-3.2,-3.2)(3.2,3.2)%
\pspolygon[linewidth=.001pt,linecolor=white,fillstyle=solid,fillcolor=gray1](0,3)(0,0)(3,3)%
\pspolygon[linewidth=.001pt,linecolor=white,fillstyle=solid,fillcolor=gray3](3,3)(0,0)(3,0)%
\pspolygon[linewidth=.001pt,linecolor=white,fillstyle=solid,fillcolor=gray5](3,0)(0,0)(3,-3)%
\pspolygon[linewidth=.001pt,linecolor=white,fillstyle=solid,fillcolor=gray4](3,-3)(0,0)(0,-3)%
\pspolygon[linewidth=.001pt,linecolor=white,fillstyle=solid,fillcolor=gray2](0,-3)(0,0)(-1,1)(-3,1)(-3,-3)%
\pspolygon[linewidth=.001pt,linecolor=white,fillstyle=solid,fillcolor=gray5](-3,1)(-1,1)(-3,3)%
\pspolygon[linewidth=.001pt,linecolor=white,fillstyle=solid,fillcolor=gray3](-3,3)(-1,1)(-2,3)%
\pspolygon[linewidth=.001pt,linecolor=white,fillstyle=solid,fillcolor=gray6](-2,3)(-1,1)(0,0)(0,3)%

\psset{linewidth=1pt}%
\psline{-}(0,3)(0,0)(3,3)%
\psline{-}(3,3)(0,0)(3,0)%
\psline{-}(3,0)(0,0)(3,-3)%
\psline{-}(3,-3)(0,0)(0,-3)%
\psline{-}(0,-3)(0,0)(-1,1)(-3,1)%
\psline{-}(-3,1)(-1,1)(-3,3)%
\psline{-}(-3,3)(-1,1)(-2,3)%
\psline{-}(-2,3)(-1,1)(0,0)(0,3)%
\end{pspicture}}

\newcommand{\cotangfanHirzebruchOne}{%
\psset{xunit=0.5cm,yunit=.4cm}
\begin{pspicture}(-3.2,-3.2)(3.2,3.2)%
\pspolygon[linewidth=.001pt,linecolor=white,fillstyle=solid,fillcolor=gray1](0,3)(0,0)(1,0)(3,2)(3,3)%
\pspolygon[linewidth=.001pt,linecolor=white,fillstyle=solid,fillcolor=gray3](3,2)(1,0)(3,0)%
\pspolygon[linewidth=.001pt,linecolor=white,fillstyle=solid,fillcolor=gray5](3,0)(1,0)(3,-2)%
\pspolygon[linewidth=.001pt,linecolor=white,fillstyle=solid,fillcolor=gray4](3,-3)(3,-2)(1,0)(0,0)(0,-3)%
\pspolygon[linewidth=.001pt,linecolor=white,fillstyle=solid,fillcolor=gray2](0,-3)(0,0)(-3,0)(-3,-3)%
\pspolygon[linewidth=.001pt,linecolor=white,fillstyle=solid,fillcolor=gray5](-3,0)(0,0)(-3,3)%
\pspolygon[linewidth=.001pt,linecolor=white,fillstyle=solid,fillcolor=gray3](-3,3)(0,0)(-1.5,3)%
\pspolygon[linewidth=.001pt,linecolor=white,fillstyle=solid,fillcolor=gray6](-1.5,3)(0,0)(0,3)%

\psset{linewidth=1pt}%
\psline{-}(0,3)(0,0)(1,0)(3,2)%
\psline{-}(1,0)(3,0)%
\psline{-}(3,0)(1,0)(3,-2)%
\psline{-}(0,0)(0,-3)%
\psline{-}(0,0)(-3,0)%
\psline{-}(0,0)(-3,3)%
\psline{-}(0,0)(-1.5,3)%
\psline{-}(0,0)(0,3)%
\end{pspicture}}

\newcommand{\tailFanCotangHirzebruch}{%
\psset{xunit=0.5cm,yunit=.4cm}
\begin{pspicture}(-3.2,-3.2)(3.2,3.2)%
\pspolygon[linewidth=.001pt,linecolor=white,fillstyle=solid,fillcolor=gray1](0,3)(0,0)(3,3)%
\pspolygon[linewidth=.001pt,linecolor=white,fillstyle=solid,fillcolor=gray3](3,3)(0,0)(3,0)%
\pspolygon[linewidth=.001pt,linecolor=white,fillstyle=solid,fillcolor=gray5](3,0)(0,0)(3,-3)%
\pspolygon[linewidth=.001pt,linecolor=white,fillstyle=solid,fillcolor=gray4](3,-3)(0,0)(0,-3)%
\pspolygon[linewidth=.001pt,linecolor=white,fillstyle=solid,fillcolor=gray2](0,-3)(0,0)(-3,0)(-3,-3)%
\pspolygon[linewidth=.001pt,linecolor=white,fillstyle=solid,fillcolor=gray5](-3,0)(0,0)(-3,3)%
\pspolygon[linewidth=.001pt,linecolor=white,fillstyle=solid,fillcolor=gray3](-3,3)(0,0)(-1.5,3)%
\pspolygon[linewidth=.001pt,linecolor=white,fillstyle=solid,fillcolor=gray6](-1.5,3)(0,0)(0,3)%

\psset{linewidth=1pt}%
\psline{-}(0,3)(0,-3)%
\psline{-}(-3,0)(3,0)%
\psline{-}(3,3)(0,0)%
\psline{-}(3,-3)(-3,3)%
\psline{-}(0,0)(-1.5,3)%
\end{pspicture}}

\newcommand{\degreeCotangHirzebruch}{%
\psset{xunit=0.5cm,yunit=.4cm}
\begin{pspicture}(-3.2,-3.2)(3.2,3.2)%
\pspolygon[linewidth=.001pt,linecolor=white,fillstyle=solid,fillcolor=gray1](0,3)(0,1)(1,1)(3,3)%
\pspolygon[linewidth=.001pt,linecolor=white,fillstyle=solid,fillcolor=gray3](3,3)(1,1)(1,0)(3,0)%
\pspolygon[linewidth=.001pt,linecolor=white,fillstyle=solid,fillcolor=gray5](3,0)(1,0)(1,-1)(3,-3)%
\pspolygon[linewidth=.001pt,linecolor=white,fillstyle=solid,fillcolor=gray4](3,-3)(1,-1)(0,-1)(0,-3)%
\pspolygon[linewidth=.001pt,linecolor=white,fillstyle=solid,fillcolor=gray2](0,-3)(0,-1)(-1,0)(-3,0)(-3,-3)%
\pspolygon[linewidth=.001pt,linecolor=white,fillstyle=solid,fillcolor=gray5](-3,0)(-1,0)(-1,1)(-3,3)%
\pspolygon[linewidth=.001pt,linecolor=white,fillstyle=solid,fillcolor=gray3](-3,3)(-1,1)(-1,2)(-1.5,3)%
\pspolygon[linewidth=.001pt,linecolor=white,fillstyle=solid,fillcolor=gray6](-1.5,3)(-1,2)(0,1)(0,3)%
\pspolygon[linewidth=.001pt,linecolor=white,fillstyle=crosshatch,hatchsep=0.1,hatchangle=0](0,1)(1,1)(1,0)(1,-1)(0,-1)(-1,0)(-1,2)%

\psset{linewidth=1pt}%
\psline{-}(0,1)(1,1)(1,0)(1,-1)(0,-1)(-1,0)(-1,2)(0,1)%
\psline{-}(0,1)(0,3)
\psline{-}(1,1)(3,3)%
\psline{-}(1,0)(3,0)%
\psline{-}(1,-1)(3,-3)%
\psline{-}(0,-1)(0,-3)%
\psline{-}(-1,0)(-3,0)%
\psline{-}(-1,1)(-3,3)%
\psline{-}(-1,2)(-1.5,3)%
\end{pspicture}}

\newcommand{\boxCotangHirzebruchOne}{%
\psset{xunit=.7cm,yunit=.6cm}
\begin{pspicture}(-3,-2)(3,3)%
\pspolygon[linewidth=1pt,fillstyle=solid,fillcolor=gray5](0,-1)(2,0)(2,2)(0,2)(-2,0)(-1,-1)
\psgrid[gridwidth=0.3pt,griddots=7,subgriddiv=1,gridlabels=5pt](0,0)(-3,-2)(3,3)
\end{pspicture}}

\newcommand{\normalCrossing}{%
\psset{xunit=0.3cm,yunit=.4cm}
\begin{pspicture}(-5,-1.7)(5,1)%
\psline(-5,-0.2)(5,0.2)
\rput(6,0.2){$E_2$}
\psline(0,1)(0,-1)
\rput(-3,-1.6){$Z_1$}
\psline(-3,1)(-3,-1)
\rput(0,-1.6){$Z_2$}
\psline(3,1)(3,-1)
\rput(3,-1.6){$E_1$}
\end{pspicture}}

\newcommand{\cayleyResolution}{%
\psset{xunit=0.8cm,yunit=0.8cm,linestyle=solid, linecolor=black}
\begin{pspicture}(-2,-1)(3.5,5.5)%
\psline(0,0)(2,0)(0,3)(0,0)
\psset{linestyle=dotted,dotsep=1pt}
\psline(1,1)(2,0)
\psline(1,1)(0,3)
\psline(1,1)(0,2)
\psline(1,1)(0,0)
\psline(1,1)(1,0)
\psline(1,1)(0,1)
\rput(1.8,1.2){$\scriptstyle (1,1,1)$}
\rput(0.6,3.1){$\scriptstyle (1,0,3)$}
\rput(2.6,0){$\scriptstyle (1,2,0)$}
\rput(-0.6,0){$\scriptstyle (1,0,0)$}
\rput(-0.6,1){$\scriptstyle (1,0,1)$}
\rput(-0.6,2){$\scriptstyle (1,0,2)$}
\rput(1,-0.2){$\scriptstyle (1,1,0)$}
\end{pspicture}}

\newcommand{\rbox}[1]{\psframebox*[framesep=0.5pt,boxsep=false,framearc=.3]{#1}}
\newcommand{\varchenkoResolution}{%
\psset{xunit=6cm,yunit=6cm,linestyle=solid, linecolor=black, Alpha=56.5,Beta=17.6}
\begin{pspicture}(-0.8,-0.25)(1.1,1.1)%
\psset{linewidth=1pt}
\pstThreeDLine(0,0,1)(1,0,0)(0,1,0)(0,0,1)
\pstThreeDLine(0,1,0)(15,10,6)
\pstThreeDLine(1,0,0)(15,10,6)
\pstThreeDLine(0,0,1)(15,10,6)
\psset{linewidth=0.3pt,linestyle=dotted,dotsep=1pt}
\pstThreeDLine(3,2,1)(15,10,6)
\pstThreeDLine(3,2,1)(1,0,0)
\pstThreeDLine(3,2,1)(8,5,3)
\pstThreeDLine(3,2,1)(0,1,0)
\pstThreeDLine(3,2,1)(0,1,0)
\pstThreeDLine(3,2,1)(5,4,2)
\pstThreeDLine(3,2,1)(10,7,4)
\pstThreeDDot(10,7,4)
\pstThreeDDot(5,4,2)
\pstThreeDDot(3,2,2)
\pstThreeDDot(9,6,4)
\pstThreeDDot(8,5,3)
\pstThreeDDot(12,8,5)
\pstThreeDDot(6,4,3)

\pstThreeDLine(2,1,1)(1,0,0)
\pstThreeDLine(2,1,1)(0,0,1)
\pstThreeDLine(2,1,1)(15,10,6)
\pstThreeDLine(2,1,1)(12,8,5)
\pstThreeDLine(2,1,1)(9,6,4)
\pstThreeDLine(2,1,1)(6,4,3)
\pstThreeDLine(2,1,1)(3,2,2)
\pstThreeDLine(2,1,1)(5,3,2)
\pstThreeDLine(5,3,2)(1,0,0)
\pstThreeDLine(5,3,2)(8,5,3)
\pstThreeDLine(5,3,2)(15,10,6)

\pstThreeDLine(1,1,1)(0,1,0)
\pstThreeDLine(1,1,1)(0,0,1)
\pstThreeDLine(1,1,1)(15,10,6)
\pstThreeDLine(1,1,1)(12,8,5)
\pstThreeDLine(1,1,1)(9,6,4)
\pstThreeDLine(1,1,1)(6,4,3)
\pstThreeDLine(1,1,1)(3,2,2)

\pstThreeDLine(2,2,1)(0,1,0)
\pstThreeDLine(2,2,1)(5,4,2)
\pstThreeDLine(2,2,1)(10,7,4)
\pstThreeDLine(2,2,1)(15,10,6)
\pstThreeDLine(2,2,1)(1,1,1)
\pstThreeDLine(4,3,2)(15,10,6)
\pstThreeDLine(4,3,2)(2,2,1)
\pstThreeDLine(4,3,2)(1,1,1)
\pstThreeDLine(7,5,3)(4,3,2)
\pstThreeDLine(7,5,3)(2,2,1)
\pstThreeDLine(7,5,3)(15,10,6)

\pstThreeDPut(1,0,-0.1){$(1,0,0)$}
\pstThreeDPut(0,1,-0.1){$(0,1,0)$}
\pstThreeDPut(0,0,1.1){$(0,0,1)$}
\pstThreeDPut(14.9,10,6){\rbox{$\rho$}}
\pstThreeDPut(3,2,1){\rbox{$\scriptstyle (3,2,1)$}}
\pstThreeDPut(8,5,3){\rbox{$\scriptstyle (8,5,3)$}}
\pstThreeDPut(10,7,4){\rbox{$\scriptstyle (10,7,4)$}}
\pstThreeDPut(5,4,2){\rbox{$\scriptstyle (5,4,2)$}}
\pstThreeDPut(1,1,1){\rbox{$\scriptstyle (1,1,1)$}}
\pstThreeDPut(2,2,1){\rbox{$\scriptstyle (2,2,1)$}}
\pstThreeDPut(2,1,1){\rbox{$\scriptstyle (2,1,1)$}}
\pstThreeDPut(12,8,5){\rbox{$\scriptstyle (12,8,5)$}}
\pstThreeDPut(9,6,4){\rbox{$\scriptstyle (9,6,4)$}}
\pstThreeDPut(6,4,3){\rbox{$\scriptstyle (6,4,3)$}}
\pstThreeDPut(3,2,2){\rbox{$\scriptstyle (3,2,2)$}}
\pstThreeDPut(5,3,2){\rbox{$\scriptstyle (5,3,2)$}}
\pstThreeDPut(4,3,2){\rbox{$\scriptstyle (4,3,2)$}}
\pstThreeDPut(7,5,3){\rbox{$\scriptstyle (7,5,3)$}}
\end{pspicture}}

\def\defineTColor#1#2{%
  \newpsstyle{#1}{%
    fillstyle=vlines,hatchcolor=#2,%
    hatchwidth=0.1\pslinewidth,hatchsep=1\pslinewidth%
}}

\def\defineTHColor#1#2{%
  \newpsstyle{#1}{%
    fillstyle=hlines,hatchcolor=#2,%
    hatchwidth=0.1\pslinewidth,hatchsep=1\pslinewidth%
}}

\newgray{dgray}{0.65}
\defineTColor{TDGray}{dgray}
\defineTHColor{TD2Gray}{dgray}

 \newpsstyle{newtonCoor}{xMax=12,yMax=15,zMax=12,linecolor=gray, linestyle=dashed,   Alpha=190,Beta=10}
 \newpsstyle{newtonDualCoor}{nameX=$\,$,nameY=$\,$,nameZ=$\,$,xMax=12,yMax=15,zMax=12,linecolor=gray, linestyle=dashed,   Alpha=190,Beta=10}

\newcommand{\newtonDiagram}{
 \psset{unit=0.3cm}%
 \psset{style=newtonCoor}
 \begin{pspicture}(-3,-1)(15,13)%
   \pstThreeDCoor[style=newtonCoor]
   \psset{linestyle=solid,linewidth=1.5pt,fillstyle=none}%

   \psset{linewidth=0.3pt,linestyle=solid}%
   \pstThreeDLine(0,10,0)(0,0,10)(10,0,0)(0,10,0)
   \psset{linewidth=0.3pt,linestyle=solid}%

   \psset{style=TDGray}
   \pstThreeDLine(0,0,10)(0,0,5)(0,3,0)(0,10,0)
   
   \psset{style=TD2Gray}
   \pstThreeDLine(10,0,0)(2,0,0)(0,3,0)(0,10,0)
   \pstThreeDLine(0,0,10)(0,0,5)(2,0,0)(10,0,0)

   \psset{style=TDGray}
   \pstThreeDLine(0,3,0)(2,0,0)(0,0,5)(0,3,0)

 \end{pspicture}   
}

\newcommand{\newtonDual}{
 \psset{unit=0.3cm}%
 \psset{style=newtonCoor}
 \begin{pspicture}(-3,-1)(15,13)%
   \pstThreeDCoor[style=newtonDualCoor]
   \psset{linestyle=solid,linewidth=1.5pt,fillstyle=none}%

   \psset{linewidth=0.3pt,linestyle=solid}%
   \pstThreeDLine(0,0,10)(5,3.3,2)(10,0,0)
   \pstThreeDLine(0,0,10)(5,3.3,2)(0,10,0)
   \pstThreeDLine(0,10,0)(5,3.3,2)(10,0,0)
   \pstThreeDLine(0,10,0)(0,0,10)(10,0,0)(0,10,0)
   \psset{linewidth=0.3pt,linestyle=solid}%

   \psset{style=TDGray}
   \pstThreeDLine(10,0,0)(0,0,0)(0,10,0)

   \pstThreeDLine(5,3.3,2)(0,0,0)(0,0,10)
   \pstThreeDLine(5,3.3,2)(0,0,0)(10,0,0)
   \pstThreeDLine(5,3.3,2)(0,0,0)(0,10,0)

   \psset{style=TD2Gray}
   \pstThreeDLine(0,0,10)(0,0,0)(0,10,0)
   \pstThreeDLine(0,0,10)(0,0,0)(10,0,0)
 \end{pspicture}   
}

 \newpsstyle{cayleyCoor}{xMax=3,yMax=4,zMax=3.5,nameX=$N$, nameY=$e_1$, nameZ=$e_2$, linecolor=gray, linestyle=dashed,   Alpha=193,Beta=15}
\newcommand{\cayleyCone}{
 \psset{unit=1.1cm}%
 \psset{style=cayleyCoor}
 \begin{pspicture}(-1,-1)(4,4)%
   \pstThreeDCoor[style=cayleyCoor]
   \psset{linestyle=solid,linewidth=1.5pt,fillstyle=none}%

   \psset{linewidth=0.3pt}%

   \pstThreeDLine[linestyle=solid,linecolor=black,linewidth=0.3pt](0.5,1,0)(3.5,1,0)
   \pstThreeDLine[linestyle=dotted,dotsep=1pt,linewidth=1.2pt](0,0,1)(0.34,0,1)

   \psset{style=TDGray}
   \pstThreeDLine(2,4,0)(0,0,0)(2,0,0)

   \psset{style=TD2Gray}
   \pstThreeDLine(1,0,3)(0,0,0)(2,0,0)
   \pstThreeDLine(2,4,0)(0,0,0)(1,0,3)

   \pstThreeDLine(2,4,0)(0,0,0)(1,0,3)

   \pstThreeDLine[linestyle=dotted,dotsep=1pt,linewidth=1pt](0,1,0)(0.51,1,0)
   \pstThreeDLine[linestyle=solid,linecolor=black,linewidth=0.5pt](0.33,0,1)(3.5,0,1)
   \pstThreeDPut(4.2,1,0){$\pDiv_{Z_1} \!\times\! e_1$}
   \pstThreeDPut(4.2,0,1){$\pDiv_{E_2} \!\times\! e_2$}
   \psset{linecolor=black}
   \pstThreeDDot(1,2,0)
   \pstThreeDDot(1,0,3)
   \pstThreeDDot(1,0,0)
   \pstThreeDDot(0,1,0)
   \pstThreeDDot(0,0,1)
 \end{pspicture}   
}

\newcommand{\toroidalResI}{%
 \psset{unit=1cm}
 \begin{pspicture}(-1.2,-.2)(1.2,2.2)%
 \psset{linewidth=1pt}%
 \psline[fillstyle=solid,fillcolor=lightgray](-1,2)(0,0)(1,0)%
 \psline[linestyle=dotted,dotsep=1pt](0,1)%
 \psdot(-1,2)
 \psdot(0,1)
 \psdot(1,0)
 \rput(-0.4,2){$(-1,2)$}
\end{pspicture}}

\newcommand{\toroidalResII}{%
 \psset{unit=1cm}
 \begin{pspicture}(-1,-0.2)(2,3.2)%
 \psset{linewidth=1pt}%
 \psline[fillstyle=solid,fillcolor=lightgray](1,3)(0,0)(1,0)%
 \psline[linestyle=dotted,dotsep=1pt](1,1)%
 \psline[linestyle=dotted,dotsep=1pt](1,2)%
 \psdot(1,3)
 \psdot(1,2)
 \psdot(1,1)
 \psdot(1,0)
\rput(1.5,3){$(1,3)$}
 \end{pspicture}}

\newcommand{\toroidalResIII}{%
 \psset{unit=1cm}
 \begin{pspicture}(-1,-0.2)(2,5.2)%
 \psset{linewidth=1pt}%
 \psline[fillstyle=solid,fillcolor=lightgray](1,5)(0,0)(1,0)%
 \psline[linestyle=dotted,dotsep=1pt](1,1)%
 \psline[linestyle=dotted,dotsep=1pt](1,2)%
 \psline[linestyle=dotted,dotsep=1pt](1,3)%
 \psline[linestyle=dotted,dotsep=1pt](1,4)%
 \psdot(1,5)
 \psdot(1,4)
 \psdot(1,3)
 \psdot(1,2)
 \psdot(1,1)
 \psdot(1,0)
 \rput(1.5,5){$(1,5)$}
\end{pspicture}}

\newcommand{\canonical}{%
\psset{xunit=0.8cm,yunit=0.8cm,linestyle=solid, linecolor=black, Alpha=190,Beta=15}
\begin{pspicture}(-1.8,-1)(2,2)%
\psset{linewidth=1pt}
\psset{style=TDGray}
\pstThreeDLine[linestyle=none](-1.5,-2,1)(1.5,-2,1)(1.5,1.5,1)(-1.5,1.5,1)

\psset{style=TD2Gray}
\pstThreeDLine(0,1,1)(0,0,0)(1,0,1)(0,1,1)
\pstThreeDLine(0,-1,1)(0,0,0)(-1,0,1)(0,-1,1)
\pstThreeDLine(0,1,1)(0,0,0)(-1,0,1)(0,1,1)
\pstThreeDLine(0,-1,1)(0,0,0)(1,0,1)(0,-1,1)
\pstPlanePut[plane=xy](0,-2.7,1){\small height=1}

\pstThreeDDot(1,0,1)
\pstThreeDDot(-1,0,1)
\pstThreeDDot(0,1,1)
\pstThreeDDot(0,-1,1)
\pstThreeDDot(0,0,1)
\end{pspicture}}

\newcommand{\noncanonical}{%
\psset{xunit=0.8cm,yunit=1.6cm,linestyle=solid, linecolor=black, Alpha=190,Beta=10}
\begin{pspicture}(-2.8,-0.5)(3,2)%
\psset{linewidth=1pt}
\psset{style=TDGray}
\pstThreeDLine[linestyle=none](-1.5,-1.5,1)(1.5,-1.5,1)(1.5,1.5,1)(-1.5,1.5,1)

\psset{style=TD2Gray}
\pstThreeDLine(0,1,1)(0,0,0)(1,0,1)(0,1,1)
\pstThreeDLine(0,-1,1)(0,0,0)(-1,0,1)(0,-1,1)
\pstThreeDLine(0,1,1)(0,0,0)(-1,0,1)(0,1,1)
\pstThreeDLine(0,-1,1)(0,0,0)(1,0,1)(0,-1,1)
\pstPlanePut[plane=xy](0.2,-2,1){\small height=2}

\pstThreeDDot(1,0,1)
\pstThreeDDot(-1,0,1)
\pstThreeDDot(0,1,1)
\pstThreeDDot(0,-1,1)
\pstThreeDDot(0,0,0.5)
\end{pspicture}}

\newcommand{\terminal}{%
\psset{xunit=0.8cm,yunit=0.8cm,linestyle=solid, linecolor=black, Alpha=190,Beta=20}
\begin{pspicture}(-1.8,-1)(2,2)%
\psset{linewidth=1pt}
\psset{style=TDGray}
\pstThreeDLine[linestyle=none](-1,-1,1)(2,-1,1)(2,2,1)(-1,2,1)

\psset{style=TD2Gray}
\pstThreeDLine(0,0,1)(0,0,0)(1,0,1)(0,0,1)
\pstThreeDLine(1,0,1)(0,0,0)(1,1,1)(1,0,1)
\pstThreeDLine(1,1,1)(0,0,0)(0,1,1)(1,1,1)
\pstThreeDLine(0,1,1)(0,0,0)(0,0,1)(0,1,1)

\pstThreeDDot(1,0,1)
\pstThreeDDot(1,1,1)
\pstThreeDDot(0,1,1)
\pstThreeDDot(0,0,1)
\pstPlanePut[plane=xy](0.5,-1.7,1){\small height=1}
\end{pspicture}}

\newcommand{\fanhirza}{
 \psset{unit=1cm}
 \begin{pspicture}(-1.4,-1.4)(1.2,3)%
 \psgrid[gridwidth=0.3pt,griddots=5,subgriddiv=1,gridlabels=5pt,gridlabelcolor=white](-1,-1)(1,3)
\psline{<->}(0,-1)(0,1)
\psline{->}(1,1)
\psline{->}(-1,3)
\psline[linecolor=lightgray,linewidth=1.5pt](-1,1)(1,1)
\rput(-1.8,1){\gray{$[r=1]$}}
\rput(0,1.3){$\rho$}
  \end{pspicture}
}
\newcommand{\fanhirzb}{
 \psset{unit=1cm}
 \begin{pspicture}(-1.2,-1.4)(1.05,3)%
 \psgrid[gridwidth=0.3pt,griddots=5,subgriddiv=1,gridlabels=5pt,gridlabelcolor=white](-1,-1)(1,3)
\psline{<->}(0,-1)(0,1)
\psline{->}(1,-1)
\psline{->}(-1,3)
  \end{pspicture}
}

\newcommand{\hirzmink}{
 \psset{xunit=2cm}
 \begin{pspicture}(-1.4,-1.4)(1.4,1.2)%
\psline{<->}(-.6,1)(1.3,1)
\psline{<->}(-.6,0)(1.3,0)
\psline{<->}(-.6,-1)(1.3,-1)
\psline[linecolor=lightgray](-.33333,1)(0,-1)(0,1)(1,-1)
\rput(-1,1){$\{\Delta^0\}$}
\rput(-1,-1){$\{\Delta^1\}$}
\rput(-1.05,0){$\Sigma\cap[r=1]$}
\rput(-.33,1.3){$-\frac{1}{3}$}
\rput(0,1.3){$0$}
\rput(0,-1.3){$0$}
\rput(1,-1.3){$1$}
\end{pspicture}}

\newcommand{\nablapic}{
\psset{unit=.5cm}
\begin{pspicture}(-1,-1)(9,11)
\psdots[dotstyle=o](0,0)(3,0)(1,2)(4,2)(0,3)(3,3)(1,5)(7,5)(5,7)(8,7)(8,10)(5,10)(7,8)(4,8)
\psline(0,0)(3,0)(7,5)(7,8)(8,10)(8,7)(7,5)
\psline(8,10)(5,10)(4,8)(7,8)
\psline(0,0)(0,3)(3,3)(3,0)
\psline(3,3)(7,8)
\psline(0,3)(1,5)
\psline(0,3)(4,8)
\psline(1,5)(5,10)
\psline[linestyle=dashed,linecolor=gray](0,0)(1,2)(1,5)
\psline[linestyle=dashed,linecolor=gray](1,2)(4,2)(3,0)
\psline[linestyle=dashed,linecolor=gray](5,10)(5,7)(8,7)
\psline[linestyle=dashed,linecolor=gray](1,2)(5,7)
\psline[linestyle=dashed,linecolor=gray](4,2)(8,7)
\end{pspicture}}

\newcommand{\boxpic}{
\psset{unit=1cm}
\begin{pspicture}(-2.5,-2.5)(2.5,2.5)
   \psgrid[gridwidth=0.3pt,griddots=5,subgriddiv=1,gridlabels=5pt](-2,-2)(2,2)
\pspolygon(-1,-1)(-1,0)(0,1)(1,1)(1,0)(0,-1)(-1,-1)
\end{pspicture}
}
\newcommand{\psizero}{
\psset{unit=1cm}
\begin{pspicture}(-1.5,-2.5)(1.5,5.5)
\pspolygon(-1,-1)(-1,0)(0,1)(1,1)(1,0)(0,-1)(-1,-1)
\psline(1,0)(0,0)(0,1)
\psline(-1,-1)
\rput(1.2,1.2){$1$}
\rput(0,1.2){$1$}
\rput(1.2,0){$1$}
\rput(0.2,0.2){$1$}
\rput(-1.2,-1.2){$0$}
\rput(0,-1.2){$0$}
\rput(-1.2,0){$0$}
\end{pspicture}
}
\newcommand{\psiinfty}{
\psset{unit=1cm}
\begin{pspicture}(-1.5,-2.5)(1.5,5.5)
\pspolygon(-1,-1)(-1,0)(0,1)(1,1)(1,0)(0,-1)(-1,-1)
\psline(-1,0)(0,0)(0,-1)
\psline(1,1)
\rput(1.2,1.2){$0$}
\rput(0,1.2){$0$}
\rput(1.2,0){$0$}
\rput(-0.2,-0.2){$1$}
\rput(-1.2,-1.2){$1$}
\rput(0,-1.2){$1$}
\rput(-1.2,0){$1$}\end{pspicture}
}

\newcommand{\sigmaMinusVier}{%
\begin{tikzpicture}
  [latticept/.style={circle,fill,inner sep=0.3mm}]
\begin{scope}[scale=0.5]

\draw[thick,fill=black!5] (-1.2,2.4) -- (0,0) -- (1.2,2.4);

\draw[->] (-1.5,0) -- (1.8,0);
\draw[->] (0,-0.5) -- (0,2.8);

\foreach \x in {-1,...,1}
  \foreach \y in {0,...,2}
    \node[latticept] at (\x,\y) {};

\end{scope}
\end{tikzpicture}
}%

\newcommand{\sigmaDualMinusVier}{%
\begin{tikzpicture}
  [latticept/.style={circle,fill,inner sep=0.3mm},
   varlabel/.style={above,fill=white,inner sep=0.1mm,outer sep=1mm}]
\begin{scope}[scale=0.5]

\draw[thick,fill=black!5] (-2.4,1.2) -- (0,0) -- (2.4,1.2);

\draw[->] (-2.5,0) -- (2.8,0);
\draw[->] (0,-0.5) -- (0,2.2);

\foreach \x in {-2,...,2}
  \foreach \y in {0,...,1}
    \node[latticept] at (\x,\y) {};

\foreach \x/\cap in {-2/0,-1/1,1/3,2/4}
  \node[varlabel] at (\x,1) {$\scriptstyle y_{\cap}$};
\node[varlabel] at (0,1) {$\scriptstyle y_2$};

\end{scope}
\end{tikzpicture}
}%

\newcommand{\minusVierQR}{%
\begin{tikzpicture}
  [latticept/.style={circle,fill,inner sep=0.3mm}]

\begin{scope}[scale=0.7]

\draw[thick,fill=black!5] (-1.2,2.4) -- (0,0) -- (1.2,2.4);

\draw[->] (-1.5,0) -- (1.8,0);
\draw[->] (0,-0.5) -- (0,2.8);

\foreach \x in {-1,...,1}
  \foreach \y in {0,...,2}
    \node[latticept] at (\x,\y) {};

\draw (-1.2,1) -- (1.2,1);
\node[right] at (1.2,1) {$\scriptstyle [r = 1]$};
\draw[very thick] (-0.5,1) -- (0.5,1);
\node[draw,circle,fill=white,inner sep=0.4mm] at (-0.5,1) {};
\node[draw,circle,fill=white,inner sep=0.4mm] at (0.5,1) {};
\end{scope}
\end{tikzpicture}
}%

\newcommand{\minusVierGrid}{%
  \draw[gray,->] (-1.2,0) -- (1.4,0);
  \foreach \x in {-1,...,1} {
    \node[circle,fill=gray,inner sep=0.3mm] at (\x,0) {};
    \node at (\x,-0.4) {$\scriptscriptstyle \x$};
  };
}%

\newcommand{\minusVierDecomp}{%
\begin{tikzpicture}
\matrix[ampersand replacement=\&,row sep=4]{
  \begin{scope}[scale=0.6]
  \minusVierGrid
  \draw[very thick] (-0.5,0) -- (0.5,0);
  \node[draw,circle,fill=white, inner sep=0.4mm] at (-0.5,0) {};
  \node[draw,circle,fill=white, inner sep=0.4mm] at (0.5,0) {};
  \end{scope}
  \&
  \node at (0,0) {$=$};
  \&
  \begin{scope}[scale=0.6]
  \minusVierGrid
  \draw[very thick] (-0.5,0) -- (0,0);
  \node[draw,circle,fill=white, inner sep=0.4mm] at (-0.5,0) {};
  \node[draw,circle,fill,inner sep=0.4mm] at (0,0) {};
  \end{scope}
  \&
  \node at (0,0) {$+$};
  \&
  \begin{scope}[scale=0.6]
  \minusVierGrid
  \draw[very thick] (0,0) -- (0.5,0);
  \node[draw,circle,fill,inner sep=0.4mm] at (0,0) {};
  \node[draw,circle,fill=white, inner sep=0.4mm] at (0.5,0) {};
  \end{scope}
  \\
  \&
  \node at (0,0) {$=$};
  \&
  \begin{scope}[scale=0.6]
  \minusVierGrid
  \node[draw,circle,fill=white, inner sep=0.4mm] at (-0.5,0) {};
  \end{scope}
  \&
  \node at (0,0) {$+$};
  \&
  \begin{scope}[scale=0.6]
  \minusVierGrid
  \draw[very thick] (0,0) -- (1,0);
  \node[draw,circle,fill,inner sep=0.4mm] at (0,0) {};
  \node[draw,circle,fill, inner sep=0.4mm] at (1,0) {};
  \end{scope}
  \\
};
\end{tikzpicture}
}%

\newcommand{\tvardefBase}{%
\begin{tikzpicture}

\draw[thick] (-0.7,-0.7) -- (0.7,0.7) node[above]{$\scriptstyle D_0^1$};
\draw[thick] (0,-0.7) -- (0,0.7) node[above]{$\scriptstyle D_0^0$};
\draw[thick] (1.3,-0.7) -- (1.3,0.7) node[above]{$\scriptstyle D_\infty$};
\draw[dashed] (-0.8,0) -- (1.6,0);

\draw[->] (1.8,0) -- (2.2,0);

\begin{scope}[xshift=2.4cm]
  \draw[thick] (0,-0.7) -- (0,0.7);
  \node at (0,0) [inner sep=0.3mm,circle,fill] {};
  \node[right] at (0,0) {$\scriptstyle 0$};
  \node at (0,0.5) [inner sep=0.3mm,circle,fill] {};
  \node[right] at (0,0.5) {$\scriptstyle t$};
\end{scope}

\end{tikzpicture}
}%

\newcommand{\tvardefDivisors}{%
\begin{tikzpicture}

\foreach \y/\lbl in { 0.3/$\scriptstyle\mathcal{E}_t$,
                      -0.3/$\scriptstyle\mathcal{D} = \mathcal{E}_0$ }
  \draw[thick] (1.6,\y) -- (-0.7,\y) node[left]{\lbl};

\foreach \x/\lbl in {
      0/$\scriptstyle \Delta_0^0$/,
      0.5/$\scriptstyle \Delta_0^1$/,
      1.3/$\scriptstyle \Delta_\infty$/
    }
{
  \node[inner sep=0.3mm,circle,fill] at (\x,0.3) {};
  \node[above] at (\x,0.3) {\lbl};
}

\foreach \x/\lbl in {
      0/$\scriptstyle \Delta_0^0 + \Delta_0^1$/,
      1.3/$\scriptstyle \Delta_\infty$/
    }
{
  \node[inner sep=0.3mm,circle,fill] at (\x,-0.3) {};
  \node[below] at (\x,-0.3) {\lbl};
}

\foreach \x/\lbl in {0/$\scriptstyle 0$, 0.5/$\scriptstyle t$/,
                     1.3/$\scriptstyle \infty$}
  \node at (\x,0) {\lbl};

\end{tikzpicture}
}%

\SelectTips{cm}{10}

\begin{document}

\title[The Geometry of $T$-Varieties]
{The Geometry of $T$-Varieties}

\author[K.~Altmann]{Klaus Altmann}
\address{Institut f\"ur Mathematik und Informatik,
        Freie Universit\"at Berlin,
        Arnimallee 3,
        14195 Berlin, Germany}
\email{altmann@math.fu-berlin.de}
\author[N.O.~Ilten]{Nathan Owen Ilten}
\address{Max Planck Institut f\"ur Mathematik,
        PF 7280,
        53072 Bonn, Germany}
\email{nilten@cs.uchicago.edu}
\author[L.~Petersen]{Lars Petersen}
\address{Institut f\"ur Mathematik und Informatik,
        Freie Universit\"at Berlin,
        Arnimallee 3,
        14195 Berlin, Germany}
\email{petersen@math.fu-berlin.de}
\author[H.~S\"u\ss{}]{Hendrik S\"u\ss{}}
\address{Institut f\"ur Mathematik, LS Algebra und Geometrie,
        Brandenburgische Technische Universit\"at Cottbus,
        PF 10\,13\,44,
        03013 Cottbus, Germany}
\email{suess@math.tu-cottbus.de}
\author[R.~Vollmert]{Robert Vollmert}
\address{Institut f\"ur Mathematik und Informatik,
        Freie Universit\"at Berlin,
        Arnimallee 3,
        14195 Berlin, Germany}
\email{vollmert@math.fu-berlin.de}

\begin{abstract}
This is a survey of the language of polyhedral divisors describing
$T$-varieties. This language is explained in parallel to the well established theory of toric varieties.
In addition to basic constructions, subjects touched on include singularities, separatedness and properness, divisors and intersection theory, cohomology, Cox rings, polarizations, and equivariant deformations, among others.
\end{abstract}

\maketitle

\setcounter{tocdepth}{1}
\tableofcontents

\section{Introduction}
\label{sec:intro}

\subsection{$\CC^*$-actions and toric varieties}
\label{subsec:introToric}
The present paper is a survey about complex $T$-varieties, i.e.\ about normal
($n$-dimensional) varieties $X$ over $\CC$ with the effective action of a torus $T:=(\CC^*)^k$.
The case $k=1$ is a classical one --- especially singularities with
a so-called ``good'' $\CC^*$-action have been studied intensively
by Pinkham \cite{0304.14006,0331.14018,0351.14004}.
The case $n=k$ is classical as well --- first studied by Demazure \cite{demazureTV}, such varieties are called ``toric varieties''. The basic theory encodes the category of toric varieties and 
equivariant morphisms in purely combinatorial terms.
Many famous theories in algebraic geometry have their combinatorial
counterpart and thus illustrate the geometry from an alternative point
of view. The dictionary between algebraic geometry and combinatorics sometimes even helps to establish new theories. The most prominent example of this phenomenon might be the development of mirror symmetry in the 1990's. Building on the notion of reflexive polytopes, Batyrev gave a first systematic mathematical treatment of this subject \cite{batyrevMS}.

\subsection{Higher complexity}
\label{subsec:introComplex}
For some features, however, it is useful to consider lower dimensional torus
actions as well. For instance, when deforming a toric variety, the embedded torus $T$ still acts naturally on the total- and base-spaces $\til{X}\to S$, but it is usually too
small to provide a toric structure on them. The adjacent fibers $X_s$
are even worse: depending on the isotropy group of $T$ at
the point $s\in S$, only a subtorus of $T$ still acts.
Thus, it is worth to study the general case as well. The difference
$n-k$ is then called the complexity of a $T$-variety $X$.
\\[1ex]
While complexity $0$ means toric, the next case of complexity $1$ was 
systematically studied by Timashev (even for more general algebraic groups) in \cite{timGVar}. On the other hand, Flenner and Zaidenberg \cite{1093.14084} gave a very useful description of $\CC^*$-surfaces even for ``non-good'' actions.
Starting in 2003, there is a series of papers dealing with a general treatment
of $T$-varieties in terms of so-called polyhedral divisors. The idea is to
catch the non-combinatorial part of a $T$-variety $X$ in an
$(n-k)$-dimensional variety $Y$ which is a sort of
quotient $Y=X/T$. Now, $X$ can be described by presenting a 
``polyhedral'' divisor $\pDiv$ on $Y$ with coefficients being not numbers but instead convex polyhedra
in the vector space $N_\QQ:=N\otimes_\ZZ\QQ$ where $N$ is the lattice of
one-parameter subgroups of $T$.

\subsection{What this paper is about}
\label{subsec:whatAbout}
The idea of the present paper is to give an introduction to this subject
and to serve as a survey for the many recent papers on $T$-varieties.
Moreover, since the notion of polyhedral divisors and the theory of
$T$-varieties closely follows the concept of toric varieties,
we will treat both cases in parallel. This means that the present 
paper also serves as a quick introduction to the fascinating field of toric varieties.
For a broader discussion and detailed proofs of facts related to toric geometry which are merely mentioned here, the reader is asked to consult any of the standard textbooks like 
\cite{mumford}, \cite{danilov}, \cite{fulton}, and \cite{oda}.
The subjects of non-toric $T$-varieties are covered in
\cite{tvar_0},
\cite{tvar_1},
\cite{tvar_2},
\cite{fansy},
\cite{tcodes},
\cite{candiv},
\cite{tidiv},
\cite{NathHend},
\cite{defrat_tvar},
\cite{andreas+nathan},
\cite{toroidal},
\cite{tvarcox},
\cite{coxmds}, and
\cite{coxComp1}.

There are also applications of the theory of $T$-varieties and polyhedral divisors
in affine geometry \cite{alvaro08,alvaro09,alvaro10} and coding theory \cite{tcodes}
which are not covered by this survey. Very recently, $\SL_2$-actions on affine T-varieties
have also been studied \cite{1105.4494}.

\section{Affine $T$-varieties}
\label{sec:affine}

\subsection{Affine toric varieties}
\label{subsec:affToric}
Let $M$ and $N$ be two mutually dual, free abelian groups of rank $k$.
In other words, both are isomorphic to $\ZZ^k$, and we have a natural
perfect pairing $M\times N \to \Z$.
Then $T := \Spec\CC[M] = N \otimes_\ZZ\CC^*$ is the coordinate free version 
of the torus $(\CC^*)^k$ mentioned in (\ref{subsec:introToric}).
On the other hand, $N=\gHom_{\algGr}(\CC^*,T)$ is the set of one-parameter subgroups of $T$,
and $M = \gHom_{\algGr}(T,\CC^*)$ equals the character group of $T$.
We denote by $M_\QQ := M\otimes_\ZZ\QQ$ and $N_\QQ := N\otimes_\ZZ\QQ$ 
the associated $\QQ$-vector spaces.
\\[1ex]
We always assume that the generators $a^i$ for a polyhedral cone 
$\sigma = \langle a^1,\ldots,a^m\rangle := \sum_{i=1}^m \QQ_{\geq 0} \,a^i \subseteq N_\QQ$ are primitive elements of $N$, i.e.\ they are not proper multiples of other elements of $N$. Dropping this assumption is the first step towards the theory of toric stacks which we will not pursue here, cf.\ \cite{toricDMStacks}.
Moreover, we will often identify a primitive element
$a\in N$ with the ray $\QQ_{\geq 0}\cdot a$ it generates.
Given $\sigma$ as above, we define its dual cone as
$\sigma\dual := \{u\in M_\QQ\kst \langle\sigma,u\rangle\geq 0\}$.
If $\sigma$ does not contain a non-trivial linear subspace (i.e.\
$0\in\sigma$ is a vertex), then $\dim\sigma\dual=k$, and we 
use the semigroup algebra 
$\CC[\sigma\dual\cap M] := \oplus_{u\in\sigma\dual\cap M}\CC\,\chi^u$
to define the $k$-dimensional \emph{toric variety}
$$
\toric(\sigma) := \toric(\sigma,N) := \Spec\CC[\sigma\dual\cap M].
$$
This is a normal variety.
If $\dim\sigma=k$, the (finite) set $E\subseteq \sigma\dual\cap M$
of indecomposable elements is called the \emph{Hilbert basis} of $\sigma\dual$. Furthermore, $\toric(\sigma)\subset\CC^E$ is defined by a binomial ideal which arises from the relations between the elements of the Hilbert basis $E$, cf.\ (\ref{subsec:mapsT}).

\begin{example}\label{ex:minus4}
Let $\sigma=\langle (-1,2),\,(1,2)\rangle\subseteq\QQ^2$.
This gives rise to $\sigma\dual=\langle[-2,1],\,[2,1]\rangle$ and
$E=\{[e,1]\kst -2\leq e \leq 2\}$, illustrated in Figure~\ref{fig:minus4}.
 The resulting
$\toric(\sigma)\subseteq\CC^5$ is the cone over the rational normal curve of
degree $4$, and its defining ideal is generated by the six minors expressing
the inequality
$\rank\smat{
  y_0 & y_1 & y_2 & y_3 \\
  y_1 & y_2 & y_3 & y_4
}\leq 1.$
\end{example}

\begin{figure}
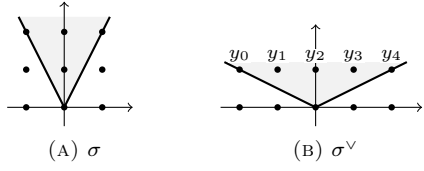

\centering
\subfloat[$\sigma$]{\hspace*{1ex}\sigmaMinusVier\hspace*{1ex}}{\hspace*{5ex}}
\subfloat[$\sigma\dual$]{\sigmaDualMinusVier}
\caption{Cones for Example~\ref{ex:minus4}.}
\label{fig:minus4}
\end{figure}

\subsection{Toric bouquets}
\label{subsec:dissToric}
The previous notion allows for a generalization. If
$\Delta\subseteq N_\QQ$ is a polyhedron, then we denote by
$$
\tail(\Delta) := \{a\in N_\QQ\kst a+\Delta\subseteq \Delta\}
$$
its so-called \emph{tailcone}. Assume that this cone is pointed, i.e. contains no non-trivial linear subspace; then $\Delta$ has a non-empty set of  vertices $\kV(\Delta)$. Denoting their compact convex hull by $\Delta^c$, the polyhedron $\Delta$ splits 
into the ``Minkowski'' sum $\Delta^c+\tail(\Delta)$.
Moreover, $\Delta$ gives rise to its \emph{inner normal fan}
$\normal(\Delta)$ consisting of the linearity regions of the
function $\min\langle\Delta,\kbb\rangle:\tail(\Delta)\dual\to\QQ$.
Note that the cones of $\normal(\Delta)$ are in a one-to-one correspondence to the faces $F\leq\Delta$ via the map $F\mapsto\normal(F,\Delta) := \{u\in M_\QQ\kst
\langle F,u \rangle = \min \langle \Delta, u\rangle\}$.
We can now define
$$
\toric(\Delta) := \Spec\CC[\normal(\Delta)\cap M]
$$
where $\CC[\normal(\Delta)\cap M] := \CC[\sigma\dual\cap M]$ as a 
$\CC$-vector space. Note, however, that the multiplication is given by
$$
\chi^u\cdot\chi^{u'} := \left\{
\begin{array}{ll}
\chi^{u+u'} & \mbox{if $u,u'$ belong to a common cone of $\normal(\Delta)$}\\
0 & \mbox{otherwise}.
\end{array}\right.
$$
A scheme of the form $\toric(\Delta)$ is called a \emph{toric bouquet}.
We
obtain the decomposition
$$
\toric(\Delta) =
\bigcup_{v\in\kV(\Delta)}\toric(\QQ_{\geq 0}\cdot (\Delta-v))
$$
into irreducible (toric) components, compare with (\ref{subsec:toricOrbits}).
Note that there is a one-to-one correspondence between polyhedral cones
and affine toric varieties, whereas the construction of a toric bouquet only depends on the normal fan $\normal(\Delta)$. Hence, dilations of the compact edges of $\Delta$ and translations do not affect the structure of the corresponding bouquet.

\begin{example}
\label{ex:affineTB}
Consider the cone $\sigma := \langle (1,0),(1,1)\rangle \subset \QQ^2$ together with the $\sigma$-polyhedron $\Delta = \ovl{(0,0)(0,1)} + \sigma = \Delta^c + \tail(\Delta)$ and its inner normal fan $\normal(\Delta)$ as depicted in Figure \ref{fig:affineTB}. Observe that $\spec \CC[\normal(\Delta) \cap M]$ is equidimensional and consists of two irreducible components isomorphic to $\AA^2$ which are glued along an affine line. 

\begin{figure}[t]
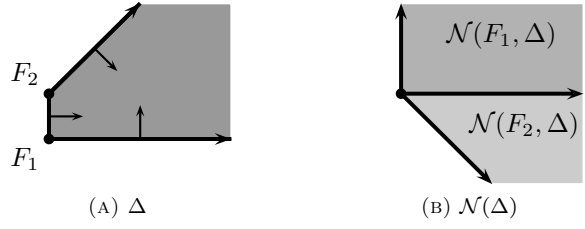

\centering
\subfloat[$\Delta$]{\affineTB}\hspace*{10ex}
\subfloat[$\normal(\Delta)$]{\quasiFanAffineTB}
\caption{An affine toric bouquet, cf.\ Example \ref{ex:affineTB}.}
\label{fig:affineTB}
\end{figure}

\end{example}

\subsection{Polyhedral and p-divisors}
\label{subsec:pDiv}
Let $M$, $N$, $\sigma\subseteq N_\QQ$ be as in (\ref{subsec:affToric}). In particular, we assume that $\sigma$ contains no non-trivial linear subspace. These data then give rise to a semigroup (with respect to Minkowski addition)
$$ 
\Pol^+_\Q(N,\sigma) := \{\Delta\subseteq N_\Q\kst
\Delta \mbox{ polyhedron with } \tail(\Delta)=\sigma\}.
$$
Note that $\Pol^+_\Q(N,\sigma)$ satisfies the cancellation law and contains $\sigma$ as its neutral element.
If, in addition, $Y$ is a normal, projective variety over $\CC$,
then we denote by $\CaDiv_{\geq 0}(Y)$ the semigroup of 
effective Cartier divisors and call elements
$$
\pDiv = \sum_Z \pDiv_Z\otimes Z\in 
\Pol^+_\Q(N,\sigma)\otimes_{\Z_{\geq 0}} \CaDiv_{\geq 0}(Y)
$$
\emph{polyhedral divisors} on $Y$ in $N$ (or simply on $(Y,N)$) with $\tail(\pDiv) := \sigma$. We will also allow $\emptyset$ as an element of $\Pol^+_\Q(N,\sigma)$ which satisfies $\emptyset+\Delta := \emptyset$. This somewhat bizarre coefficient is needed to deal
with non-compact, open \emph{loci} defined as
$\Loc(\pDiv) := Y\setminus\bigcup_{\pDiv_Z=\emptyset}Z$.
For each $u \in (\tail(\pDiv))\dual$, we may then consider the evaluation
$$
\pDiv(u) := \sum_{\pDiv_Z\neq \emptyset} 
\min\langle \pDiv_Z,u\rangle\cdot Z|_{\Loc\pDiv}
\in \CaDiv_\Q(\Loc\pDiv).
$$
This is an ordinary $\QQ$-divisor with 
$\supp\pDiv(u)\subseteq\supp\pDiv :=
\Loc\pDiv\cap\bigcup_{\pDiv_Z\neq \emptyset} Z$.

\begin{definition}
\label{def:pDiv}
$\pDiv$ is called a \emph{p-divisor} if all evaluations
$\pDiv(u)$ are semiample (i.e.\ have positive, base point free
multiples) and, additionally, are big for $u\in\innt(\tail(\pDiv))\dual$.
Note that this condition is void if $\Loc(\pDiv)$ is affine.
\end{definition}

Polyhedral divisors have the property that
$\pDiv(u)+\pDiv(u')\leq \pDiv(u+u')$. Hence, they lead to a sheaf of rings $\CO_{\Loc}(\pDiv) := 
\bigoplus_{u\in\sigma^\vee\cap M} \CO_{\Loc}(\pDiv(u))\,\chi^u$
giving rise to the schemes
$$
\ttoric(\pDiv) := \Spec_{\Loc(\pDiv)} \CO(\pDiv) 
\mbox{ and the affine }
\toric(\pDiv) := \Spec \kG(\Loc(\pDiv),\CO(\pDiv)).
$$

The latter space does not change if $\pDiv$ is pulled back via a
birational modification $Y'\to Y$ or if $\pDiv$ is altered by a
\emph{principal polyhedral divisor} on $Y$, that is, an element in the image of the natural map 
$N\otimes_\Z\kk(Y)^*\to \Pol_\Q(N,\sigma)\otimes_\Z \CaDiv(Y)$. 
Two p-divisors which differ by chains of
those operations are called \emph{equivalent}. Note that this implies that $Y$ can always be replaced by a log-resolution.

\begin{theorem}[\cite{tvar_1}, Theorems (3.1), (3.4); Corollary (8.12)]
\label{thm:equivPT}
The map $\pDiv\mapsto \TV(\pDiv)$ yields a bijection between
equivalence classes of p-divisors and normal, affine $\CC$-varieties with an effective $T$-action.
\end{theorem}

\subsection{Products of $T$-varieties}
\label{subsec:productsTV}
Consider some $T$-variety $X$ and a $T'$-variety $X'$.  Then the product $X\times X'$ carries the natural structure of a $T\times T'$-variety. As we shall see, the combinatorial data describing $X\times X'$ arise as a kind of product of the combinatorial data describing $X$ and $X'$.

For simplicity, we will only consider the affine case, although the construction easily globalizes. Thus, consider p-divisors $\pDiv,\pDiv'$ on respectively $Y,Y'$, with tailcones $\sigma,\sigma'$, respectively. We define the \emph{product} p-divisor $\pDiv\times\pDiv'$ on $Y\times Y'$ as follows:
\[
\pDiv\times\pDiv'=\sum_{Z\subset Y} (\pDiv_Z\times \sigma')\otimes (Z\times Y')+\sum_{Z'\subset Y'} (\sigma\times\pDiv_{Z'}')\otimes (Y\times Z')
\]

\begin{proposition}
For any p-divisors $\pDiv$ and $\pDiv'$, $\toric(\pDiv\times\pDiv')\cong \toric(\pDiv)\times\toric(\pDiv')$.
\end{proposition}

\begin{proof}
For any $\widetilde{u}=(u,u')\in M\oplus M'$, 
\begin{align*}
\Gamma\big(\Loc(\pDiv\times\pDiv'),\CO(\pDiv\times\pDiv'(\widetilde{u}))\big)\hspace*{-2pt}&= \hspace*{-1pt}\Gamma\big(\Loc(\pDiv\times\pDiv'),\CO(\pDiv(u)\times Y')\otimes\CO(Y\times\pDiv'( u'))\big)\\
&=\hspace*{-1pt}\Gamma\big(\Loc(\pDiv),\CO(\pDiv(u))\big)\otimes\Gamma\big(\Loc(\pDiv'),\CO(\pDiv'(u'))\big)
\end{align*}
where the last equality is due to the K\"unneth formula for coherent sheaves, see \cite[Theorem 6.7.8]{EGAIII2}.
\end{proof}

\begin{remark}
A special case of the above is when $X$ and $X'$ are toric varieties. Here, the proposition simplifies to $\toric(\sigma\times\sigma')\cong \toric(\sigma)\times\toric(\sigma')$.
\end{remark}

\section{The functor $\toric$}
\label{sec:functorTV}

\subsection{Maps between toric varieties}
\label{subsec:mapsToric}
We will now see that all constructions from the previous section
are functorial. Let us start with the setting given in (\ref{subsec:affToric}).
A $\ZZ$-linear map $F\colon N'\to N$ satisfying $F_\Q(\sigma')\subseteq\sigma$
gives rise to a morphism $\toric(F)\colon \toric(\sigma',N')\to \toric(\sigma,N)$
of affine toric varieties via $F\dual(\sigma\dual\cap M)\subseteq(\sigma')\dual\cap M'$. For example, if $E\subseteq\sigma\dual\cap M$ is a Hilbert basis, then the embedding
$\toric(\sigma)\hookrightarrow \CC^E$ from (\ref{subsec:affToric})
is induced by the map $E\colon N\to\ZZ^E$.

\subsection{Maps between $T$-varieties}
\label{subsec:mapsT}
A generalization of the functoriality to the setting of (\ref{subsec:pDiv})  also has to take care of the underlying variety $Y$.
\\[1ex]
Let $(Y',\pDiv',N')$ and $(Y,\pDiv,N)$ be p-divisors. Consider now any tuple $(F,\phi,\mathfrak f)$ consisting of a map $F:N\to N'$, a dominant morphism $\phi:Y'\to Y$, and an element $\mathfrak f=\sum v_i \otimes f_i \in N\otimes \CC(Y')^*$, satisfying $F_*\pDiv'\subseteq \varphi^*\pDiv + \div(\mathfrak f)$ (to be checked for all coefficients separately).
Here, 
$$	F_*\pDiv':= \sum_{Z'} F(\pDiv'_{Z'})\otimes Z' $$
is the \emph{push forward} of $\pDiv'$ by $F$, 
$$\varphi^*\pDiv := \sum_Z \pDiv_Z\otimes\varphi^*(Z)$$
is the \emph{pull back} of $\pDiv$ by $\phi$, and 
$$\div(\mathfrak f)=\sum_i (v_i )\otimes \div(f_i)$$
is the \emph{principal polyhedral divisor} associated to $\mathfrak f$. 
Such a tuple provides us with naturally defined morphisms
\begin{align*}
&\ttoric(F,\varphi,\mathfrak f)\colon\ttoric(\pDiv',N')\to \ttoric(\pDiv,N),
\qquad\mbox{and}\\
&\toric(F,\varphi,\mathfrak f)\colon\toric(\pDiv',N')\to \toric(\pDiv,N).
\end{align*}
This construction yields an equivalence of categories
between polyhedral cones and affine toric varieties in the toric case (where only the map $F$ plays a role). To obtain a similar result for the relation between p-divisors and $T$-varieties one first needs to \emph{define the category} of the former by turning both types of equivalences mentioned in (\ref{subsec:pDiv}) into isomorphisms. This can be done by the common technique of localization which is known from the construction of derived categories. In the category of $T$-varieties we restrict to morphisms $\psi:X'\rightarrow X$, with $T.\psi(X')\subset X$ being dense. We will call such morphisms \emph{orbit dominating}.

\begin{theorem}[\cite{tvar_1}, Corollary 8.14]
\label{thm:equivCat}
The functor $\toric$ induces an equivalence from the 
category of p-divisors to the category of 
normal affine varieties
with effective torus action and orbit dominating equivariant morphisms.
\end{theorem}

\subsection{Open embeddings}
\label{subsec:openEmb}
Let us fix a torus $T$. To glue affine $T$-varieties together one has to understand $T$-equivariant open embeddings. In the toric case,
an inclusion $F_\QQ:\sigma'\hookrightarrow\sigma$ of cones in $N_\QQ$ 
(i.e.\ $F = \id_N$ with $\sigma'\subseteq\sigma$) 
provides an open embedding $\toric(F)$
if and only if $\sigma'$ is a face of $\sigma$. 
Indeed, if $\sigma' = \face(\sigma,u) := \sigma\cap u^\bot$ 
is cut out by the supporting hyperplane $u\in\sigma\dual$,
then $\CC[{\sigma'}\dual \cap M] = \CC[\sigma\dual \cap M]_{\chi^u}$
equals the localization by $\chi^u$.
Thus, $\toric(\sigma')= [\chi^u\neq 0]\subseteq\toric(\sigma)$.

\begin{example}
\label{ex:embeddedTorus}
The origin $\{0\}$ is a common face of all cones $\sigma$. Hence, the torus
$T=\toric(0,N)\subseteq\toric(\sigma,N)$ appears as an open
subset in all toric varieties.
\end{example}

Similarly, if $\pDiv$ is a p-divisor on $Y$ and
$f\in \CO_{\Loc\pDiv}(\pDiv(u))$,
then $f(y)\chi^u\in\CC(Y)[M]\subseteq\CC(\toric(\pDiv))$ is an
$M$-homogeneous rational function on $\toric(\pDiv)$. The localization procedure of the toric setting now generalizes to
to the present setting:
the open subset $[f\chi^u\neq 0]\subseteq\toric(\pDiv)$ is again a 
$T$-variety. Its associated p-divisor $\pDiv_{f\chi^u}$
is still supported on $Y$ and its tailcone is
$\tail(\pDiv_{f\chi^u})=\face(\tail(\pDiv),u)$. Moreover,
\[\pDiv_{f\chi^u} = \face (\pDiv,u)+ \emptyset\otimes (\div f + \pDiv(u))\]
where the face operator is supposed to be applied to the polyhedral
coefficients only. More specifically,
$\face(\Delta,u) := \{a\in\Delta\kst \langle a,u\rangle = 
\min\langle \Delta,u\rangle\}$.
Note that $Z_u(f) := \div f + \pDiv(u)$ is an effective divisor which depends ``continuously'' upon $f$.
\\[1ex]
In contrast to the toric case, not all $T$-invariant open subsets are of the form $[f\chi^u\neq 0]$.
We will now present a precise characterization of open embeddings,
which is quite technical, but apparently does not easily simplify in the general case. For the much nicer case of complexity one, however, we refer to (\ref{subsec:openComp1}).
We need the following notation: for a polyhedral divisor $\pDiv = \sum_Z\pDiv_Z \otimes Z$ and a not necessarily closed point $y\in Y$ we define
$\pDiv_y := \sum_{y\in Z}\pDiv_Z\in\Pol^+_\Q(N,\sigma)$.
\\[1ex]
For any two  p-divisors $\pDiv,\pDiv'$ on a common $(Y,N)$, 
we say that $\pDiv'\subseteq\pDiv$ if
the polyhedral coefficients respectively 
satisfy $\pDiv'_Z\subseteq\pDiv_Z$. Two such p-divisors induce maps
$\til{\iota}:\ttoric(\pDiv')\to \ttoric(\pDiv)$ and
$\iota:\toric(\pDiv')\to \toric(\pDiv)$.
The former is an open embedding if and only if $\pDiv'\leq \pDiv$, i.e.\
if and only if all coefficients $\pDiv'_Z$ are faces of the corresponding
$\pDiv_Z$. In particular, this implies that
$\tail(\pDiv') \leq \tail(\pDiv)$. 

\begin{proposition}[\cite{tvar_2}, Proposition (3.4)]
\label{prop:openEmb}
The morphism $\iota$ is an open embedding $\iff$ $\pDiv'\leq \pDiv$ and,
additionally, for each $y\in\Loc(\pDiv')$
there are $u_y\in \tail(\pDiv)\dual\cap M$
and a divisor $Z_y\in |\pDiv(u_y)|$ such that $y \notin \supp Z_y$,
$\pDiv'_y=\face(\pDiv_y,u_y)$,
and
$\face(\pDiv'_z,u_y)=\face(\pDiv_z,u_y)$
for all $z \in \Loc(\pDiv)\setminus \supp(Z_y)$.
\end{proposition}

\begin{example}
Let $U \subset X = \Spec A$ be an open embedding of
normal affine varieties. By normality, $D = X \setminus U$ is
a Weil divisor. Then the open embedding $U \subset X$ of
trivial $T$-varieties corresponds to the divisors
$\pDiv' = \emptyset \otimes D \leq \pDiv = 0$ on $X$.
For $y \in \Loc(\pDiv') = U$, choose $f_y$ in the ideal of $D \subset X$
that doesn't vanish at $y$. Then $u_y := 0$ and
$Z_y := \div(f_y)$ satisfy the conditions of Proposition~\ref{prop:openEmb}.
Note that $U$ is not necessarily of the form $[f \neq 0]$.
\end{example}

\begin{example}
Consider the divisors
$$\pDiv' = \emptyset \otimes \{\infty\} \leq \pDiv = [1,\infty) \otimes \{\infty\}$$
on $\PP^1$ with tailcone $[0,\infty)$. Then $\ttoric(\pDiv) \to \toric(\pDiv)$ is the
blowup of $\AA^2$ at the origin, and the open subset
$\ttoric(\pDiv') \subset \ttoric(\pDiv)$ intersects the exceptional divisor, so
$\toric(\pDiv') \to \toric(\pDiv)$ is not injective.
\end{example}

\subsection{General toric and $T$-varieties}
\label{subsec:fansToric}
The characterization of open embeddings for affine toric varieties
via the face relation between polyhedral cones leads directly to the notion of a (polyhedral) fan which is a special instance of a polyhedral subdivision:

\begin{definition}
\label{def:Fan}
A finite set $\Sigma$ of polyhedral cones in $N_\QQ$ is called a \emph{fan} if the intersection of two cones from $\Sigma$ is a common face of both and, moreover, if $\Sigma$ contains all faces of its elements. (The normal fan $\normal(\Delta)$ from (\ref{subsec:dissToric}) is an example.)
\end{definition}

Given a fan $\Sigma$ we can glue the affine toric varieties
associated to its elements, namely
\[\toric(\Sigma) := \bigcup_{\sigma\in\Sigma} \toric(\sigma)
\hspace{0.8em}\mbox{with}\hspace{0.8em}
\toric(\sigma)\cap\toric(\tau)= \toric(\sigma\cap\tau),\]
and $\{\toric(\sigma)\kst \sigma\in\Sigma\mbox{ maximal}\}$
is an open, affine, $T$-invariant covering of $\toric(\Sigma)$.
This variety is automatically separated, and it is compact
if and only if $|\Sigma| := \bigcup_{\sigma\in\Sigma}\sigma = N_\QQ$. Moreover,
the functoriality globalizes; here the right notion for
a morphism $(N',\Sigma')\to (N,\Sigma)$ is a linear map $F:N'\to N$ such that for all $\sigma'\in\Sigma'$ there is a $\sigma\in\Sigma$ with
$F(\sigma')\subseteq\sigma$.
\\[1ex]
The above construction can be generalized in a straight-forward, if somewhat technical, manner to describe general $T$-varieties.  Let $\pDiv,\pDiv'$ be p-divisors on  $(Y,N)$. Their \emph{intersection} $\pDiv\cap\pDiv'$ is the polyhedral divisor with $Z$-coefficient $\pDiv_Z\cap\pDiv_Z'$ for any divisor $Z$ on $Y$. If $\pDiv'\subseteq \pDiv$, we say $\pDiv'$ is a face of $\pDiv$ if the induced map $\iota:\toric(\pDiv')\to\toric(\pDiv)$ is an open embedding, i.e. the conditions of Proposition \ref{prop:openEmb} hold.
\begin{definition}
	\label{def:divfan}
	A finite set $\pFan$ of p-divisors on $(Y,N)$ is called a \emph{divisorial fan} if the intersection of two p-divisors from $\pFan$ is a common face of both and, moreover, $\pFan$ is closed under taking intersections.
\end{definition}
Similar to the toric case, we can construct a scheme from a divisorial fan $\pFan$ via
\[\toric(\pFan) := \bigcup_{\pDiv\in\pFan} \toric(\pDiv)
\hspace{0.8em}\mbox{with}\hspace{0.8em}
\toric(\pDiv)\cap\toric(\pDiv')= \toric(\pDiv\cap\pDiv').\]
In general, the resulting scheme may not be separated, but there is a combinatorial criterion for checking this, see \cite[Theorem 7.5]{tvar_2}. On the other hand, by covering any $T$-variety $X$ with invariant affine open subsets, it is possible to construct a divisorial fan $\pFan$ with $X=\toric(\pFan)$, see \cite[Theorem 5.6]{tvar_2}. 

In general, the notion of divisorial fan can be cumbersome to work with, in large part due to the technical nature of Proposition \ref{prop:openEmb}.
However, a manageable substitute can be obtained in complexity 1, see (\ref{subsec:nonAffComp1}).

\section{How to construct p-divisors}
\label{sec:tDown}

\subsection{Toric varieties}
\label{subsec:toricT}
The affine toric setting presented in (\ref{subsec:affToric}) fits into the language of (\ref{subsec:pDiv}) when setting $Y = \spec \CC$. Since divisors on points are void, the only information carried by $\pDiv$ is its tailcone $\sigma$. Thus, $\ttoric(\pDiv) = \toric(\pDiv) = \toric(\sigma)$ in this case.

\subsection{Toric downgrades}
\label{subsec:tDown}
The toric case, however, can also provide us with more interesting examples. To this end we consider a toric variety $\toric(\delta,\tN)$ and fix a subtorus action $T\hookrightarrow\tT$. Now, what is the description of $\toric(\delta,\tN)$ as a $T$-variety?
Assuming that the embedding  $T\hookrightarrow\tT$
is induced from a surjection of the corresponding character groups
$p:\tM \surj M$, we denote the kernel by $M_Y$ and
obtain two mutually dual exact sequences \eqref{eqn:seq1} and \eqref{eqn:seq2}.
\\[1ex]
Setting $\sigma := N_\Q\cap\delta$, the map $p$ gives us a surjection
$\delta\dual\surj\sigma\dual$. On the dual side,
denote the surjection $\tN\surj N_Y$ by $q$. However,
instead of just considering the cone $q(\delta)$, we denote by
$\Sigma$ the coarsest fan refining the images of all faces of $\delta$ under the map $q$. Then, $|\Sigma| = q(\delta)$, and
the set $\Sigma(1)$ of rays in $\Sigma$,
i.e.\ its one-dimensional cones, contains the set $q(\delta(1))$.

\begin{align}
	\label{eqn:seq1}	\xymatrix@R=.3ex@!C=1.8cm{
0 \ar[r] & M_Y \ar[r] & 
\raisebox{0.8ex}{$\tM$} \ar[r]^-{p} & 
M \ar[r] &  0\\
& \nabla(u) := p^{-1}(u)\cap\delta\dual\, \ar@{^(->}[r] &
\delta\dual \ar@{->>}[r] & \sigma\dual\ni u
\\ {\ } \\}\\
\label{eqn:seq2}	\xymatrix@R=.3ex@!C=1.8cm{
0 & N_Y \ar[l] & 
\raisebox{0.8ex}{$\tN$} \ar[l]_-{q} & 
N \ar[l] & \ar[l] 0
\\
& a\in \Sigma(1)\subseteq |\Sigma| &
\delta \ar@{->>}[l]&
\;q^{-1}(a)\cap\delta=:\pDiv_a \ar@{_(->}[l]
}\end{align}
Next, we perform the following two mutually dual constructions.
For $u\in\sigma\dual$ and $a\in|\Sigma|$, let $\nabla(u)\subseteq M_{Y,\QQ}$ and $\pDiv_a\subseteq N_\QQ$ be as indicated in the diagram above. Note that we have to make use of some section $s: M\hookrightarrow \tM$ of $p$, i.e.\ a splitting of the two sequences, to shift both polyhedra from $p^{-1}(u)$ and $q^{-1}(a)$
into the respective fibers over $0$. Both polyhedra are
linked to each other by the equality
\[\min\langle a,\nabla(u)\rangle = \min \langle \pDiv_a, u\rangle\]
which is an easy consequence (via $a=q(x)$ and $u=p(y)$) of the
following lemma appearing in the proof of
\cite[Proposition 8.5]{tvar_0}.

\begin{lemma}
\label{lem:Gale}
$x\in\delta$, $y\in\delta\dual$ $\then$
$\min\langle q^{-1}q(x)\cap\delta, y\rangle +
\min\langle x, p^{-1}p(y)\cap\delta\dual\rangle = 
\langle x,y\rangle$.%
\end{lemma}

Thus, defining $Y := \toric(\Sigma)$, we can refer to the upcoming discussion in (\ref{subsec:toricOrbits}) for the construction of 
$T$-invariant Weil divisors $\ko{\orb}(a)\subseteq\toric(\Sigma)$ 
for $a\in\Sigma(1)$ to define the p-divisor
$\pDiv^\delta :=\sum_{a\in\Sigma(1)} \pDiv_a\otimes \ko{\orb}(a)$
on $(Y,N)$. 
Although we have indexed the coeffixients of $\pDiv$ by a ray $a$ instead of a prime divisor, this is harmless since we may always identify $a$ with the divisor $\ko{\orb}(a)\subseteq\toric(\Sigma)$.
It follows from the characterization of toric nef Cartier divisors and the above equality involving $\nabla(u)$ and $\pDiv_a$ that
$\toric(\delta)=\toric(\pDiv^\delta)$.

\begin{remark}
\label{rmk:refinemantDowngrade}
Note that $q$ does not yield a map $(\tN,\delta)\to (N_Y,\Sigma)$ in general. To achieve this, one has to further subdivide $\delta$ into the fan $\til{\delta}$  consisting of the preimages of the cones of $\Sigma$.
Moreover, observe that $\toric(\til{\delta})=\ttoric(\pDiv)$, which illustrates both the contraction $\ttoric(\pDiv)\to \toric(\pDiv)$
as well as the morphism $\ttoric(\pDiv)\to Y$.
\end{remark}

\subsection{General affine $T$-varieties}
\label{subsec:generalT}
Let $X\subseteq\CC^n$ be an equivariantly embedded,
normal, affine $T$-variety which intersects the torus non-trivially, and denote by $f_1,\ldots,f_m$ $M$-homogeneous generating
equations for $X$. If $T$ acts diagonally and effectively on
$\CC^n$, then each variable $x_i$ has a degree in $M$ which gives rise to a surjection $p:\tM := \ZZ^n\surj M$, $e_\nu\mapsto \deg x_\nu$
as in the previous section. Setting $\delta := \QQ^n_{\geq 0}$ and $u_i:=\deg f_i\in M$, we use the section 
$s: M\hookrightarrow \tM$ as mentioned in (\ref{subsec:tDown}) to
shift the equations $f_i$ to the Laurent polynomials
$f'_i := \chi^{-s(u_i)} f_i\in\CC[M_Y]$. They define
a closed subvariety of the torus $T_Y := \Spec\CC[M_Y]$, and we use this
to modify the definition of $Y$ from (\ref{subsec:tDown}) to the normalization of
\[\ko{V(f'_1,\ldots,f'_m)}\subseteq\toric(\Sigma).\]
If $\ko{\orb}(a)$ are Cartier divisors,
then we can consider their pull backs $Z_a := \ko{\orb}(a)\cap Y$. Thus, we arrive at the description $X = \toric(\pDiv)$  as a $T$-variety, where $\pDiv=\sum_{a\in\Sigma(1)} \pDiv_a\otimes Z_a$.

\begin{remark}
\label{rmk:toricEmbedding}
It is sometimes easier to consider $T$-equivariant embeddings
$X\hookrightarrow\toric(\delta)$ instead of
$X\hookrightarrow\CC^n$. In this case, everything also works out 
as discussed in the paragraph above.
\end{remark}

\subsection{Slices of divisorial fans}
\label{subsec:nonAffineT}
At the end of (\ref{subsec:fansToric}) we introduced the notion of divisorial fans, which encode invariant open affine coverings of general $T$-varieties.
The face conditions guarantees that for any divisorial fan $\pFan=\{\pDiv^i\}$ and any prime divisor $Z\subset Y$, the set
$$
\pFan_Z=\{\pDiv_Z^i\}
$$
is a polyhedral subdivision in $N_\QQ$ called a \emph{slice}. Indeed,
since open embeddings between affine $T$-varieties 
$\toric(\pDiv')\hookrightarrow\toric(\pDiv)$ imply that
$\pDiv'\leq\pDiv$, the polyhedra $\pDiv_Z^i\subseteq N_\QQ$ are supposed
to be glued along the common faces $\pDiv'_Z\leq\pDiv_Z$. 

These slices can be viewed another way as well.
Similarly to (\ref{subsec:generalT}), one can try to embed a general
$T$-variety $X$ into a toric variety $\toric(\cF)$ with some fan
$\cF$ of polyhedral cones $\delta$ in $\tN_\QQ$. Now, the method
of (\ref{subsec:tDown}) can be copied to yield an $\pFan$ on
$\toric(\Sigma)$ describing $\toric(\cF)$ as a $T$-variety 
denoted by $\toric(\pFan)$ and,
afterwards, one follows (\ref{subsec:generalT}) to find the right
$Y\subseteq \toric(\Sigma)$ to which $\pFan$ should be restricted in order to describe $X$. In this setting, the slices $\pFan_Z$ appear as intersections of the fan $\cF$ with certain affine subspaces, explaining the choice of terminology.
\\[1ex]
It is tempting to think that the slices $\pFan_Z$ capture the entire information of $\pFan$.
However, the important point is that one also must keep track of 
which of the polyhedral cells inside the different $\pDiv_Z$ belong
to a common p-divisor. This is related to the question of how much is contracting
along $\ttoric(\pFan)\to\toric(\pFan)$.
Thus, one is supposed to introduce labels
for the cells in $\pFan_Z$. Again, we refer to
(\ref{subsec:nonAffComp1}) for a nice way of overcoming this problem in complexity
one.

\section{Globalization in complexity 1}
\label{sec:globComp1}

\subsection{The degree polyhedron}
\label{subsec:degPolytope}
Let $\pDiv=\sum_{P\in Y}\pDiv_P\otimes P$ denote a p-divisor of complexity one on the curve $Y$ with tailcone $\sigma$. Since prime divisors $Z$ coincide with closed points $P \in Y$ we will from now on replace the letter $Z$ by $P$ in this setting. As in the definition of
$\pDiv_P$ in (\ref{subsec:openEmb}) we can now introduce the \emph{degree}
\[\deg \pDiv := \sum_{P\in Y} \pDiv_P\in \Pol^+_\Q(N,\sigma).\]
Note that it satisfies the equation $\min\langle \deg\pDiv,u\rangle=\deg\pDiv(u)$. In contrast, considering the sum of the polyhedral coefficients of a p-divisor in higher complexity does not make much sense. Nonetheless, one may define the degree
of a p-divisor on a polarized $Y$ via the following
construction: let $C\subseteq Y$ be a curve or, more generally, a
numerical class of curves in $Y$. Then there is a well defined
generalized polyhedron $(C\cdot\pDiv)\in \Pol_\Q(N,\sigma)$ where the term ``generalized'' refers to an element in the
Grothendieck group associated to $\Pol^+_\Q(N,\sigma)$. Hence, it is representable as the formal difference of two polyhedra.

\subsection{Using $\deg\pDiv$ to characterize open embeddings}
\label{subsec:openComp1}
For a p-divisor $\pDiv$ on a curve $Y$ it is as obvious as important to observe that
\[\deg\pDiv \neq\emptyset\;\iff\;\Loc\pDiv=Y.\]
Moreover, if $\pDiv$ is a polyhedral divisor on $Y$ then it is not hard to see that $\pDiv$ is a p-divisor $\iff$ $\deg\pDiv\subsetneq\tail(\pDiv)$
and $\pDiv( w)$ has a principal multiple for all $w\in (\tail(\pDiv))\dual$
with $w^\bot\cap(\deg\pDiv)\neq 0$. Note that the latter condition
is automatically fulfilled if $Y=\PP^1$.
\\[1ex]
Nonetheless, degree polyhedra play their most important roles in the characterization of open embeddings since they transform the technical condition from (\ref{subsec:openEmb}) into a very natural and geometric
one.

\begin{theorem}[\cite{NathHend}, Lemma 1.4]
\label{thm:openEmbComp1}
Let $Y$ be a curve and 
$\pDiv',\pDiv$
be a polyhedral and a
p-divisor, respectively, such that $\pDiv'\leq\pDiv$. Then, $\pDiv'$ is a p-divisor and
$\toric(\pDiv')\hookrightarrow \toric(\pDiv)$ is a
$T$-equivariant open embedding
$\iff$ 
$\deg\pDiv'=\deg\pDiv\cap\tail(\pDiv')$.
\end{theorem}

In particular, if $\pFan = \{\pDiv^i\}$ is a divisorial fan as in (\ref{subsec:fansToric})
then the set $\{\tail(\pDiv^i)\,|\, \pDiv^i \in \pFan\}$ forms the so-called \emph{tailfan} $\tail(\pFan)$, and the subsets $\deg\pDiv^i\subsetneq\tail(\pDiv^i)$ glue together to a
proper subset $\deg\pFan\subsetneq|\tail(\pFan)|$.
The single degree polyhedra can be recovered as
$\deg\pDiv^i=\deg\pFan\cap\tail(\pDiv^i)$.

\subsection{Non-affine $T$-varieties in complexity 1}
\label{subsec:nonAffComp1}
Another important feature of curves as base spaces $Y$ is that
$\Loc\pDiv\subseteq Y$ is either projective (if $\deg\pDiv\neq\emptyset$) or
affine (if $\deg\pDiv=\emptyset$). In the latter case,
$\ttoric(\pDiv)=\toric(\pDiv)$, and gluing those cells does not cause
problems. For the other case, an easy observation is

\begin{lemma}
\label{lem:compLoc}
Let $\pDiv$ be a p-divisor on a curve $Y$ with
$\deg\pDiv\neq\emptyset$.
Then, for all $P\in Y$, we have that $\dim \pDiv_P = \dim(\tail(\pDiv))$.
\end{lemma}

Indeed, we have the following chain of inequalities
\[\dim(\tail(\pDiv)) \leq \dim \pDiv_P \leq \dim (\deg\pDiv) \leq \dim(\tail(\pDiv))\]
each of which results from a combination of translations and inclusions.
 
Consider a divisorial fan  $\pFan=\{\pDiv^i\}$ on a curve. For the following, we will make the reasonable assumption that each (lower-dimensional) cone of $\Sigma=\tail(\pFan)$ is the face of a full-dimensional cone of $\tail(\pFan)$. 
For example, this is satisfied if the support of $\tail(\pFan)$ is a polyhedral cone or the entire space $N_\QQ$.
Now we have seen in (\ref{subsec:nonAffineT}) and (\ref{subsec:openComp1}) that a divisorial fan $\pFan=\{\pDiv^i\}$ on a curve determines a pair $(\sum_P \pFan_P \otimes P,\sdeg)$. This motivates the following definition:
\begin{definition}
Consider a pair $(\sum_P \pFan_P \otimes P,\sdeg)$
where $\pFan_P$ are all polyhedral subdivisions with some common tailfan $\Sigma$ and $\deg\subset |\Sigma|$. This pair is called an \emph{f-divisor} 
if for any full-dimensional $\sigma\in\tail(\pFan)$ with $\sdeg\cap\sigma\neq \emptyset$, $\pDiv^\sigma:=\sum \Delta_Z^\sigma$ is a p-divisor and  $\sdeg\cap\sigma=\deg \pDiv^\sigma$. Here $\Delta_Z^\sigma$ denotes the unique polyhedron in $\pFan_Z$ with $\tail(\Delta_Z^\sigma)=\sigma$.
\end{definition}

\begin{proposition}[{\cite[Proposition 1.6]{candiv,NathHend}}]
\label{prop:degHelps}
Consider an f-divisor $(\sum_P \pFan_P \otimes P,\sdeg)$.
Under the above assumptions, there is a divisorial fan $\pFan$
with slices $\pFan_P$ and degree $\sdeg$.  For any other such divisorial fan $\pFan'$, we have $\toric(\pFan)=\toric(\pFan')$. 
\end{proposition}

Indeed, we can take the divisorial fan $\pFan$ to consist of p-divisors $\pDiv^\sigma$ for full-dim\-en\-sional $\sigma\in\tail(\pFan)$ with non-empty degree, for $\sigma$ with empty degree any finite set of p-divisors providing an open affine covering of $\ttoric(\pDiv^\sigma)$, and intersections thereof. 		

\begin{remark}
\label{rmk:degreeVSmarkings}
Since the subset $\sdeg\subseteq|\tail(\pFan)|$ is completely determined by the equality $\sdeg \cap \pDiv = \sum_P \pDiv_P$,
it is just necessary to know if $\deg\pDiv$ is empty or not.
This leads to the notion of \emph{markings} (of tailcones with
$\deg\pDiv\neq\emptyset$) in \cite{candiv,NathHend}.
\end{remark}

If $\pFan$ is an f-divisor, i.e. a pair $(\sum_P \pFan_P \otimes P,\sdeg)$, we write $\toric(\pFan)$ for the $T$-variety it determines as in proposition \ref{prop:degHelps}. Note that we use the same symbol for divisorial fans and f-divisors, since they both determine general $T$-varieties (although the former also contains the information of a specific affine cover).
As has already been observed in the affine case of (\ref{subsec:pDiv}),
different f-divisors $\pFan,\pFan'$ might yield the same (or equivariantly isomorphic)
$T$-varieties $\toric(\pFan)=\toric(\pFan')$. This is the case if there are isomorphisms
$\varphi:Y \rightarrow Y'$ and $F:N \rightarrow N'$ and an element $\sum_i v_i \otimes f_i \in N \otimes \CC(Y')$ such that $\fan'_{\varphi(P)}+\sum_i \ord_P(f_i)\cdot v_i = F(\fan_P)$  holds for every $P \in Y$ and we have $\sdeg'=F(\sdeg)$.

\begin{example}
\label{ex:cotangHirzebruch}
We consider the projectivized cotangent bundle on the first Hirzebruch surface $\cF_1$. The non-trivial slices of its divisorial fan $\pFan$ over $\PP^1$ are illustrated in Figure \ref{fig:cotangFanHirzebruch}. The associated tailfan $\Sigma$ and degree $\deg \pFan$ are given in Figure \ref{fig:cotangHirzebruchPlus}.

\begin{figure}[t]
\centering
\subfloat[$\pFan_0$]{\cotangfanHirzebruchNull}\hspace*{5ex}
\subfloat[$\pFan_1$]{\cotangfanHirzebruchOne}\hspace*{5ex}
\subfloat[$\pFan_\infty$]{\cotangfanHirzebruchInfty}
\caption{Non-trivial slices of $\pFan(\PP(\Omega_{\cF_1}))$, cf.\ Example \ref{ex:cotangHirzebruch}.}
\label{fig:cotangFanHirzebruch}
\end{figure}

\begin{figure}[h]
\centering
\subfloat[$\Sigma = \tail(\pFan$)]{\hspace*{1ex}\tailFanCotangHirzebruch\hspace*{1ex}}\hspace*{8ex}
\subfloat[$\deg \pFan$]{\degreeCotangHirzebruch}
\caption{Tailfan and degree of $\pFan\big(\PP(\Omega_{\cF_1})\big)$, 
cf.\ Example \ref{ex:cotangHirzebruch}.}
\label{fig:cotangHirzebruchPlus}
\end{figure}

\end{example}

\section{The $T$-orbit decomposition}
\label{sec:orbitDecomp}

\subsection{$T$-orbits in toric varieties}
\label{subsec:toricOrbits}
Let $\sigma\subseteq N_\QQ$ be a not necessarily full-di\-men\-sio\-nal,
pointed polyhedral cone. Then, denoting $N_\sigma := N/\spann_\ZZ(\sigma\cap N)$, the affine toric variety $\toric(\sigma,N)$ contains a unique closed $T$-orbit $\orb(\sigma) := \Spec\CC[\sigma^\bot\cap M]=
N_\sigma\otimes_\ZZ\CC^*=\toric(0,N_\sigma)$
with $\dim\orb(\sigma)=\rank N-\dim\sigma$.
If $\tau\leq\sigma$ is a face, then we obtain the diagram
$$
\xymatrix@R=3.3ex@C=1.0em{
\toric(0,N_\tau)\; \ar@{=}[r] \ar@{^(->}[d]^-{\open} &
\;\Spec\CC[\tau^\bot\cap M]\; \ar@{=}[r] \ar@{^(->}[d]^-{\open} &
\;\orb(\tau)\; \ar@{^(->}[rrr]^-{\closed} \ar@{^(->}[d]^-{\open} &&&
\;\toric(\tau) \ar@{^(->}[d]^-{\open}
\\
\toric(\ko{\sigma},N_\tau)\; \ar@{=}[r] &
\;\Spec\CC[\sigma\dual\cap\tau^\bot\cap M]\;  \ar@{=}[r] &
\;\ko{\orb}(\tau)\; \ar@{^(->}[rrr]^-{\closed} &&&
\;\toric(\sigma)
}
$$
with $\ko{\sigma} := \im(\sigma\to N_{\tau,\QQ})$.
This construction does also work in the global setting:
if $\Sigma$ is a fan, then there is a stratification
$\toric(\Sigma)=\bigsqcup_{\sigma\in\Sigma}\orb(\sigma)$
with $\orb(\sigma)\subseteq\ko{\orb}(\tau)$ $\iff$
$\tau\leq\sigma$. In particular,
we obtain the $T$-invariant Weil divisors of $\toric(\Sigma)$ as
$\ko{\orb}(a)$ for $a\in\Sigma(1)$.
Moreover, the closure of the $T$-orbits are again 
toric varieties with $\ko{\orb}(\tau)=\toric(\ko{\Sigma},N_\tau)$
and $\ko{\Sigma} := \{\ko{\sigma}\kst \tau\leq\sigma\in\Sigma\}$
being a fan in $N_\tau$.
\\[1ex]
Note that we can, alternatively, generalize the $T$-orbit decomposition of
the affine $\toric(\sigma)$ to the situation of 
toric bouquets $\toric(\Delta)$ from (\ref{subsec:dissToric}).
Its $T$-orbits correspond to the faces of $\Delta$. 
In particular, as already mentioned in
(\ref{subsec:dissToric}), the irreducible components of $\toric(\Delta)$ are
counted by the vertices $v\in\kV(\Delta)$.

\subsection{$T$-orbits for p-divisors}
\label{subsec:pDivOrbits}
Let $\pDiv$ be an \emph{integral} p-divisor on a projective variety $Y$,
i.e.\ assume that the coefficients $\pDiv_Z$ of the Cartier divisors $Z$ are
lattice polyhedra. Equivalently, we ask for $\pDiv(u)$ to be an
ordinary Cartier divisor for all $u\in\tail(\pDiv)\dual\cap M$.
Then, $\pi:\ttoric(\pDiv)\to Y$ is a flat map fitting into the following
diagram:
$$
\xymatrix@R=4.0ex@C=2.5em{
& \;\ttoric(\pDiv)\; \ar[rr]^-{p} \ar[d]^-{\pi} &&
\;\toric(\pDiv)\; \ar[d]
\\
Y & 
\;\Loc\pDiv\; \ar@{_(->}[l] \ar[rr]&& 
\;\Spec\kG(\Loc\pDiv,\CO_{\Loc})
}
$$
For $y\in\Loc\pDiv\subseteq Y$, the fiber $\pi^{-1}(y)$
equals the toric bouquet $\toric(\pDiv_y)$
with
$\pDiv_y := \sum_{y\in Z}\pDiv_Z\in\Pol^+_\Q(N,\sigma)$
as it was already used in (\ref{subsec:openEmb}).
Note that the neutral element $\tail(\pDiv)$ serves as the
sum of the empty set of summands.
Hence, there is a decomposition of $\ttoric(\pDiv)$
into $T$-orbits $\orb(y,F)$ where $y\in\Loc\pDiv$ 
is a closed point
and $F\leq\pDiv_y$ is a non-empty face.
Their dimension is $\dim\, \orb(y,F)=\codim_N F$.
While every $T$-orbit $\orb(y,F)\subseteq \ttoric(\pDiv)$
maps, via the contraction $p$, isomorphically into $\ttoric(\pDiv)$,
some of them might be identified with each other:

\begin{proposition}[\cite{tvar_1}, Theorem 10.1]
\label{prop:orbitsX}
Let $(y,F)$ and $(y',F')$ correspond to the $T$-orbits
$\orb(y,F)$ and $\orb(y',F')$. Then they become identified along $p$
$\iff$ $\normal(F,\Delta)=\normal(F',\Delta')=:\lambda$
{\rm (cf.\ (\ref{subsec:dissToric}))} and 
$\Phi_u(y)=\Phi_u(y')$ for $u\in\innt\lambda$, 
where $\Phi_u$ denotes the stabilized
morphism associated to the semi-ample divisor $\pDiv(u)$.
\end{proposition}

If $\pDiv$ is a general, i.e.\ not integral p-divisor, then
a similar result holds true. However, one has to deal with the
sublattices $S_y := \{u\in M\kst \pDiv(u) \mbox{ is principal in } y\}$
as  was done in \cite[\S7]{tvar_1}.

\subsection{Invariant Weil divisors on $T$-varieties}
\label{subsec:WDivX}
In contrast to (\ref{subsec:pDivOrbits}), we are now going to use the notation $\orb(y,F)$ also for non-closed points $y\in Y$, e.g.\ for
generic points $\eta$ of subvarieties of $Y$. Note that the dimension of their closures is then given as $\dim\, \ko{\orb}(y,F)=\dim\ko{y}+\codim_N F$.
We conclude from (\ref{subsec:pDivOrbits}) that there are two kinds of $T$-invariant prime divisors in $\ttoric(\pDiv)$: on the one hand, we have the so-called \emph{vertical} divisors
\[D_{(Z,v)} := \ko{\orb}(\eta(Z),v)\]
for prime divisors $Z\subseteq Y$ and vertices $v\in\pDiv_Z$. On the other hand, there are the so-called \emph{horizontal} divisors
\[D_\varrho := \ko{\orb}(\eta(Y),\varrho)\]
with $\varrho$ being a ray of the cone $\tail(\pDiv)=\pDiv_{\eta(Y)}$.
The $T$-invariant prime divisors on $\toric(\pDiv)$
correspond exactly to those on $\ttoric(\pDiv)$ which are not
contracted via $p:\ttoric(\pDiv)\to \toric(\pDiv)$.%
\\[1ex]
In the special case of complexity one, i.e.\ if $Y$ is a curve, the vertical divisors $D_{(P,v)}$ (with $P\in Y(\CC)$ and
$v\in \Delta_P$) survive completely in $\toric(\pDiv)$.
In contrast, some of the horizontal divisors $D_\varrho$ may be contracted.

\begin{definition}
\label{def:extremalRays}
Let $\pFan$ be an f-divisor on $Y$. The set of rays $\varrho \in (\tail(\pFan))(1)$ for which the prime divisor $D_\varrho$ is not contracted via the map $p: \wt{\TV}(\pFan) \to \TV(\pFan)$ is denoted by $\cR := \cR(\pFan)$.
\end{definition}

\begin{remark}
\label{rmk:extremalRaysDegree}
The set $\cR(\pFan)$ can  be determined from $(\pFan, \deg \pFan)$. Indeed, $\cR(\pFan) = \{\rho \in (\tail(\pFan))(1) \,|\, \rho \cap \deg \pFan = \emptyset\}$, cf.\ for example \cite[Proposition 2.1]{phdLars}.
\end{remark}

\section{Cartier divisors}
\label{sec:divisors}

\subsection{Invariant principal divisors on toric varieties}
\label{subsec:toricPrinc}
Let $\Sigma \subset N_\QQ$ be a polyhedral fan. The monomials $\chi^u$ in $\CC[M]\subseteq\CC(\toric(\Sigma))$ are the $T$-invariant, rational
functions on the toric variety $\toric(\Sigma)$. Hence,
the principal divisor associated to such a function must be a linear combination of the
$T$-invariant Weil divisors $\ko{\orb}(a)$
from (\ref{subsec:toricOrbits}). The coefficients can be obtained by restricting the situation to the open subsets
$\CC\times(\CC^*)^{k-1}\cong\toric(a,N)\subseteq\toric(\Sigma)$ 
for each $a\in\Sigma(1)$ separately. This gives us
\[\div(\chi^u)=\sum_{a\in\Sigma(1)}\langle a,u\rangle\cdot \ko{\orb}(a),\]
i.e.\ 
the orbits of codimension one satisfy the relation
$\sum_{a\in\Sigma(1)} a \otimes \ko{\orb}(a)\sim0$.
Note that every divisor is linearly equivalent to an invariant one. Assuming that $\Sigma$ contains a cone whose dimension equals the rank of $N$, this leads to the famous exact sequence
$$
0 \to M \stackrel{\div}{\longrightarrow} 
\big[\Div_T\toric(\Sigma)=\ZZ^{\Sigma(1)}\big] \to
\Cl(\toric(\Sigma))\to 0.
$$

\subsection{Invariant principal divisors on $T$-varieties}
\label{subsec:princT}

Let $\pDiv$ be a p-divisor on $Y$. 
Similarly to (\ref{subsec:openEmb}),
the place of the monomials $\chi^u$ in
(\ref{subsec:toricPrinc}) is now taken by the functions
$f(y)\chi^u\in\CC(Y)[M]\subseteq\CC(\toric(\pDiv))$.
Identifying a ray $\varrho\in\tail(\pDiv)(1)$ with
its primitive lattice vector and denoting by $\mu(v)$ the smallest integer $k \geq 1$ for $v\in N_\Q$ such that $k\cdot v$ is a lattice point, one has the following characterization
of $T$-invariant principal divisors:

\begin{theorem}[\cite{tidiv}, Proposition 3.14]
\label{thm:principDiv}
The principal divisor which is associated to $f(y)\,\chi^u\in \CC(Y)[M]$ on $\wt{TV}(\pDiv)$ or $\TV(\pDiv)$ is given by
\[\div\big(f(y)\,\chi^u\big)=
\sum_{\varrho
} \langle\varrho,u\rangle D_\varrho +
\sum_{(Z,v)}\mu(v)\big(\langle v,u\rangle + \ord_Z f\big) \cdot D_{(Z,v)}\]
where, if focused on $\toric(\pDiv)$, one is supposed
to omit all prime divisors being contracted via 
$p:\ttoric(\pDiv)\to\toric(\pDiv)$.
\end{theorem}

A proof which (formally) locally inverts the toric downgrade construction from (\ref{subsec:tDown}) is given in \cite[(2.4)]{coxComp1}. Hence, the claim becomes a direct consequence of (\ref{subsec:toricPrinc}).

\subsection{Invariant Cartier divisors on toric varieties}
\label{subsec:CDtoric}
We consider an invariant Cartier divisor $D$ on $\TV(\Sigma)$. Restricting to an open affine chart $\TV(\sigma) \subset \TV(\Sigma)$ we have that $D|_{\TV(\sigma)} = \divisor (\chi^{-u_\sigma})$ with $u_\sigma \in M_{\sigma} := M/(M \cap \sigma^\perp)$. The compatibility condition on intersections gives us that $\phi_{\tau\sigma}(u_\sigma) = u_\tau$ where $\tau$ denotes a face of $\sigma$ and $\phi_{\tau\sigma}: M_\sigma \to M_\tau$ is the natural restriction. Hence, we can regard $\{u_\sigma\ |\ \sigma \in \Sigma\}$ as a piecewise continuous linear function $\psi_D: |\Sigma| \to \QQ$ which is locally defined as $\psi_D|_{|\sigma|} = u_\sigma$. Vice versa, any continuous function $\psi: |\Sigma| \to \QQ$ which is linear on the cones $\sigma \in \Sigma$ and integral on $|\Sigma| \cap N$ can be evaluated along the primitive generators of the rays $\rho \in \Sigma(1)$ and thus defines an invariant Cartier divisor $D_\psi$ on $\TV(\Sigma)$.

A special instance of the above correspondence shows that every lattice polyhedron $\nabla\subseteq M_\QQ$ such that $\Sigma$ is a subdivision of the inner normal fan $\normal(\nabla)$, cf.\ (\ref{subsec:dissToric}) and (\ref{subsec:projTor}),
gives rise to a Cartier divisor $\div(\nabla)$ whose sheaf $\cO_{\TV(\Sigma)}(\div(\nabla))$ is globally generated by the monomials corresponding to $\nabla\cap M$. Locally,
on the chart $\toric(\sigma)\subseteq\toric(\Sigma)$ for a
maximal cone $\sigma\in\Sigma$,
one has 
$\cO_{\toric(\Sigma)}(\div(\nabla)) = \chi^u\cdot\cO_{\toric(\Sigma)}$ with $u\in\nabla\cap M$ being the vertex corresponding to 
the unique cone $\delta\in\normal(\nabla)$ containing $\sigma$.
The associated Weil divisor is then given by
$$
\div(\nabla)= -\!\sum_{a\in\Sigma(1)}
\min\langle a,\nabla\rangle\cdot \ko{\orb}(a).
$$

\subsection{Invariant Cartier divisors in complexity 1}
\label{subsec:CDivX}
Let $\pFan$ be an f-divisor on the curve $Y$. A \emph{divisorial support function} on $\pFan$ is a collection $(h_P)_{P \in Y}$ of continuous piecewise affine linear functions $h_P: |\pFan_P| \to \QQ$ such that 
\begin{enumerate}
\item $h_P$ has integral slope and integral translation on every polyhedron in the polyhedral complex $\pFan_P \subset N_\QQ$.
\item all $h_P$ have the same linear part $=: \ul{h}$.
\item the set of points $P \in Y$ for which $h_P$ differs from $\ul{h}$ is finite. 
\end{enumerate}

Observe that we may restrict $h_P$ to a subcomplex of $\pFan_P$. Similarly, we may restrict a divisorial support function $h$ to a p-divisor $\cD \in \pFan$ which we will denote by $h|_\cD$.

In addition, we can associate a divisorial support function $\suppFunc(D)$ to any Cartier divisor $D \in \cadiv Y$ by setting
$\suppFunc(D)_P \equiv \coeff_P(D)$. Moreover, we can consider any element $u \in M$ as a divisorial support function by setting $\suppFunc(u)_P \equiv u$.

\begin{definition}
\label{def:CaSF}
A divisorial support function $h$ on $\pFan$ is called \emph{principal} if $h =  \suppFunc(u) + \suppFunc(D)$ for some $u \in M$ and some principal divisor $D$ on $Y$. It is called \emph{Cartier} if its restriction $h|_\cD$ is principal for every $\cD \in \pFan$ with $\Loc \cD = Y$. The set of Cartier support functions is a free abelian group which we denote by $\CaSF(\pFan)$.
\end{definition}

The following result not only plays an important role for the proof of Proposition \ref{prop:cadiv} but it is also interesting by itself, since it generalizes the fact that the Picard group of an affine toric variety is trivial.

\begin{lemma}[\cite{tidiv}, Proposition 3.1]
Assume that $\pDiv$ has complete locus. Then every invariant Cartier divisor on $\TV(\pDiv)$ is already principal.
\end{lemma}

Let $\TV(\pFan)$ be a complexity-one $T$-variety and denote by $\tcadiv(\TV(\pFan))$ the free abelian group of $T$-invariant Cartier divisors on $\TV(\pFan)$. Choosing affine invariant charts on which the Cartier divisor in question trivializes, one can use its local representations $f(y) \chi^u \in \CC(Y)[M]$ to define piecewise affine functions $h_P := \ord_P f + \langle \cdot, u\rangle$. It is then not hard to see that these functions glue and yield an element $h \in \CaSF(\fan)$ which leads to the following correspondence:

\begin{proposition}[\cite{tidiv}, Proposition 3.10]
\label{prop:cadiv}
$\tcadiv(\TV(\pFan)) \cong \CaSF(\pFan)$ as free abelian groups.
\end{proposition}

We will sometimes identify an element $h \in \CaSF(\pFan)$ with its induced $T$-invariant Cartier divisor $D_h$ via this correspondence. As a consequence of Theorem \ref{thm:principDiv} we obtain that the Weil divisor associated to a given Cartier divisor $h = (h_P)_P$ on $\TV(\pFan)$ is equal to
\[-\sum_{\varrho} \ul{h} (\varrho) D_\varrho - \sum_{(P,v)} \mu(v) h_P(v) D_{(P,v)}.\]

\begin{remark}
\label{rmk:downgradeCaSF}
Let $D_\psi$ denote a Cartier divisor on the toric variety $\TV(\Sigma)$ which is associated to the integral piecewise linear function $\psi: |\Sigma| \to \QQ$ (see (\ref{subsec:CDtoric})). Performing a downgrade to complexity one (cf.\ (\ref{subsec:tDown})), we denote the induced f-divisor by $\pFan$. The associated Cartier support function $h^\psi$ or, more specifically, its only non-trivial parts $h^\psi_0$ and $h^\psi_\infty$ then result from the restriction of $\psi$ to the slices $\pFan_0$ and $\pFan_\infty$, respectively.
\end{remark}

\subsection{The divisor class group}
\label{subsec:classT}
Using the results from (\ref{subsec:WDivX}) and (\ref{subsec:princT}) one obtains that (cf.\ \cite[Section 3]{tidiv})
\[\cl(\TV(\pFan)) = \frac {\bigoplus_\varrho \ZZ \cdot D_\varrho \oplus \bigoplus_{D_{(Z,v)}} \ZZ \cdot D_{(Z,v)}}{\big\langle \sum u(\varrho) D_\varrho + \sum_{D_{(Z,v)}} \mu(v) (\langle v, u \rangle + a_Z) D_{(Z,v)} \big\rangle},\] 
where $u$ runs over all elements of $M$ and $\sum_Z a_Z Z$ over all principal divisors on $Y$. However, assuming that $\TV(\pFan)$ is a complete rational $T$-variety of complexity one, it is possible to find a representation of $\cl(\TV(\pFan))$ which is analogous to the exact sequence from (\ref{subsec:toricPrinc}) as follows.

Let $\fan = \sum_{P \in \PP^1} \fan_P \otimes [P]$ be a complete f-divisor on $\PP^1$. In particular, $\deg \fan \subsetneq |\tail(\fan)| = N_\QQ$. Choose a minimal finite subset $\cP \subseteq \PP^1$ containing at least two points and such that $\fan_P$ is trivial (i.e.\ $\fan_P = \tail(\fan)$) for $P \in \PP^1 \setminus \cP$. If $v$ is a vertex of the slice $\fan_P$, we denote  $P(v)=P$. Let us define $\cV \;:=\; \{v \in \fan_{P}(0) \kst P \in \cP\}$ together with the maps
\[Q:\; \ZZ^{\cV \cup \cR} \to \ZZ^\cP/\ZZ \quad \textnormal{with} \quad
e(v) \mapsto \mu(v)\,\ovl{e(p(v))} \quad \textnormal{and} \quad
e(\varrho) \mapsto 0\,,\]
and
\vspace{-1ex}
\[\phi:\; \ZZ^{\cV \cup \cR} \to N \quad \textnormal{with} \quad e(v) \mapsto \mu(v)v \quad \textnormal{and} \quad e(\varrho)\mapsto \varrho \]
with $e(v)$ and $e(\varrho)$ denoting the natural basis vectors.

\begin{proposition}[\cite{coxComp1}, Section 2]
\label{prop:rationalCl}
For a complete rational complexity-one $T$-variety $\TV(\fan)$, one has the following exact sequence
\[0 \to (\ZZ^\cP/\ZZ)^\vee \oplus M \to \big(\ZZ^{\cV \cup \cR}\big)^\vee \to \cl(\TV(\fan)) \to 0 \]
where the first map is induced from $(Q,\phi)$.
\end{proposition}

\begin{example}
\label{ex:picRankCotangHirz}
We return to the projectivized cotangent bundle of $\cF_1$, cf.\ Example \ref{ex:cotangHirzebruch}. From the discussion above we see that its Picard rank is equal to $-2 -2 + 7 = 3$ which confirms a classical result.
\end{example}

\section{Canonical divisors, positivity, and divisor ideals}
\label{sec:furtherDivisors}

\subsection{The canonical divisor}
\label{subsec:canonicalDiv}
Given an $n$-dimensional toric variety $\TV(\Sigma)$ the invariant logarithmic $n$-form $\omega := \frac{d\chi^{e_1}}{\chi^{e_1}} \wedge \dots \wedge \frac{d\chi^{e_n}}{\chi^{e_n}}$, where $\{e_1,\dots,e_n\}$ denotes a basis of $M$, defines a rational differential form on $\TV(\Sigma)$. Moreover, the latter turns out to be independent of the chosen basis and thus leads to the description $K_{\TV(\Sigma)} = - \sum_{\varrho \in \Sigma(1)} D_\varrho$. Considering a torus action of complexity $k>0$, let $K_Y$ denote a representation of the canonical divisor on the base $Y$, i.e.\ $K_Y = \div \omega_Y$ for some rational differential form $\omega_Y \in \Omega^1(Y)$. Then one can construct a rational differential form
\[\Omega^1(\TV(\fan)) \ni \omega_{\TV(\fan)} := \omega_Y \wedge \frac{d\chi^{e_1}}{\chi^{e_1}} \wedge \dots \wedge \frac{d\chi^{e_k}}{\chi^{e_k}}\]
on the $T$-variety $\TV(\fan)$ where, as above, $\{e_1,\dots,e_k\}$ denotes a $\ZZ$-basis of the lattice $M$. By (formally) locally inverting the toric downgrade construction of (\ref{subsec:tDown}) the following equality is an immediate consequence of (\ref{subsec:toricPrinc}) and the above representation of $K_Y$:
\[K_{\TV(\fan)} = \sum_{(P,v)} (\mu(v)\,\coef_P K_Y + \mu(v) -1)\, D_{(P,v)} - \sum_{\varrho} D_\varrho\,.\]

\begin{example}
\label{ex:anticanDivCotangHirz}
Let us revisit $\PP(\Omega_{\cF_1})$ from Example \ref{ex:cotangHirzebruch}. Recall that $\cR(\pFan) = \emptyset$ and note that the anticanonical divisor $-K_{\PP(\Omega_{\cF_1})}$ is Cartier. As a Weil divisor it may be represented as
\[-K_{\PP(\Omega_{\cF_1})} = 2 \big (D_{([0],(0,0))}+ D_{([0],(0,1))} + D_{([0],(0,-1))}\big).\]
Our aim now is to provide the relevant data of the associated Cartier support function $h := \suppFunc(-K_{\PP(\Omega_{\cF_1})})$. To this end we denote the maximal cones of the tailfan by $\sigma_1,\dots,\sigma_8$ (starting with the cone $\langle (1,0),(1,1)\rangle$ and counting counter-clockwise). To illustrate the associated piecewise linear functions $h_P$ over the relevant slices $\pFan_P$ one may use the following list which displays the local representations $\suppFunc(u_i) + \suppFunc(E_i)$ of $-K_{\PP(\Omega_{\cF_1})}|_{\TV(\pDiv^i)}$ where $\pDiv^i$ is the p-divisor which is associated to the maximal cone $\sigma_i$:

\[\begin{array}{l} \begin{array}{c|c|c|c|c} & \sigma_1 & \sigma_2 & \sigma_3 & \sigma_4 \\ \hline u_i & [-2,0] & [0,-2] & [-2,-2] & [2,0] \\[.5ex] E_i & -2[0] + 2[1] & 0 & 0 & -2[0] +2[\infty] \end{array} \\[4ex]
\begin{array}{c|c|c|c|c} & \sigma_5 & \sigma_6 & \sigma_7 & \sigma_8 \\ \hline u_i & [2,0] & [2,2] & [0,2] & [-2,0] \\[.5ex] E_i & -2[0] + 2[\infty] & 0 & 0 & -2[0] + 2[1] \end{array} \end{array}\]
More specifically, this means that for every point $P \in \PP^1$ we have that $h_P|_{\pDiv^i_P} = u_i + \coeff_P E_i$ which is to be considered as an affine linear map $N_\QQ \supset |\pDiv^i_P| \to \QQ$.
\end{example}

\subsection{Positivity criteria}
\label{subsec:positivity}
We consider a complete complexity-one $T$-variety $\TV(\fan)$. For a cone $\sigma \in \tail(\fan)$ of maximal dimension and a point $P \in Y$ we denote by $\Delta^\sigma_P \in \fan_P$ the unique polyhedron whose tailcone is equal to $\sigma$. Given a Cartier support function $h \in \CaSF(\fan)$ we may write $h_P|_{\Delta^\sigma_P} = u(\sigma) + a_P(\sigma)$ and define $h|_\sigma(0) \;:=\; \sum_P a_P(\sigma) P$.

\begin{definition}
\label{def:concave}
A function $f:C  \to \QQ$ defined over some convex subset $C \subset \QQ^k$ is called \emph{concave} if $f(tx_1 + (1-t)x_2) \geq tf(x_1) + (1-t)f(x_2)$ for all $x_1, x_2 \in C$ and $0 \leq t \leq 1$. Given a polyhedral subdivision of $C$, the function $f$ is called \emph{strictly concave} if the inequality from above becomes strict for all pairs of points $(x_1,x_2) \in C^2$ which lie in different maximal cells.
\end{definition}

\begin{proposition}[\cite{tidiv}, Section 3.4]
\
\begin{enumerate}
\item $D_h$ is \emph{nef} if and only if $h_P$ is concave for all $P \in Y$ and $\deg h|_\sigma(0) \leq 0$ for every maximal cone $\sigma \in \tail(\pFan)$.
\item $D_h$ is \emph{semiample} if and only if $h_P$ is concave for all $P \in Y$ and $-h|_\sigma(0)$ is semiample for every maximal cone $\sigma \in \tail(\pFan)$.
\item $D_h$ is \emph{ample} if and only if $h_P$ is strictly concave for all $P \in Y$ and $-h|_{\tail(\pDiv)}(0)$ is ample for all maximal $\pDiv \in \fan$ with affine locus.
\end{enumerate}
\end{proposition}

Similar results for the case of varieties with general reductive group actions of complexity one can be found in \cite{timDivs}.

It is well known for toric varieties that the anti-canonical divisor is always big. This is not true in general for higher complexity $T$-varieties. Nevertheless, for rational $T$-varietes of complexity one, it is sufficient to impose the condition of connected isotropy groups. In our language, this is equivalent to the condition that all the vertices of $\pFan_P$ are lattice points. 

\begin{proposition}
  For  a rational $T$-variety $X$ of complexity one with connected isotropy groups, the anti-canonical divisor is big.
\end{proposition}
\begin{proof}
  By (\ref{subsec:canonicalDiv}) we get an effective anti-canonical divisor
$-K_X  = \sum_\rho D_\rho + 2\cdot \sum_v D_{Q,v}$ for any point $Q \in Y$. Moreover, the effective cone is generated by invariant divisor classes. Now, for every $D=D_\rho$ or $D=D_{P,v}$ we see that the class of $ (-K_X) - D$ is effective.  This implies that $-K_X$ sits in the interior of the effective cone, and is thus big.
\end{proof}

\subsection{The ideal of a Weil divisor}
\label{subsec:idealWD}
Previous considerations allow us to express the ideal
$I\subseteq \kG(\toric(\pDiv),\cO_{\toric(\pDiv)})=
\kG(\Loc\pDiv,\,\cO_{\Loc}(\pDiv))$
of a $T$-invariant prime divisor:
\[
\textstyle
\begin{array}{r@{\hspace{0.6em}}c@{\hspace{0.6em}}l}
I(D_{(Z,v)})
&=&
\bigoplus_{u\in\tail(\pDiv)\dual\cap M}
\{f\in\kG(\Loc,\pDiv(u))\kst
\div(f\,\chi^u)\geq D_{(Z,v)}\}\cdot\chi^u
\\[1.0ex]
&=&
\bigoplus_{u\in\tail(\pDiv)\dual\cap M}
\{f\in\kG(\Loc,\pDiv(u))\kst
\langle v,u\rangle + \ord_Z f> 0\}\cdot\chi^u.
\end{array}
\]
Since, if $f$ belongs to $\kG(\Loc \pDiv ,\pDiv(u))$,  the 
expression $\min \langle\pDiv_Z,u\rangle + \ord_Z f$ is non-negative,
we obtain for $D_{(Z,v)}$ the affine coordinate ring

\[
\textstyle
\begin{array}{r@{\hspace{0.6em}}c@{\hspace{0.6em}}l}
\kG(\Loc\pDiv,\,\cO)/I 
&=&
\bigoplus_{u\in\normal(v,\pDiv_Z)\cap M;\, \langle v,u\rangle\in\ZZ}
\kG(\Loc,\pDiv(u))/\kG(\Loc,\pDiv(u)-Z)
\\[1.0ex]
&\subseteq &
\bigoplus_{u\in\normal(v,\pDiv_Z)\cap M;\, \langle v,u\rangle\in\ZZ}
\kG(Z, \pDiv(u)|_{Z}).
\end{array}\]
Let $\varphi: \widetilde{Z} \rightarrow Z$ denote the normalization map. Furthermore, we define the cone $\sigma_{(Z,v)} := \QQ_{\geq 0} \cdot (\pDiv_Z -v)$ and denote by $\iota:Z \hookrightarrow Y$ the inclusion. Then we have that
\[\pDiv_{(Z,v)} := \sum_{W \subset Y} (\pDiv_W +\sigma_{(Z,v)})\otimes (\varphi \circ \iota)^* W\]
defines a p-divisor over $\widetilde{Z}$ with lattice $N_{(Z,v)} = (N + \ZZ v) \subset N_\QQ$ such that 
\[\TV(\pDiv_{(Z,v)}) = \wt{D_{(Z,v)}}\]
equals the normalization of the prime divisor $D_{(Z,v)} \subset \toric(\pDiv)$ (see \ref{subsec:toricOrbits}). Note that $(\varphi \circ \iota)^*Z$ is defined only up to linear equivalence, and different choices give rise to isomorphic $T$-varieties.

For a prime divisor $D_\rho \subset \toric(\pDiv)$ of horizontal type we obtain the ideal
\[\mathcal{I}_{\varrho}  := \bigoplus_{u \in \pDiv_Y^\vee \setminus \varrho^\perp} \cO(\pDiv(u))\]
in an analogous way. The corresponding p-divisor 
\[\pDiv_\varrho := \sum_{Z} \project_\varrho(\pDiv_Z) \cdot Z,\]
lives on $Y$ and has tailcone $\project_\varrho(\tail(\pDiv))$ with lattice $\project_\varrho(N)$, where $\project_\varrho$ is the projection $N_\QQ \rightarrow N_\QQ/\QQ\cdot\varrho$. As above, we have that $D_\varrho = \TV(\pDiv_\varrho)$.

\section{Cohomology groups of line bundles in complexity 1}
\label{sec:cohomologyGroups}

\subsection{Toric varieties}
\label{subsec:cohomTorVar}
Consider an equivariant line bundle $\cL$ on the toric variety $\TV(\Sigma)$. The induced torus action on the cohomology spaces $H^i(\TV(\Sigma),\cL)$ yields a weight decomposition of the latter, i.e.\ $H^i(\TV(\Sigma),\cL) = \bigoplus_{u\in M} H^i\big(\TV(\Sigma),\cL\big)(u)$. As in (\ref{subsec:CDtoric}), we denote by $\psi: |\Sigma| \to \QQ$ a continuous piecewise linear function representing the line bundle $\cL$. Introducing the closed subsets $Z_u \;:=\; \{v \in N_{\RR} \;|\; \bangle{u,v} \geq \psi(v)\}$, the complex with entries $H^0(\sigma, Z_u; \CC)$ associated to the covering of $|\Sigma|$ by the cones $\sigma \in \Sigma$ can be identified with the usual \v{C}ech complex for $\TV(\Sigma)$ given by the covering via the affine open subsets $\TV(\sigma)$. This gives us that
\[H^i\big(\TV(\Sigma),\cL \big)(u) = H^i\big(|\Sigma|,Z_u;\CC \big),\]
cf.\ \cite{demazureTV}. It follows that if $\Sigma$ is complete and $h$ is concave, e.g.\ $\cL$ is globally generated, then
\[H^i(\TV(\Sigma),\cL) = 0 \quad \textnormal{for} \quad i>0\,,\]
see \cite[Corollary 7.3]{danilov}. One also has an analogous result for toric bouquets which is proved by induction on the number of irreducible components and the fact that the restriction homomorphism on the level of global sections is surjective (cf.\ \cite[Lemma 6.8.1]{danilov}):

\begin{proposition}[\cite{phdLars}, Section 2.4]
\label{prop:cohomDappledTVLB}
Let $X$ be a complete equidimensional toric bouquet and $\cL$ a nef line bundle on $X$. Then $H^i(X,\cL) = 0$ for $i>0$.
\end{proposition}

\subsection{Global sections}
\label{subsec:globalSections}
Recall from (\ref{subsec:CDtoric}) that a globally generated equivariant line bundle $\cL$ on a complete toric variety $\TV(\Sigma)$ corresponds to a polytope $\nabla \subset M_\QQ$ such that $\Gamma(\TV(\Sigma),\cL) = \bigoplus_{u \in \nabla \cap M} \CC \chi^u$. In the following we will see how to generalize this formula to complexity-one $T$-varieties.

Consider a $T$-invariant Weil divisor $D$ on a complexity-one $T$-variety $\TV(\fan)$ over the curve $Y$ and define the polyhedron
\[\Box_D := \convHull \; \{u\in M_\QQ \kst \langle\varrho,u \rangle \geq -\coef_\varrho D \mbox{ for all } \varrho\in\kR\}\,,\]
where $\coef_\varrho D$ denotes the coefficient of $D_\varrho$ inside $D$. Moreover, there is a map $D^*:\Box_D \to \kDiv_\QQ Y$ with $\coef_P D^*(u) := \min\{\langle v,u \rangle +\coef_{(P,v)}D/\mu(v) \kst v \in \pFan_P(0)\}$. As a direct consequence of Theorem \ref{thm:principDiv}
one obtains that
\[\kG\big(X,\cO_X(D)\big)(u)=\kG \big(Y, \cO_Y(D^*(u))\big),\]
cf.\ \cite[Section 3.3]{tidiv}. In the case that $D = D_h$ is Cartier the function $D^*$ is usually denoted by $h^*$ and we write $\Box_h$ instead of $\Box_{D}$.

\begin{example}
\label{ex:globalSecAnticanCotangHirz}
Let us return to Example \ref{ex:anticanDivCotangHirz} and compute $\dim \Gamma(\PP(\Omega_{\cF_1}),-K_{\PP(\Omega_{\cF_1})})$. The weight polytope of the anticanonical divisor is pictured in Figure \ref{fig:box}. Together with all weights $(u_1,u_2) = u \in \Box_h$ the following list displays the degree of the induced divisor $h^*(u)$ on $\PP^1$:
\[\begin{array}{c|rrrrrrrrrrrrrr} u_1 & 0 & 1 & 2 & -1 & 0 & 1 & 2 & -2 & -1 & 0 & 1 & 2 & -1 & 0 \\ u_2 & 2 & 2 & 2 & 1 & 1 & 1 & 1 & 0 & 0 & 0 & 0 & 0 & -1 & -1 \\[.5ex] \hline \deg h^*(u) & 0 & 0 & 0 & 0 & 1 & 1 & 0 & 0 & 1 & 2 & 1 & 0 & 0 & 0 \end{array}\]
Summing up over all degrees yields $\dim \Gamma(X,-K_X) = 20$.
\end{example}

\begin{figure}[htbp]
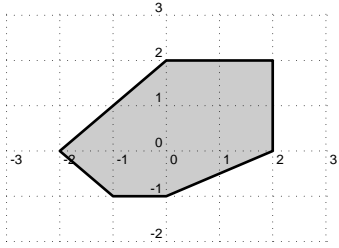

\centering
\boxCotangHirzebruchOne
\caption{Weight polytope associated to $-K_\PP(\Omega_{\cF_1})$, cf.\ Example \ref{ex:globalSecAnticanCotangHirz}.}
\label{fig:box}
\end{figure}

\subsection{Higher cohomology groups for complexity-one torus actions}
\label{subsec:higherCohomology}
We suspect that higher cohomology group computations for equivariant line bundles on complexity-one $T$-varieties are rather hard in general, in particular if one cannot directly make use of a quotient map to the base curve $Y$. Hence, we restrict to the case where $X = \wt{X}$. Invoking Proposition \ref{prop:cohomDappledTVLB}, one can directly generalize the toric vanishing result from (\ref{subsec:cohomTorVar}) with a ``cohomology and base change'' argument for the flat projective morphism $\pi: \wt{X} \to Y$ to the complexity-one case:

\begin{proposition}[\cite{phdLars}, Section 2.4]
\label{prop:nefCohom}
Let $X = \wt{X}$ be a complete complexity-one $T$-variety over the base curve $Y$. For any nef $T$-invariant Cartier divisor $D_h$ on $X$ we have that $R^i\pi_* \cO_X(D_h) = 0$ for $i>0$. Hence,
\[H^i(X,\cO_X(D_h)) = \bigoplus_{u \in \Box_h \cap M} H^i\big(Y,\cO_Y(\hstar(u))\big)\,.\]
In particular, $H^i(X,\cO_X(D_h)) = 0$ for $i \geq 2$.
\end{proposition}

A precursor of this result for smooth projective surfaces and semiample line bundles is formulated as Corollary 3.27 in  \cite{tcodes}. Its proof uses the well known intersection theory for smooth projective surfaces, cf.\ \cite[Chapter V]{ha}.

Furthermore, it is worthwhile to note that the direct image sheaf always splits into a direct sum of line bundles, namely $\pi_* \cO_{\wt{X}}(D_h) = \bigoplus_{u \in \, \square_h \cap M} \cO_Y(h^*(u))$ with no positivity assumptions on $D_h$.
Since the higher direct image sheaves are not locally free in general, explicit higher cohomology group computations have so far proven unsuccessful.

\section{Cox rings as affine $T$-varieties}
\label{sec:CoxR}

Assume $X$ to be a complete normal variety with finitely generated divisor class group $\cl(X)$. The \emph{Cox ring} of $X$ is then given as  the $\Cl(X)$-graded abelian group
\[\Cox(X) := \bigoplus_{D\in\Cl(X)}\kG(X,\CO_X(D))\]
which carries a canonical ring structure, see \cite[Section 2]{tvarcox}.
In analogy to \cite{MDS}, but neither supposing $X$ to be $\QQ$-factorial nor projective, we call $X$ a Mori dream space (MDS) if $\Cox(X)$ is a finitely generated $\CC$-algebra.

\subsection{The Cox ring of a toric variety}
\label{subsec:toricCox}
The well known quotient construction $\PP^n \cong \big(\spec\CC[z_0,\dots,z_n]\setminus V(z_0,\dots,z_n)\big)/\Cstar$ can be generalized to toric varieties $\TV(\Sigma)$ for which the set of rays $\Sigma(1)$ generates $N_\QQ$ (cf.\ \cite{cox}).

The graded homogeneous coordinate ring $\CC[z_0,\dots,z_n]$ for $\PP^n$ is replaced by the polynomial ring $\CC[z_\varrho \kst \varrho \in \Sigma(1)]$ whose variables correspond to the rays of $\Sigma$ and whose grading is induced by the divisor class group $\cl(\TV(\Sigma))$, which assigns the degree $[D_\varrho]$ to the variable $z_\varrho$. Furthermore, applying $\Hom_\ZZ(-,\Cstar)$ to the short exact sequence from (\ref{subsec:toricPrinc}), one defines the algebraic group $G := \ker\big((\CC^*)^{\Sigma(1)}\to T_N\big)$ where $T_N := N \otimes_\ZZ \Cstar = \Hom_\ZZ(M,\Cstar)$. Note that $G$ acts naturally on $\CC^{\Sigma(1)}$ and leaves $V(B(\Sigma))$ invariant, where $B(\Sigma) \subset \CC[z_\varrho \kst \varrho \in \Sigma(1)]$ denotes the \emph{irrelevant ideal} 
which is generated by $\big\{z_{\hat{\sigma}} := \prod_{\varrho \notin \sigma(1)} z_\varrho \;\big|\; \sigma \in \Sigma \big\}$. The ideal $B(\Sigma)$ generalizes the standard 
irrelevant ideal $\langle z_0,\dots,z_n \rangle \subset \CC[z_0,\dots,z_n]$ of projective space.
 Furthermore, it follows that $\TV(\Sigma)=\big(\CC^{\Sigma(1)}\setminus V(B(\Sigma))\big)/G$ is a good quotient.

On the other hand, one can also approach the Cox ring via a more polyhedral point of view.
As usual we denote the first non-trivial lattice point on
a ray $\varrho \in \Sigma(1) \subseteq N_\QQ$ with the same letter~$\varrho$. Then we consider the canonical map $\varphi:\Z^{\Sigma(1)}\to N$, $e(\varrho) \mapsto \varrho$.
It sends some faces (including the rays) of the positive orthant $\QQ^{\Sigma(1)}_{\geq 0}$ to cones of the fan $\Sigma$. Applying the functor $\TV$, we obtain a rational map $\Spec \CC[z_\varrho \kst \varrho \in \Sigma(1)] \rato \TV(\Sigma)$.
In particular, we recover the affine spectrum of
$\CC[z_\varrho \kst \varrho\in\Sigma(1)] = \Cox(\TV(\Sigma))$
as the toric variety $\TV(\QQ^{\Sigma(1)}_{\geq 0})$.
Thus, the Cox ring of a toric variety gives rise to an affine toric variety itself, and the defining cone $\QQ^{\Sigma(1)}_{\geq 0}$ can be seen as a polyhedral resolution of the given fan $\Sigma$, since all linear relations among the rays have been removed.

\subsection{The action of the Picard torus}
\label{subsec:picardTorus}
Let $X$ be an MDS and suppose that $\cl(X)$ is torsion free. It follows that the total coordinate space $\spec \Cox(X)$ is a normal affine variety, and the $\cl(X)$-grading encodes an effective action of the so-called Picard torus $T := \gHom_{\algGr}(\cl(X),\CC^*)$. One could now ask for a description of $\spec \Cox(X)$ in terms of some p-divisor $\cE_{\Cox}$ on the so-called Chow quotient $Y := \spec \Cox(X) \chQ T$ which is defined as the normalization of the distinguished component of the inverse limit of the GIT quotients of $\spec \Cox(X)$, cf.\ \cite[Section 6]{tvar_1}. This has been done in \cite{coxmds}, and in the case of smooth Mori dream surfaces, the result is as follows (cf.\ \cite[Section 6]{coxmds}):
$Y=X$ and, up to shifts of the polyhedral coefficients, $\cE_{\Cox} = \sum_{E\subseteq X}\Delta_E \otimes E$ with
\[\Delta_E = \{D \in \Eff(X) \subseteq \cl_\QQ(X) \kst
(D\cdot E)\geq -1 \mbox{ and } (D\cdot E')\geq 0 \mbox{ for } E'\neq E \}\]
where $E,E'$ run through all negative curves in $X$. The common tailcone of the $\Delta_E$ equals $\Nef(X)$, which is dual to $\Eff(X)$, where the latter cone carries the degrees of $\Cox(X)$. If  $X$ is a del Pezzo surface, the formula for $\Delta_E$ simplifies to
$\Delta_E = \ko{0 E} + \Nef(X) \subseteq \cl_\QQ(X)$.

\subsection{Generators and relations for complexity-one torus actions}
\label{subsec:coxComp1alg}
It is a fundamental problem to give a presentation of the Cox ring of an MDS in terms of generators and relations. The crucial idea to approach this problem for a $T$-variety $\TV(\pFan)$  is to relate its presentation to an appropriate quotient of $\TV(\pFan)$ by the torus action. This ansatz was pursued in \cite{tvarcox} whose key result states that the Cox ring of $\TV(\pFan)$ is equal to a finitely generated algebra over the Cox ring of $Y^\circ$, where $Y^\circ\subseteq Y$  is the image of the rational map coming from the composition of the rational $p^{-1}:\toric(\pFan)\to\ttoric(\pFan)$ with $\pi:\ttoric(\pFan)\to Y$. Let us now become more specific for complexity-one torus actions.
For a point $P \in \PP^1$ together with the set $\cV_P := \{D_{(P,v)}\,|\, v \in \pFan_P(0)\}$ of all vertical divisors lying over $P$, we define the tuple $\mu(P) \;:=\; \big(\mu(v) \;|\; v \in \pFan_P(0)\big).$
After applying an automorphism of $\PP^1$, we may assume that $\cP = \{0,\infty,c_1,\ldots, c_r\}$, with $c_i \in \CC^*$ (cf.\ (\ref{subsec:classT})).

\begin{theorem}[\cite{tvarcox}, Theorem 1.2]
\label{thm:coxCplxOneExplicit}
The Cox ring of $\toric(\fan)$ is given by
\[\CC\big[T_{D_{(P,v)}}, S_{D_\rho} \;|\; P \in \cP, \varrho \in \cR \big] \Big/ \Big\langle  T^{\mu(0)}+c_iT^{\mu(\infty)}+T^{\mu(c_i)}  \;\Big|\; i=1,\ldots,r  \Big\rangle \,,\]
where $T^{\mu(P)} := \prod_{v \in \pFan_P(0)} T_{D_{(P,v)}}^{\mu(v)}$ and $\cR$ is as in Definition \ref{def:extremalRays}.
\end{theorem}

\begin{example}
\label{ex:coxCotangHirzAlg}
Let us describe the Cox ring of $\PP(\Omega_{\cF_1})$, cf.\ Example \ref{ex:cotangHirzebruch}. Using Theorem \ref{thm:coxCplxOneExplicit}
and representatives $(1,1)$, $(1,0)$, and $(0,1)$ for $1$, $0$, and $\infty$, respectively, we see that $\Cox(\PP(\Omega_{\cF_1}))  = \CC[T_1,\dots,T_7]/(T_1T_2T_3-T_4T_5-T_6T_7)$.
\end{example}

\subsection{Polyhedral point of view for complexity-one torus actions}
\label{subsec:coxComp1pol}
Similar to the viewpoint taken in (\ref{subsec:picardTorus}) and keeping the setting from (\ref{subsec:coxComp1alg}), one can describe $\spec \Cox(X)$ via a p-divisor $\pDiv_{\Cox}$. In this case, however, the given $T$-action on $X$ carries over to $\Cox(X)$, and by combining it with the action of the Picard torus, $\Spec \Cox(X)$ turns into a variety with an effective complexity-one action of a diagonizable group which in general involves torsion. Factoring out the latter gives rise to a finite abelian covering $C \to \PP^1$ (see (\ref{subsec:GalCovP1})), so that $\pDiv_{\Cox}$ lives on a curve $C$ of usually higher genus rather than on $\PP^1$. Nonetheless, if the class group $\cl(X)$ is torsion free $C \to \PP^1$ is the identity map. In this case, the p-divisor $\pDiv_{\Cox}$ describing $\Spec \Cox(X)$ utilizes the very same points $P \in \PP^1$ as $\pFan$.
Denoting by $\kV := \bigcup_{P \in \kP} \pFan_P(0)$, we define the compact polyhedra
\[\Delta_P^c := \conv\{e(v)/\mu(v)\kst v\in \pFan_P(0)\}
\subseteq \kFQ\]
where the $e(v)\in\Z^{\kV}$ denote the canonical basis vectors and $\cR$ is as usual. We similarly define  the polyhedral cone
\[\sigma \;:=\; \QQ_{\geq 0}\cdot\prod_{S \in \kP} \Delta_P^c 
\;+\; \kpR \;\subseteq\; \kpVR.\]

\begin{proposition}[\cite{coxComp1}, Theorem 1.2]
\label{prop:coxTorsionFree}
Assume $\cl(X)$ to be torsion free. Then the p-divisor $\pDiv_{\Cox}$ of
$\Spec \Cox(X)$ is, up to shifts of the polyhedral coefficients,
given by $\sum_{P\in\kP}(\Delta_P^c+\sigma)\otimes [P]$
on $\PP^1$. 
\end{proposition}

For more details and the general result which also covers the case of a class group with torsion we refer to \cite[Theorem 4.2]{coxComp1}.
Comparing Proposition \ref{prop:coxTorsionFree} with the analogous result in the toric setting (\ref{subsec:toricCox}),
we see that it consists of a similar construction which resolves the linear dependencies among the elements of $\pFan_P(0) \in N_\QQ$ and the rays $\varrho \in \cR$ by assigning separate dimensions to each
of them.

\begin{example}
\label{ex:coxCotangHirzPol}
Once more we consider the threefold $\PP(\Omega_{\cF_1}))$, cf.\ Example \ref{ex:cotangHirzebruch}. The tailcone $\sigma := \tail(\pDiv_{\Cox})$ of the p-divisor $\pDiv_{\Cox}$ is then given as the cone over the product of a quadrangle (product of two intervals) and a triangle. This realization is induced by the three polytopes $\Delta^c_1$ (triangle), $\Delta^c_0$ (compact interval), and $\Delta^c_\infty$ (compact interval). The non-trivial coefficients of $\pDiv_{\Cox}$ are thus equal to $\Delta^c_P + \sigma$, $P \in \{1,0,\infty\}$ (up to a shift).
\end{example}

It turns out that the polyhedral divisors $\cE_{\Cox}$ (as constructed in (\ref{subsec:picardTorus})) and $\pDiv_{\Cox}$ are related by upgrade and downgrade constructions via the exact sequence from above, cf.\ \cite{upDowngrade} and \cite[Section 3.5]{phdLars}.

\section{Invariant valuations and proper equivariant morphisms}
\label{sec:properness}

For a function field $\CC(X)$ of a variety $X/\CC$, we consider discrete $\CC$-valuations of $F$. These are maps $\nu: \CC(X) \rightarrow \QQ \cup\{\infty\}$ fulfilling the properties
\begin{enumerate}
\item $\nu(f\cdot g) = \nu(f)+\nu(g)$,
\item $\nu(f+g) \geq \min(\nu(f),\nu(g)$),
\item $\nu(\CC^*)=0$, and $\nu(f)=\infty \Leftrightarrow f=0$.
\end{enumerate}

A center of a valuation is a point $x \in X$ such that $\nu(\CO_{X,x}) \geq 0$ and $\nu(\mathfrak{m}_{X,x}) > 0$.  
 By the valuative criterion for properness \cite[Theorem 4.7]{ha}, a scheme is complete if every discrete valuation has a unique center. Moreover, it is separated if there exists at most one center.

\subsection{Completeness and properness for toric varieties}
\label{subsec:propernessToric}
For a toric variety $X=\toric(\Sigma)$, the function field of $X$ equals the quotient field of $\CC[M]$ and a valuation $\nu$ defines a linear form on $M$ via $v:u \mapsto \nu(\chi^u)$. Moreover, this gives a one-to-one correspondence between $N_\QQ$ and the discrete valuations on $\CC(X)$. Here, valuations with center on $X$ correspond to elements in the support of $\Sigma$. Hence, a toric variety $\toric(\Sigma)$ is complete if and only if its fan is complete, i.e. $|\Sigma|=N_\QQ$. Moreover, let  $F:\Sigma \rightarrow \Sigma'$ be a map of fans. Then $F$ defines a proper morphism exactly when $|\Sigma| = F^{-1}(|\Sigma'|)$.

\subsection{Valuations on $T$-varieties and hypercones}
The function field of a $T$-variety $X=\toric(\pFan)$ equals the quotient field of $\CC(Y)[M]$. An invariant valuation corresponds via the relation $\nu(f\cdot \chi^u) = \mu(f) + \langle u, v \rangle$ to a pair $(\mu,v)$, where $\mu$ is a valuation on $\CC(Y)$ and $v \in N_\QQ$. In this situation we write simply $\nu=(\mu,v)$.

If $\mu$ is a valuation with center $y \in Y$, then we get a well defined group homomorphism
\[\mu:\CaDiv_{\geq 0}(Y) \rightarrow \QQ_{\geq 0};\quad D \mapsto \mu(f),\]
where $D=\div(f)$ locally at $y$. This map extends to a homomorphism $\mu: \Pol^+(N,\sigma)\otimes_{Z_{\geq 0}} \CaDiv_ {\geq 0} (Y) \rightarrow \Pol^+(N,\sigma)$. 
Now we have the following

\begin{proposition}[{\cite[Lemma~7.7]{tvar_2}}]
\label{prop:val-center}
  A valuation $(\mu,v)$ has a center on $\toric(\pDiv)$ if and only if $v \in \mu(\pDiv)$.
\end{proposition}
\begin{remark}
	The criterion $v \in \mu(\pDiv)$ implies in particular that $\mu$ has a center in the locus of $\pDiv$.
\end{remark}
Now, we turn to the case of complexity one. In this situation $Y$ is a smooth complete curve and the non-negative multiples $m \cdot \ord_y$ of the vanishing orders at points $y \in Y$ are the only discrete valuations. For those $\mu$ we thus have $\mu(\pDiv)=m \cdot \pDiv_y$ which is equal to the tailfan if $m=0$.

For an invariant valuation $(\mu,v)$ with center on $X=\toric(\pDiv)$, we must either have $\mu \equiv 0$ and $v \in \tail(\pDiv)$ or $\mu \not \equiv 0$ and $(v,m) \in \langle \pDiv_y \times \{1\}\rangle \subset N_\QQ \times \QQ$. Thus, the set of invariant valuations with center on $X$ can be identified with the disjoint union of the polyhedral cones $\langle \tail(\pDiv) \times \{0\} \cup \pDiv_y \times \{1\}\rangle$ glued together along the common subsets $\tail(\pDiv) \times \{0\}$.  Such an object is called a hypercone. It was introduced by Timashev in \cite{timGVar} (in a version suitable to describe more generally reductive group actions of complexity one). Since the valuations with center already determine an affine variety, a hypercone gives an alternative description of a affine $T$-variety of complexity one. In this language a polyhedral divisor is nothing but a (hyper-)cut through a hypercone.

\subsection{Complete $T$-varieties and proper morphisms}
Consider a $T$-variety $X = \TV(\pFan)$ with $\fan=\{\pDiv_i\}_{1\leq i \leq r}$. The divisorial fan $\pFan$ provides us with an open affine torus invariant covering of $X$ by the varieties $U_i=\toric(\pDiv_i)$, $1 \leq i \leq r$. As already mentioned at the start of (\ref{sec:properness}), the valuative criterion of completeness states that a variety $X$ is complete exactly when every valuation of $\CC(X)$ has a unique center. For $T$-varieties it is in fact sufficient to restrict oneself to \emph{invariant} valuations. Using the results of the previous section we obtain the following 

\begin{theorem}[{\cite[Theorem~7.5]{tvar_2}}]
\label{sec:thm-complete}
  $\toric(\pFan)$ is complete if and only if every \emph{slice} $\mu(\pFan) :=\{\mu(\pDiv) \mid \pDiv \in \pFan\}$ is a complete subdivision of $N_\QQ$.
\end{theorem}

As a relative version of this result we obtain a characterization of proper equivariant morphisms. Recall from (\ref{subsec:mapsT}) that an orbit dominating equivariant morphism $X' \rightarrow X$ of $T$-varieties is given by a triple $\psi = (\varphi, F, \f)$, consisting of a dominant morphism $\varphi:Y' \rightarrow Y$, a lattice homomorphism $F:N' \rightarrow N$ and an element $\f \in \CC(Y') \otimes N$. First we define the preimage of a p-divisor over $Y$ and $N$ via $\psi$ as a polyhedral divisor over $Y'$ and $N$
\[\psi^{-1} \pDiv := F^{-1}(\varphi^*\pDiv + \div(\f)).\]

Let $X' = \TV(\pFan')$ be a $T$-variety with $\fan' = \{\pDiv'_i\}_{1 \leq i \leq r}$.  Assume that $\psi$ defines a map of polyhedral divisors $\pDiv'_i \rightarrow \pDiv$ for every $i$.

\begin{theorem}
\label{sec:thm-proper}
The morphism $X' \rightarrow X=\toric(\pDiv)$ induced by $\psi$ is proper if and only if $\fan'_\mu:=\{\mu(\pDiv'_i)\}_i$ defines a complete polyhedral subdivision of $\mu(\psi^{-1}(\pDiv))$ for every $\mu$.
\end{theorem}
Since the properness can be checked locally on the base, the above theorem can be used to determine whether an arbitrary orbit dominating equivariant morphism $X' \rightarrow X$ is proper.

\begin{remark}
	It follows from \cite[Proposition 1.5]{andreas+nathan} that it suffices to check the conditions of Theorems~\ref{sec:thm-complete} and \ref{sec:thm-proper} just for slices of type $\fan_y$ instead of for all slices $\fan_\mu$.
\end{remark}

\section{Resolution of singularities}
\label{sec:singCanDiv}

\subsection{Toric resolution of singularities}
The only non-singular affine toric varieties are 
products of tori and affine spaces and are given by cones $\sigma$ spanned by a subset of a lattice basis. We will call such  cones regular. Reducing to the sublattice $N \cap (\QQ \cdot \sigma)$ we may assume that $\sigma$ is of maximal dimension. A cone $\sigma=\langle a^1,\ldots, a^r  \rangle$ is regular if and only if it is simplicial and 
its \emph{multiplicity} $m(\sigma):=\det (a^1, \ldots, a^r)$ equals $\pm 1$.

Since equivariant proper birational morphisms to an affine toric variety $X=\toric(\sigma)$ correspond to fans subdividing the cone $\sigma$, a resolution of $X$ is given by a triangulation of $\sigma$ into regular cones. Indeed, such a triangulation always exists. It can be constructed by starting with an arbitrary triangulation and subsequently refining non-regular cones $\tau=\langle a^1, \ldots, a^r \rangle$ with rays along lattice elements $\tilde{a} \in \sum_i [0,1)\cdot a^i$. On the one hand, every non-regular cone contains such a lattice element. On the other, the resulting cones of the stellar subdivision along $\tilde{a}$ have strictly smaller multiplicities than $\tau$. Indeed, if $\tilde{a}=\sum \lambda_i a^i$ we get $m(\tau_i)= \lambda_i m(\tau)$, where $\tau_i$ is the cone spanned by $\tilde{a}$ and all the generator of $\tau$ except $a^i$. So after a finite number of steps we end up with a subdivision in regular cones.

This combinatorial approach to resolving toric singularities was generalized by Varchenko in \cite{varchenkoRes} to obtain (embedded) resolutions of non-degenerate hypersurface singularities. Consider a regular cone $\omega \subset M_\QQ$ and a polynomial 
$h=\sum_{u \in M} \alpha_u \chi^u \in \CC[\omega \cap M] \cong \CC[x_1,\ldots, x_n]$.
Now, the normal fan $\Sigma_h$ of the unbounded polyhedron $\nabla = \conv(\supp h + \omega)$ refines the dual cone $\sigma=\omega\dual$. For a compact face $F \prec \nabla$ we consider the restricted polynomial $h_F = \sum_{u \in F \cap M} \alpha_u \chi^u$. If all restricted polynomials define smooth hypersurfaces in the torus $T=\spec \CC[M]$, we call $h$ \emph{non-degenerated}. 

\begin{theorem}
  If $h$ is non-degenerated and $\varphi:\toric{(\Sigma')} \rightarrow \toric{(\Sigma_h)}$ is a toric resolution of singularities then the strict transform of $V(h)$ under the induced morphism $\toric(\Sigma') \rightarrow \toric(\sigma)=\AA^n$ is an (embedded) log-resolution of the singularities of $V(h)$.
\end{theorem}

\begin{example}
\label{exmp-varchenko}
  We consider the surface singularity $X$ given by the equation $h=x^2+y^3+z^5$. In a first step, we refine the positive orthant $\sigma=\omega^\vee$ by replacing it with the normal fan $\Sigma_h$ of the Newton diagram of $h$. This Newton diagram and its normal fan are pictured in Figure~\ref{fig:varchenko}; $\Sigma_h$ is obtained from the positive orthant $\sigma$ via a stellar subdivision along the ray $(15,5,6)$.  In a second step we refine this fan and obtain $\Sigma'$ as pictured in Figure~\ref{fig:var-res}.

The first step introduces a \emph{central} exceptional divisor in $X$ corresponding to the ray $\rho =\langle (15,5,6) \rangle$, with the remaining singularities in the strict transform of $X$ being toric. In the second step, the singularities of $\toric(\Sigma_h)$ are resolved, which also resolves the (toric) singularities of the strict transform of $X$.
\end{example}


\begin{figure}
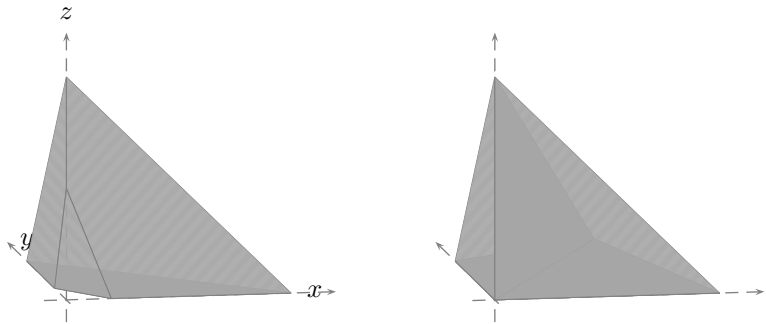

  \centering
  \newtonDiagram
  \newtonDual
  \caption{The Newton diagram of $h$ and its normal fan $\Sigma_h$. }
  \label{fig:varchenko}
\end{figure}

\begin{figure}
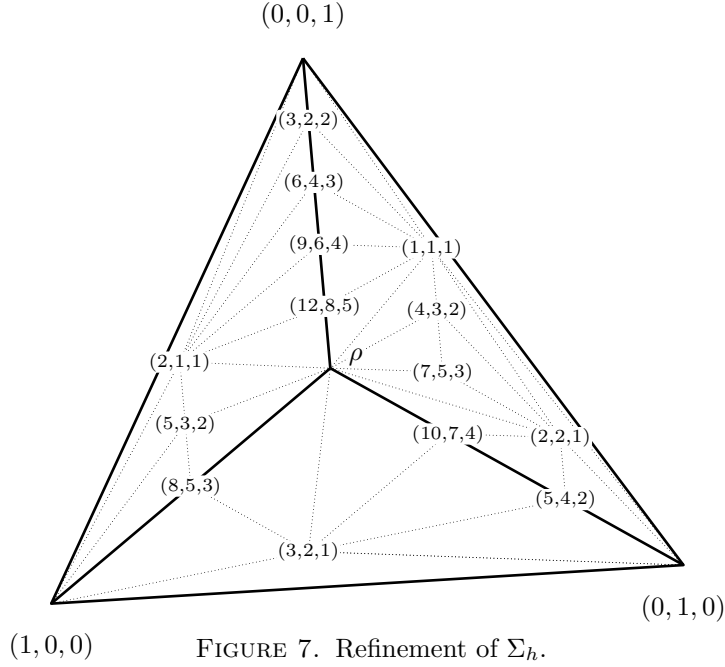

  \centering
  \varchenkoResolution
  \caption{Refinement of $\Sigma_h$.}
   \label{fig:var-res}
\end{figure}

\begin{remark}
  The approach of Varchenko and the notion of non-degeneracy is generalized in the context of tropical geometry, see \cite{1154.14039}.
\end{remark}

\subsection{Toroidal resolutions}
\label{sec:toroidal-resolutions}
We recall from \cite{mumford,danilovDeRham} that a complex normal variety $X$ is \emph{toroidal at the closed point} $x \in X$ if there exists an affine toric variety $\TV(\sigma)$  such that the germ $(\an{X},x)$ is analytically isomorphic to the germ $(\an{\TV(\sigma)},x_\sigma)$ with $x_\sigma$ being the unique fixed point in $\TV(\sigma)$. If the above property holds for every $x \in X$ then $X$ itself is called \emph{toroidal}.

More  generally, let $B \subset X$ be a closed subvariety. The pair $(X,B)$ is called \emph{toroidal} if for every closed point $x \in X$, there exists a toric variety $\TV(\sigma)$ together with a distinguished set of faces $F(\sigma)$ of $\sigma$ such that the germ $(\an{X},\an{B};x)$ is analytically isomorphic to the germ $(\an{\TV(\sigma)},\bigcup_{\tau \in F(\sigma)} \an{\TV(\tau)};x_\sigma)$.

Given a p-divisor on some $Y$, we consider a log-resolution $\varphi$ of the pair $(Y, \supp \pDiv)$. The pullback $\varphi^* \pDiv$ leads to the same variety, i.e. $X=\toric(\pDiv) = \toric(\varphi^*\pDiv)$, but by \cite[Proposition~2.6]{tsing} we get a toroidal pair $(\widetilde{X},B)$ consisting of the variety $\widetilde{X}= \ttoric(\varphi^*\pDiv)$ and the exceptional set $B$ of the contraction morphism $r:\ttoric(\varphi^*\pDiv) \rightarrow \toric(\varphi^*\pDiv)$.

 Hence, $\ttoric(\varphi^*\pDiv)$ is called a \emph{toroidal resolution} of $X$. Moreover, for any closed point $x\in\widetilde{X}$, $\widetilde{X}$ is formally locally isomorphic in a neighborhood of $x$ to the toric variety which is given by the so-called ``Cayley cone''
\[C(\pDiv_{Z_1},\ldots, \pDiv_{Z_r})
= \langle(\tail(\pDiv)) \times 0,\; \pDiv_{Z_1} \times e_1, \ldots, \pDiv_{Z_r}\times e_r \rangle \subset N_\QQ \times \QQ^r.\]
Here, $Z_1, \ldots, Z_r$ denotes a set of prime divisors intersecting transversely in the point $\pi(x)\in Y$. Note that one can obtain a resolution of singularities from this construction by purely toric methods when applying them to the occurring Cayley cones. In particular, $\ttoric(\pDiv)$ is smooth if and only if all the Cayley cones $C(\pDiv_{Z_1},\ldots, \pDiv_{Z_r})$ are regular. If $\Loc \pDiv$ is affine we have that $\ttoric(\pDiv)=\toric(\pDiv)$ and thus obtain a smoothness criterion in this case. If $\Loc \pDiv$ is complete then, as in the toric situation, $\toric(\pDiv)$ is smooth if and only if
it is a product of toric and affine spaces \cite[Lemma~5.2]{tsing}. Hence, $\pDiv$ can be obtained by a toric downgrade, see Section~\ref{sec:tDown}. For general loci the situation seems to be more complicated.

\begin{remark}
	The theory of toroidal embeddings~\cite{mumford} associates with the toroidal resolution $\widetilde{X}\to X$
a fan that is glued from the Cayley cones described above. The same construction shows
that p-divisors generalize Mumford's description of complexity-one
torus actions~\cite[Chapter 4, \S 1]{mumford}: if
$\pDiv = \sum \pDiv_P \otimes P$ is a p-divisor on a curve, then the corresponding toroidal
fan is obtained by gluing the cones
$C(\pDiv_P) = \langle (\tail(\pDiv)) \times 0, \pDiv_P \times 1 \rangle$
along their common face $\tail(\pDiv)$~\cite{toroidal}. This in turn 
coincides with Timashev's description in terms of hypercones and 
hyperfans~\cite{timGVar,timTorus}.
\end{remark}

\begin{example}
Given $N=\ZZ$, $Y=\PP^2=\proj(\CC[x,y,z])$ and the two  prime divisors $Z_1=V(y)$ and $Z_2=V(x^2-zy)$ we consider the p-divisor
\[\pDiv=[\nicefrac{1}{2},\infty)\otimes Z_1 + [-\nicefrac{1}{3},\infty)\otimes Z_2.\]
The corresponding affine variety is the hypersurface in $\AA^4$ given by $x_1x_2^2-x_3^2-x_4^3$. To obtain a toroidal resolution we first need to blow up the intersection points of $Z_1$ and $Z_2$ or their strict transforms, respectively, in order to get to a normal crossing situation, which is pictured in Figure~\ref{fig:nc}.

\begin{figure}
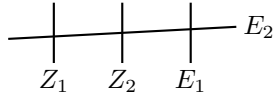

\centering  
\normalCrossing
\caption{Normal crossing situation after three blow ups.}
\label{fig:nc}
\end{figure}

Let us denote the morphism coming from the composition of these blowups by $\varphi: \widetilde{Y}\rightarrow Y$. For the pullback of $\pDiv$ we obtain
\[ \varphi^*\pDiv=[\nicefrac{1}{2},\infty)\otimes Z_1 + [-\nicefrac{1}{3},\infty)\otimes Z_2  +[\nicefrac{1}{6},\infty)\otimes E_1+[\nicefrac{1}{3},\infty)\otimes E_2.\]
Here, $E_1$ and $E_2$  denote the exceptional divisors. Now, $\ttoric(\varphi^*\pDiv) \rightarrow X$ is a toroidal resolution. We will consider the intersection point $Z_1 \cap E_2$. The corresponding Cayley cone is illustrated in the left picture of Figure~\ref{fig:cayley}. It is spanned by $(1,0,0)$, $(1,2,0)$ and $(1,0,3)$. There is a canonical resolution spanned by the additional rays $(1,1,1)$, $(1,1,0)$, $(1,0,1)$ and $(1,0,2)$, cf.\ right-hand picture in Figure~\ref{fig:cayley}. For the other intersection points one has to refine the Cayley cone similarly.

\begin{figure}
  \centering    
  \cayleyCone \hspace*{4ex} \cayleyResolution
  \caption{Cayley cone $C(\pDiv_{Z_1},\pDiv_{E_2})$ and its refinement.}
  \label{fig:cayley}
\end{figure}
\end{example} 

Let us now turn to the case of complexity one. Here,  $(Y,\supp \pDiv)$ is already smooth. Hence, we only have to consider $\ttoric(\pDiv)$. It has only toric singularities which are given by the cones $\sigma_P = \langle \pDiv_P \times \{1\} \rangle \subset  N \times \QQ$ and can be resolved by toric methods.

\begin{example}
On $Y=\PP^1$ we consider the p-divisor
 \[\pDiv=[-\nicefrac{1}{2},\infty)\otimes 0 + [\nicefrac{1}{3},\infty)\otimes \infty 
+ [\nicefrac{1}{5},\infty)\otimes 1.\]
This is exactly the surface from Example~\ref{exmp-varchenko}. The corresponding cones $\sigma_0$, $\sigma_\infty$, and $\sigma_1$ and their toric resolutions are pictured in Figure~\ref{fig:toroidalRes}.

\begin{figure}
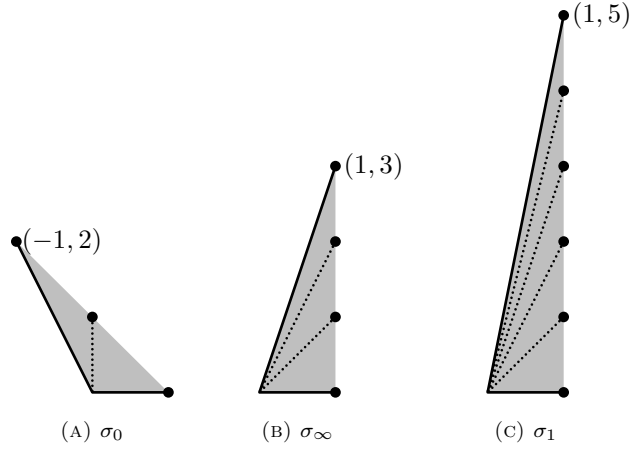

  \centering
  \subfloat[$\sigma_0$]{\toroidalResI}
  \subfloat[$\sigma_\infty$]{\toroidalResII}
  \subfloat[$\sigma_1$]{\toroidalResIII}
  \caption{The cones $\sigma_y$ and their refinements.}
\label{fig:toroidalRes}
\end{figure}
Note that the resolution obtained here is exactly the resolution from Example~\ref{exmp-varchenko}. Moreover, the cones $\sigma_0$, $\sigma_\infty$, and $\sigma_1$  can be identified with the facets $\langle e_i, \rho \rangle$ in the fan $\Sigma_h$. Now, the subdivisions of these cones can be seen as induced by the subdivision of $\Sigma_h$ we have considered.

This observation is not a coincidence, but a  result of the toric embedding coming from the Cox ring representation in Theorem~\ref{thm:coxCplxOneExplicit}. Indeed, in the case of hypersurfaces singularities with complexity one torus action these two resolution strategies are closely related. 
\end{example}


\section{(Log-)terminality and rationality of singularities}

\subsection{Canonical divisors and discrepancies in the toric case}
Recall from (\ref{subsec:canonicalDiv}) that a canonical divisor on a toric variety $X=\toric(\sigma)$ is given by $K_X=-\sum_\rho D_\rho$. Since every line bundle trivializes on an affine toric variety, $X$ is $\QQ$-Gorenstein of index $\ell$ exactly when there is an element $w \in \frac{1}{\ell}M$ such that $\langle w, \rho \rangle=1$ for all rays $\rho \in \sigma(1)$ and $\ell \in \NN$ is minimal with this property. Hence, $\sigma$ is the cone over a lattice polytope in the hyperplane $[\langle w, \cdot \rangle = l]$ with primitive vertices.

Let $\Sigma$ denote a triangulation of $\sigma$ which comes from a resolution of singularities. A ray $\rho \in \Sigma{(1)} \setminus \sigma(1)$ then corresponds to an exceptional divisor with respect to the map $\toric(\Sigma) \rightarrow \toric(\sigma)$ and its discrepancy is given by $\text{discr}_\rho=\langle w, \rho \rangle - 1$. Obviously, this value is always greater than $-1$ which implies log-terminality of toric varieties. Note that $X$ is canonical if and only if 
\[
\text{trunc}(\sigma) := \sigma \cap \{\langle w, \cdot \rangle < 1\}
\]
contains no lattice points apart from the origin. Furthermore, $X$ is terminal if and only if $\overline{\text{trunc}}(\sigma)= \sigma \cap [\langle w, \cdot \rangle \leq 1]$ contains only the cone generators as lattice points.

\begin{figure}
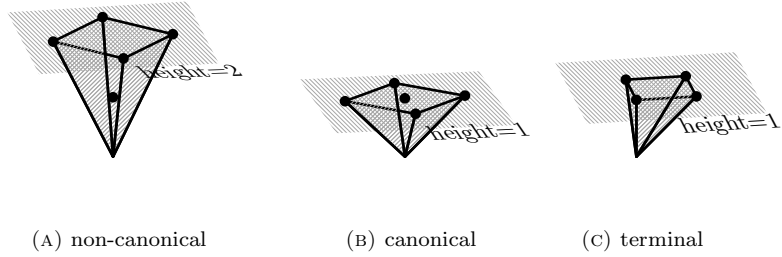

  \centering
  \subfloat[non-canonical]{\noncanonical}
  \subfloat[canonical]{\canonical}
  \subfloat[terminal]{\terminal}
  \caption{Different types of toric singularities.}
\label{fig:singtypes}
\end{figure}

\subsection{Rationality and log-terminality for p-divisors}
Since toric singularities are log-terminal and thus rational it is sufficient to study a toroidal resolution of $X$, e.g. the contraction map $\ttoric(\varphi^*\pDiv) \rightarrow \toric(\varphi^*\pDiv)$, for checking the rationality or log-terminality of a variety.  This observation leads to the results of \cite{tsing}, which are summarized in this section.

\begin{theorem}[{\cite[Thm. 3.4]{tsing}}]
  $\toric(\pDiv)$ has rational singularities if and only if $H^i(\Loc \pDiv, \CO(\pDiv(u)))=0$ holds for all $u \in M\cap\tail(\pDiv)^\vee$ and all $i>0$.
\end{theorem}

\begin{corollary}
  If $\toric(\pDiv)$ has complexity one, i.e. $Y$ is a curve, then $\toric(\pDiv)$ has rational singularities if and only if 
  \begin{enumerate}
  \item $\Loc \pDiv$ is affine, or
  \item $Y=\PP^1$ and  $\deg \lfloor \pDiv(u) \rfloor \geq -1$ for all $u \in (\tail(\pDiv))\dual \cap M$.
  \end{enumerate}
\end{corollary}

To obtain a nice characterization of log-terminality we need to restrict ourselves to sufficiently simple torus actions. A torus action on an affine variety is called a \emph{(very) good} one,  if there is a unique fixed point, which is contained in the closures of all orbits (and the GIT-chamber structure of $X$ is trivial). Such a situation corresponds to a p-divisor $\pDiv$ having a tailcone of maximal dimension, such that the locus of $\pDiv$ is projective (and $\pDiv(u)$ is ample for all $u \in \relint \sigma\dual$). 

Fix a canonical divisor $K_Y = \sum_Z c_Z \cdot Z$ on $Y$. Recall from (\ref{subsec:canonicalDiv}) that a $T$-invariant canonical divisor on $\ttoric(\pDiv)$ and
$\toric(\pDiv)$, respectively, is given by
\begin{align}
  \label{eq:KX}
   K &=  \sum_{(Z,v)} (\mu(v)(1+c_Z)-1)D_{(Z,v)} - \sum_{\rho}E_\rho \nonumber
 \end{align}
where one is again supposed, if focused on $\toric(\pDiv)$,
to omit all prime divisors being contracted via $p:\ttoric(\pDiv)\to\toric(\pDiv)$.

\begin{theorem}[\cite{tidiv}, Corollary 3.15]
 An affine $T$-variety $\toric(\pDiv)$ with a good torus action is $\QQ$-factorial if and only if
  \[
  \rank \cl(Y) + \sum_Z (\#\pDiv_Z^{(0)} -1) + \#\pDiv_0^{(1)} = \dim N.
  \]
  In particular, $Y$ has a finitely generated class group.
\end{theorem}
\begin{remark}
  Note, that the left hand side of the equation in the theorem is always at least $\dim N$ due to the properness condition for $\pDiv$.
\end{remark}

We now turn to criteria for (log-)terminality. Due to the toroidal resolution from above one obtains the following result which generalizes the fact of log-terminality of toric varieties.

\begin{proposition}[{\cite[Cor. 4.10]{tsing}}]
  Every $\QQ$-factorial variety $X$ with torus action of complexity strictly smaller than $\codim (\Sing_X)-1$ is log-terminal.
\end{proposition}

For a p-divisor we introduce a \emph{boundary divisor} on its locus $\loc(Y)$. Let 
$\mu_Z$ be the maximal multiplicity $\mu(v)$ of all vertices $v \in \pDiv_Z$. We define
\[B=\sum_Z \frac{\mu_Z-1}{\mu_Z} \cdot Z.\]

\begin{theorem}[{\cite[Thm. 4.7]{tsing}}]
  A $\QQ$-factorial affine variety $\toric(\pDiv)$ with a very good torus action is log-terminal if and only if  the pair $(\loc(Y),B)$ is log-terminal and Fano.
\end{theorem}

\begin{remark}
  Here, for simplicity we restricted ourselves to the case of $\QQ$-factorial singularities.  In \cite{tsing} the more general case of $\QQ$-Gorenstein singularities is considered. 
\end{remark}

\section{Polarizations}
\label{sec:polar}

\subsection{Toric varieties built from polyhedra}
\label{subsec:projTor}
Let $\nabla$ be a lattice polyhedron in $M_\QQ$, i.e. its
vertices (we assume that there is at least one)
are contained in $M$. Then we can choose a finite
subset $F\subseteq \nabla\cap M$ such that 
$F+(\tail(\nabla) \cap M)=\nabla\cap M$.
Moreover, let $E\subseteq \tail(\nabla) \cap M$ be the Hilbert basis,
cf.\ (\ref{subsec:affToric}).
Then, if $y_e$ ($e\in E$) and $z_f$ ($f\in F$)
denote the affine and projective coordinates, respectively,
we define
$\ptoric(\nabla)\subseteq\CC^{\#E}\times\PP^{\#F-1}_\CC$ as the zero set of the
equations 
$$
\textstyle
\prod_i y_{e_i}\prod_j z_{f_j} = \prod_k y_{e_k}\prod_l z_{f_l}
$$
associated to relations
$$
\textstyle
\sum_i (e_i,0) + \sum_j (f_j,1)= \sum_k (e_k,0) + \sum_l (f_l,1)
$$
inside $M\oplus\ZZ$. If $\nabla=\sigma\dual$ from
(\ref{subsec:affToric}), then this yields the affine variety $\ptoric(\sigma\dual)=\toric(\sigma)$.
The easiest compact examples are
$\nabla=$ $[0,1]\subseteq\QQ$, 
$[0,1]^2\subseteq\QQ^2$, and $\conv\big((0,0), (1,0), (0,1)\big)\subseteq\QQ^2$,
yielding $\PP^1$, $\,\PP^1\times\PP^1$, and $\PP^2$, respectively.
A nice mixed example $\nabla$ arises from adding $\QQ^2_{\geq 0}$ to the
line segment $\ko{(1,0),(0,1)}$, see Figure \ref{fig:polytopeTV}. With $E=\{(1,0),(0,1)\}$ and $F=\{(1,0),(0,1)\}$
we obtain $\ptoric(\nabla)\subseteq\CC^2\times\PP^1$ cut out by the single equation
$y_{(1,0)}z_{(0,1)}=y_{(0,1)}z_{(1,0)}$. That is,
$\ptoric(\nabla)$ equals $\CC^2$ blown up in the origin.

\begin{figure}
\centering
\subfloat[$\nabla$]{\polytopeTV}
\subfloat[$\normal(\nabla)$]{\normalPolytopeTV}
\caption{From $\nabla \subset M_\QQ$ to the toric variety $\TV(\normal(\nabla))$, cf.\ (\ref{subsec:projTor}).}
\label{fig:polytopeTV}
\end{figure}

\subsection{Maps induced from basepoint free divisors}
\label{subsec:bpfreeMaps}
We will present the relation between the toric varieties
$\toric(\Sigma)$ from (\ref{subsec:fansToric}) and those rather
simple ones defined in (\ref{subsec:projTor}). 
Let $\nabla$ be again a lattice polyhedron in $M_\QQ$
and denote by $\Sigma := \normal(\nabla)$ its inner normal fan,
cf.\ (\ref{subsec:dissToric}). Then, $\sigma := |\Sigma|$ and
$\tail(\nabla)$ are mutually dual cones and there is
a structure morphism $\toric(\Sigma)\to\toric(\sigma)$.
By (\ref{subsec:toricPrinc}),
the polyhedron $\nabla$ corresponds to a basepoint free divisor
which is globally generated by $F\subseteq\nabla\cap M$.
If $\nabla\cap M$ generates the abelian group $M$,
then the associated morphism
$$
\Phi_\nabla:\toric(\Sigma)\to \ptoric(\nabla) 
\hspace{1em}\mbox{over}\hspace{0.7em}
\toric(\sigma)\subseteq\CC^E
$$
is the normalization map. If $\nabla$ is replaced by
a sufficiently large scalar multiple $\nabla := N\cdot\nabla$
(for each vertex $u\in\nabla$, the set $\nabla-u$ has to contain the Hilbert basis of the cone generated by it),
then it becomes an isomorphism. Thus, $\nabla$ is even an ample divisor.
Moreover, if $\Gamma\leq\nabla$ is a face, then we can use the 
face-cone correspondence from (\ref{subsec:dissToric}) to define
$\ko{\orb}(\Gamma) := \ko{\orb}(\normal(\Gamma,\nabla))\subseteq \toric(\Sigma)$.
Note that $\ko{\orb}(\Gamma)=\ptoric(\Gamma)\subseteq\ptoric(\nabla)$ 
if $\Phi_\nabla$ is an isomorphism.


\subsection{Coordinate changes}
\label{subsec:coordCh}
For projective toric varieties $\PP(\nabla)=\toric(\Sigma)$
with $\Sigma=\normal(\nabla)$,
there are three sorts of coordinates hanging around.
First, for every $u\in M$, the element $\chi^u\in\CC[M]$
is a rational function on $\toric(\Sigma)$.
Second, in (\ref{subsec:toricCox}) we have seen that
$\toric(\Sigma)=\big(\CC^{\Sigma(1)}\setminus V(B(\Sigma)\big))/G$
with $G=\ker\big((\CC^*)^{\Sigma(1)}\to T\big)$.
This leads to the ``Cox coordinates'' $y_a$ for $a\in\Sigma(1)$. Finally, in (\ref{subsec:projTor}) and (\ref{subsec:bpfreeMaps}) we have seen the projective
coordinates $z_u$ for $u\in\nabla\cap M$. What does the relation between these
coordinates $\chi^u$, $y_a$, and $z_u$ look like?
\\[0.5ex]
First, if $u,u'\in \nabla\cap M$, then $z_u/z_{u'}=\chi^{u-u'}$. In particular,
if $u'$ is the vertex corresponding to $\sigma\in \Sigma$,
i.e.\ if $\sigma=\normal(u',\nabla)$,
then $\chi^{u-u'}\in\CC[\sigma\dual\cap M]$ is a regular function
on the open subset $\toric(\sigma)\subseteq\toric(\Sigma)$.
If $u\in\nabla\cap M$, then the associated effective divisor
$\nabla(u)\in |\CO_{\toric(\nabla)}(\nabla)|$
equals $\nabla(u)=\div(\chi^u)+\div(\nabla)=
\sum_{a\in\Sigma(1)} \big(\langle a,u\rangle - 
\min \langle a,\nabla\rangle\big)$. Hence,
the two sets of homogeneous coordinates compare via
$$
z_u=\prod_{a\in\Sigma(1)} y_a ^{\langle a,u\rangle - 
\min \langle a,\nabla\rangle}.
$$
Note that shifting $\nabla$ without moving $u$ accordingly does
change the exponents. They just measure
the lattice distance of $u$ from the facets of $\nabla$. 
Finally, since the rational functions $\chi^u$ are quotients of $z_\kbb$
coordinates, one obtains that
$$
\chi^u=\prod_{a\in\Sigma(1)} y_a ^{\langle a,u\rangle}.
$$
This product is invariant under the $G$-action.

\subsection{Divisorial polytopes}
\label{subsec:DivPol}
In (\ref{subsec:projTor}) and (\ref{subsec:bpfreeMaps}) we have seen that  polarized toric varieties correspond to lattice polyhedra. This can be generalized to complete complexity-one $T$-varieties if we replace lattice polyhedra with so-called \emph{divisorial polytopes}:

\begin{definition}
	Let $Y$ be a smooth projective curve.
A divisorial polytope $\Psi$ on $Y$ consists of a lattice polytope $\Box\subset M_\QQ$ and a piecewise affine concave function (cf.\ Definition \ref{def:concave}) $$\Psi=\sum \Psi_P\otimes P:\Box\to \DivQ Y,$$ 
such that
\begin{enumerate}
\item For $u$ in the interior of $\Box$, $\deg \Psi(u) > 0$ ;
\item For any vertex $u$ of $\Box$ with $\deg \Psi(u)=0$, $\Psi(u)$ is trivial;
\item For all $P\in Y$, the graph of $\Psi_P$ is integral, i.e. has its vertices in $M \times \ZZ$.
\end{enumerate}
\end{definition}

We construct a polarized $T$-variety from $\Psi$ in a roundabout manner.
For $v\in N_\QQ$ and any point $P\in Y$, set $\Psi_P^*(v)=\min_{u\in \Box} (\langle v,u\rangle -\Psi_P(u))$; this gives us a collection of piecewise affine concave functions on $N_\QQ$. Now let $\pFan_P$ be the polyhedral subdivision of $N_\QQ$ induced by $\Psi_P^*$ and take $\pFan=\sum \pFan_P\otimes P$.  Furthermore, let $\marked$ be the subset of $\tail(\pFan)$ consisting of those $\sigma\in\tail(\pFan)$ such that $(\deg \circ \Psi)|_{F_\sigma} \equiv 0$, where $F_\sigma \prec \Box$ is the face where $\langle \cdot, v \rangle$ takes its minimum for all $v \in \sigma$.

As in remark \ref{rmk:degreeVSmarkings}, the sum of slices $\pFan$ together with markings $\marked$ determines a $T$-variety $X_{\Psi}$.	
	Furthermore, $\Psi^*=(\Psi_P^*)_{P\in Y}$ is a divisorial support function determining an ample Cartier divisor on $X_\Psi$ as in (\ref{subsec:CDivX}) such that $(\Psi^{* *},\Box_{\Psi^*})=(\Psi,\Box)$.

\begin{theorem}[\cite{NathHend}, Theorem 3.2]
The mapping
$$
\Psi\mapsto(X_\Psi,\cO(D_{\Psi^*}))
$$
gives a correspondence between divisorial polytopes and pairs of complete complexity-one $T$-varieties with an invariant ample line bundle.
\end{theorem}

\subsection{Constructing divisorial polytopes}
\label{subsec:constrDivPoly}
Given a polarized toric variety $X$ determined by some lattice polytope $\nabla\subset {\widetilde M}_\QQ$, and given some action of a codimension-one subtorus $T$, one may ask what the corresponding divisorial polytope looks like. A downgrading procedure similar to that of (\ref{subsec:tDown}) yields the answer:
the inclusion of the subtorus $T$ in the big torus once again gives an exact sequence
$$
\xymatrix@R=0.3ex@C=2.5em{
0 \ar[r] & \ZZ \ar[r] & 
\raisebox{0.8ex}{$\tM$} \ar[r]^-{p} & 
M \ar[r] &  0\\
}
$$
where $M$ is the character lattice of $T$.
We again choose a section $s:M\hookrightarrow \tM$ of $p$.
The downgrade of $\nabla$ is then the divisorial polytope $\Psi:p(\nabla) =: \Box\to\DivQ \PP^1$, where for $u\in \Box$,
\begin{align*}
	\Psi_0(u)=\max (p^{-1}(u)-s(u)) \\
\Psi_\infty(u)=\min (p^{-1}(u)-s(u)).
\end{align*}
The corresponding variety $X_\Psi$ is canonically isomorphic to $X$ as a polarized $T$-variety.

\begin{figure}
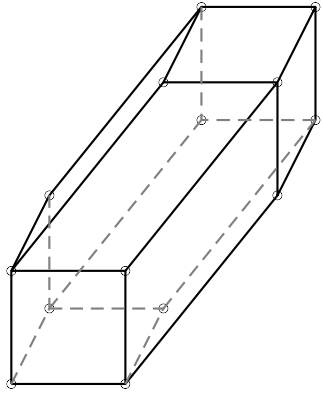
 
\begin{center}
\nablapic
\end{center}
\caption{$\nabla$ for a toric Fano threefold, cf.\ Example \ref{ex:divpoly}.}\label{fig:nablafano}
\end{figure}

\begin{figure}
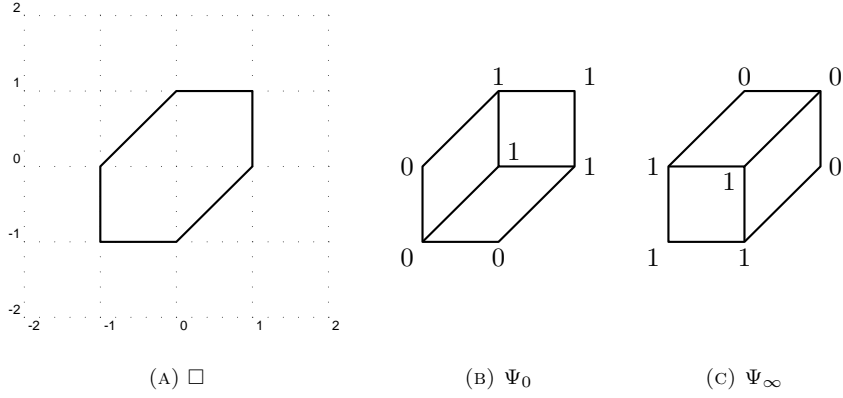

\begin{center}
\subfloat[$\Box$]{\boxpic}
\subfloat[$\Psi_0$]{\psizero}
\subfloat[$\Psi_\infty$]{\psiinfty}
\end{center}
\caption{A downgraded divisorial polytope, cf.\ Example \ref{ex:divpoly}.}\label{fig:boxdivpoly}
\end{figure}

\begin{example}\label{ex:divpoly}
Consider the polytope
\begin{align*}
\nabla=\conv \big\{ (-1,-1,-1),(-1,0,-1),(0,-1,-1),(0,0,-1),\\(-1,-1,0),(-1,0,0),(0,-1,0),(0,1,0)(1,0,0),(1,1,0)\\(0,0,1),(0,1,1),(1,0,1),(1,1,1)\big\}
\end{align*}
pictured in Figure \ref{fig:nablafano}. Then $\PP(\nabla)$ is a toric Fano threefold. We now consider the subtorus action corresponding to the projection $\ZZ^3\to\ZZ^2$ onto the first two factors. Using the above downgrading procedure we get a divisorial polytope $\Psi:\Box\to\Div_\QQ\PP^1$, where $\Box$ is the hexagon pictured in Figure \ref{fig:boxdivpoly}(a), and $\Psi_0$ and $\Psi_\infty$ are the piecewise affine functions with domains of linearity and values as pictured in Figure \ref{fig:boxdivpoly}(b) and (c).
\end{example}

\subsection{Divisorial polytopes and the $\proj$ construction}
\label{subsec:divPolyProj}
Similar to the toric case, we can construct a projective variety directly from $\Psi$:
$$
\ptoric(\Psi) := \proj \left(\Sym \bigoplus_{u\in \Box\cap M} \Gamma\left(Y,\cO(\Psi(u))\right)\chi^u\right).
$$
Although the torus $T$ acts on $\ptoric(\Psi)$, this action will in general not be effective.

Note that the relationship between $\ptoric (\Psi)$ and $X_\Psi$ is similar to that between $\ptoric(\nabla)$ and $\toric(\normal(\nabla))$ for a lattice polytope $\nabla$, see \cite[Section 4]{NathHend}. Indeed, the ample divisor $D_{\Psi^*}$ on $X_\Psi$ determines a dominant rational map $\Phi_\Psi:X_\Psi\dashrightarrow \ptoric (\Psi)$. $\Phi_\Psi$ is a regular morphism if $D_{\Psi^*}$ is globally generated, which is in particular the case if $Y=\PP^1$. Furthermore, assuming that $\Phi_\Psi$ is  regular, $\Phi_\Psi$ is the normalization map if it is birational. 
A sufficient condition for $\Phi_\Psi$ to be birational is that 
$\Psi(u)$ is very ample for some $u\in\Box\cap M$, and the set 
$$
\{u\in\Box\cap M\ |\ \dim \Gamma(Y,\CO(\Psi(u)))>0\}
$$
generates the lattice $M$.

\subsection{The moment map}
\label{moMap}
Consider an action of a torus $T$ on $
\PP^k$ given by weights $u_i\in M$, i.e.
$$
t.(x_0:\ldots:x_k)=(t^{u_0}x_0:\ldots:t^{u_k}x_k).
$$
Then a \emph{moment map} for this action is given by
\begin{align*}
\cm:\PP^n&\to M_\RR\\
(x_0:\ldots:x_n)&\mapsto \frac{1}{\sum |x_i|^2}\cdot \sum |x_i|^2 u_i.
\end{align*}
If $X$ is any $T$-invariant subvariety $\iota:X\hookrightarrow\PP^n$; then a moment map $\mu$ for $X$ is given by the composition $
\mu=\cm\circ\iota$. In particular, if $\iota$ is given by sections $s_0,\ldots,s_k$ of some very ample line bundle $\cL$ on $X$, then
we have
\begin{align*}
\mu:X&\to M_\RR\\
x&\mapsto \frac{1}{\sum |s_i(x)|^2}\cdot \sum |s_i(x)|^2 u_i.
\end{align*}
The image of this map is a rational polytope called the \emph{moment polytope}. By the theorem of Atiyah, Guillemin, and Sternberg, this polytope is the convex hull of the images of the points of $X$ fixed by $T$, see \cite{GuilleminSternberg}.

If $X$ is toric, and $\cL$ corresponds to some lattice polytope $\nabla\subset M_\QQ$, then $\mu(X)=\nabla$, see for example \cite{fulton}, section 4.2. This can be easily seen by considering the images of the $T$-fixed points. Indeed, here the weights $u_i$ are simply the elements of $\nabla \cap M$, and the torus fixed points correspond to vertices of $\nabla$. Setting $\rho_j(x) := |s_j(x)|^2/(\sum |s_i(x)|^2)$, one can show that $\rho_j(x)=0$ unless $x$ corresponds to $u_j$.

On the other hand, if $X=X_\Psi$ for some divisorial polytope $\Psi$ and $\cL=\cO(D_{\Psi^*})$, then $\mu(X)=\Box$.
Indeed, let $p:\widetilde X \to X$ resolve the indeterminacies of the rational quotient map $X\to Y$ to a regular map $\pi:\widetilde{X}\to Y$. Then the fixed points of $\widetilde X$ in the fiber $\pi^{-1}(P)$ correspond exactly to the vertices of the graph of $\Psi_P$. By locally upgrading $\widetilde X$ to a toric variety in a formal neighborhood of $\pi^{-1}(P)$, we get that $\rho_j(x)=0$ for a fixed point $x$ unless it corresponds to a vertex of weight $u_j$, in which case $\rho_j(x)$ might be non-zero. Since the fixed points of $\widetilde{X}$ map surjectively onto the fixed points of $X$, the claim follows. 

We can also easily determine the momentum polytopes of the individual components of each fiber $\pi^{-1}(P)$. A component $Z$ of $\pi^{-1}(P)$ corresponds to a facet $\tau$ of the graph of $\Psi_P$, and the fixed points of $Z$ correspond exactly to the vertices of $\tau$. By arguments similar to above, $\mu(Z)$ is nothing other than the projection of $\tau$ to $\Box$. Thus, the subdivision of $\Box$ induced by the linearity regions of $\Psi_P$ reflects the structure of the irreducible components of $\pi^{-1}(P)$. This is similar to the situation presented in \cite[(1.2)]{kapranov:93a} for Grassmannians, which is addressed in more generality in \cite{hu:05a}. 

\begin{example}
Consider the divisorial polytope $\Psi$ from Example \ref{ex:divpoly}. Then the moment polytope of $X_\Psi$ is the hexagon of Figure \ref{fig:boxdivpoly}(a). Furthermore, we can see the fibers of $X_\Psi\mapsto \PP^1$.  The general fiber is the toric del Pezzo surface of degree six corresponding to the aforementioned hexagon, whereas the fibers over $0$ and $\infty$ each consist of three intersecting copies of $\PP^1\times\PP^1$ as can be seen from the subdivisions in Figures \ref{fig:boxdivpoly}(a) and \ref{fig:boxdivpoly}(b).
\end{example}


\section{Toric intersection theory}
\label{sec:toricInt}

\subsection{Simplicial toric varieties}
\label{subsec:simpToricInt}
Let $X=\toric(\Sigma)$ be a $k$-dimensional, complete
toric variety induced from a simplicial
fan. Thus, $X$ has at most abelian quotient singularities and 
one can define (rational) intersection numbers among
$k$ divisors. This intersection product is uniquely determined by
exploiting the following basic rules:
\begin{enumerate}
\item[(i)]
If, with the notation as in (\ref{subsec:affToric}),
$\sigma=\langle a^1,\ldots,a^k\rangle$ is a full-dimensional cone
of $\Sigma$,
then $(\ko{\orb}\,{a^1}\cdot\ldots\cdot\ko{\orb}\,{a^k})
= 1/\vol(\sigma)$
with
$\vol(\sigma) := \det(a^1,\ldots,a^k)$.
\item[(ii)]
If $a^1,\ldots,a^k$ (not necessarily distinct) are not contained in a common
cone of $\Sigma$, 
then $(\ko{\orb}\,{a^1}\cdot\ldots\cdot\ko{\orb}\,{a^k})= 0$.
\item[(iii)]
The intersection product factors via the relation
$\sum_{a\in\Sigma(1)} a \otimes \ko{\orb}(a)\sim0$
of linear equivalence from (\ref{subsec:toricPrinc}).
\end{enumerate}

\begin{example}
\label{ex:intersectionToricSurfaces}
In the case of a smooth, complete toric surface $\TV(\Sigma)$ with $\Sigma(1) = \{a^1,a^2,\dots\}$, one always has relations of the sort $a^{i-1} - b_i\, a^i + a^{i+1}=0$ between adjacent generators. This leads to the self intersection numbers $(\ko{\orb}(a^i))^2=-b_i$. In particular, if $\til{\CC^2}\to\CC^2$ is the blowup of the origin given by the fan $\Sigma$
with $\Sigma(1)=\{(1,0), (1,1), (0,1)\}$, then, despite the fact that it is not complete, one sees that $(E^2)=-1$ with $E := \ko{\orb}\,(1,1)$.

  Let us assume that $a^0=(-1,0)$, $a^i=(p_i,q_i)$, with $a_i$ being numbered counterclockwise. Now, one can prove by induction that the self-intersection numbers $-b_i$ and the generators $a_i$ are related by the following continued fraction formula
\[
[b_{1},\ldots,b_{i}] := b_{1} - \frac{1}{b_{2}-\frac{1}{\ddots}} = \frac{q_i}{p_i'}, \quad \text{ with } 0 < p_i' < q_i \text{ and } p_i' \equiv p_i \,(\text{mod } q_i).
\]
\end{example}

\subsection{Intersecting with Cartier divisors}
\label{subsec:toricIntCartier}
For general, i.e.\
non-simplicial toric varieties, intersection numbers as in
(\ref{subsec:simpToricInt}) do not make 
sense. However, one can always intersect cycles with Cartier divisors.
\\[1ex]
Let $\nabla\subseteq M_\QQ$ be a lattice polyhedron and
$\Sigma := \normal(\nabla)$. Then, every face $\Gamma\leq\nabla$ gives rise
to two different geometric objects. First,
in (\ref{subsec:toricPrinc}) we have constructed the Cartier divisor
$\div(\Gamma)$ on the toric variety $\toric(\normal(\Gamma))$.
Afterwards, in (\ref{subsec:bpfreeMaps}), we have defined the cycle
$\ko{\orb}(\Gamma)\subset\toric(\Sigma)$.
Note that, in accordance with (\ref{subsec:toricOrbits}),
this cycle equals $\toric(\normal(\Gamma))$.
Now, the fundamental but trivial observation in intersecting $T$-invariant
cycles with Cartier divisors is
$$
\div(\nabla)\cdot \ko{\orb}(\Gamma) = \div(\Gamma).
$$

\begin{proof}
According to (\ref{subsec:toricOrbits}), this formula is obtained by noting first that the vertices of $\nabla$ corresponding
to cones $\sigma\geq \tau := \normal(\Gamma,\nabla)$ are exactly those from $\Gamma$. Then, one restricts the corresponding monomials by shifting
$\Gamma$ into the abelian subgroup $M_\Gamma := \Gamma-\Gamma$.
This shift makes it possible to understand the above formula as
a relation between divisor \emph{classes}.
\end{proof}

The projective way of expressing the formula for $\div(\nabla)$ from the end of (\ref{subsec:CDtoric}) is
$$
\div(\nabla) = \sum_\Gamma\dist(0,\Gamma)\cdot \ko{\orb}(\Gamma)
$$
where $\Gamma$ runs through the codimension-one faces (``facets'')
of $\nabla$ and $\dist(0,\Gamma)$ denotes the oriented
(positive, if $0\in\nabla$)
lattice distance of the origin 
to the affine hyperplane containing $\Gamma$. 

\begin{example}
\label{ex:blowingUp}
Let $\nabla$ be the polyhedron from (\ref{subsec:projTor}) describing
the blowup $\til{\CC^2}\to\CC^2$ of the origin.
Denoting the unique compact edge by $e$, we have that $\div(\nabla) = - \ko{\orb}(e) = -E$ since $0\notin\nabla$. Thus,
$(E^2) = -\div(\nabla)\cdot \ko{\orb}(e) = \div(e)$ and the latter
describes the ample divisor class of degree one, i.e. a point,  on $\PP(e)=\PP^1$.
\end{example}

\begin{corollary}
\label{cor:selfIntVol}
$(\div(\nabla)^k)=\vol(\nabla)\cdot k!$.
\end{corollary}

\begin{proof}
We proceed by induction.
$$
\renewcommand{\arraystretch}{1.3}
\begin{array}{rcl}
(\div(\nabla)^k)
& = &
(\div(\nabla)^{k-1})\cdot \sum_\Gamma\dist(0,\Gamma)\cdot \ko{\orb}(\Gamma)
\\
& = &
\sum_\Gamma\dist(0,\Gamma)\cdot (\div(\Gamma)^{k-1})
\\
& = &
\sum_\Gamma\dist(0,\Gamma)\cdot \vol(\Gamma)\cdot (k-1)!
\; = \;
k\cdot \vol(\nabla) \cdot (k-1)!\;.
\end{array}
\vspace{-3.5ex}
$$
\end{proof}

\subsection{A result for complexity-one $T$-varieties}
\label{subsec:intersecT}

We already know from (\ref{subsec:DivPol}) that polytopes have to be replaced by divisorial polytopes when proceeding from toric varieties to complexity-one $T$-varieties. To do intersection theory on a complete complexity-one $T$-variety $\TV(\fan)$ we thus have to come up with a natural extension of the volume function.

\begin{definition}
For a function $h^*:\Box\rightarrow \wdiv_\QQ Y$ we define its \emph{volume} to be
\[\vol h^* := \sum_P \int_{\Box} h^*_P \vol_{M_\RR}\,.\]
We associate a \emph{mixed volume} to functions $h^*_1,\ldots, h^*_k$ by setting
\[V(h^*_1,\ldots, h^*_k) := \sum_{i=1}^k (-1)^{i-1} \hspace*{-2ex} \sum_{1\leq j_1 \leq \ldots j_i \leq  k} \hspace*{-2ex} \vol(h^*_{j_1} + \cdots + h^*_{j_i})\,.\]
\end{definition}

\begin{proposition}[\cite{tidiv}, Proposition 3.31]
\label{prop:intersectionNumbers}
Let $\pFan$ be a complete f-divisor on the curve $Y$ and let $D_h$ be a semiample Cartier divisor on $\TV(\pFan)$. Setting $n := \dim \TV(\pFan)$, the top self-intersection number of $D_h$ is given by
\[(D_h)^n = n!\vol h^*.\]
Moreover, assuming that $h_1, \ldots, h_n$ define semiample Cartier divisors $D_{h_i}$, we have that 
\[(D_{h_1} \cdots D_{h_n}) = n!\,V(h^*_1, \ldots, h^*_n).\]
\end{proposition}

This assertion follows from the fact that $(D_h)^n$ is equal to $\lim_{k \to \infty} \frac{n!}{k^n} \dim \Gamma(kD_h)$ and the explicit description of the global sections from (\ref{subsec:globalSections}). Note that this statement is a special instance of \cite[Theorem 8]{timDivs}.

\subsection{Intersection graphs for smooth $\CC^*$-surfaces}
Let $\pFan=\sum_{P} \pFan_P\otimes P$
together with $\sdeg=\emptyset$ be a complete f-divisor in the lattice $N=\ZZ$ on a curve $Y$ of genus $g$. Let $\{P_1, \ldots, P_r\}$ be the support of $\pFan$ and $\pFan_i:=\pFan_{P_i}$. In this situation we also write $\pFan=\sum_{P} \pFan_P\otimes P = \sum_{i=1}^r \pFan_i \otimes P_i$.

The slices $\pFan_i$ are complete subdivisions of $\QQ$ into bounded and half-bounded intervals. This data gives rise to a $\CC^*$-surface $X=\toric(\pFan)$. Here, we only consider the case of smooth surfaces (for a criterion on smoothness see (\ref{sec:toroidal-resolutions})). Since we assume that $\sdeg = \emptyset$, $\ttoric(\pFan)=\toric(\pFan)$. By the results on invariant prime divisors from (\ref{subsec:WDivX}), we will find two horizontal prime divisors on $X$, one for each ray in the tail fan $\{\QQ_{\leq 0},0 , \QQ_{\geq 0}\}$. In the case of a $\CC^*$-surface a horizontal prime divisor is simply a curve consisting of fixed points. Moreover, these curves turn out to be isomorphic to $Y$; we will denote them by $F_+$ and $F_-$. In addition to these curves we get vertical prime divisors corresponding to boundary points $\frac{p_{ij}}{q_{ij}} \in \pFan_{i}$, $i=1,\ldots,r$ and $j=1,\ldots,n_i$. Here we assume, that  $v_{i(j+1)}:=\frac{p_{i(j+1)}}{q_{i(j+1)}}> \frac{p_{ij}}{q_{ij}}=:v_{ij}$. These curves are closures of maximal orbits of the $\CC^*$-action, and we will denote them by $D_{ij}$. 

Now, similar to the observations in \cite{tsing}, it turns out that in a neighbourhood of $D_{i1}, \ldots, D_{ir_i}$, $X$ is locally formally isomorphic to the toric surface spanned by the rays $(-1,0),(p_{i1},q_{i1}), \ldots, (p_{in_i},q_{in_i}), (1,0)$ in a neighbourhood of the invariant prime divisors. In particular, the mutual  intersection behaviour of the $ D_{i1}, \ldots, D_{in_i}$ is exactly the same as that of the invariant divisors of the toric surface.
Hence, two curves $D_{ij}$ and  $D_{lm}$ intersect (transversely) exactly when $i=l$ and $|m-j|=1$. A curve $D_{ij}$ intersects $F^+$ if and only if  $\frac{p_{ij}}{q_{ij}}$ is a maximal boundary point and  it intersects $F_-$ if and only if it is a minimal boundary point, i.e. $j=1$. In addition, we obtain the self-intersection number
\[c_-=(F_-)^2=((F_-+\div(\chi^{1}))\cdot F_-)=((\sum_i v_{i1}D_{i1})\cdot F_-) = \sum_i v_{i1}\]
and similarly 
\[c_+=(F_+)^2=-\sum_i v_{in_i}.\]
 By the considerations in Example~\ref{ex:intersectionToricSurfaces}, the self-intersection  numbers $b_{j}^i=-(D_{ij})^2$ can be related to the boundary points $v_{ij}=\frac{p_{ij}}{q_{ij}}$ using the formulae
\[
q_{i1}=1,\quad [b_{1}^i,\ldots,b_{(j-1)}^i] := b_{1}^i - \frac{1}{b_{2}^i-\frac{1}{\ddots}} = \frac{q_{ij}}{p_{ij}'}.
\]
Hence, we obtain an intersection graph of the following form
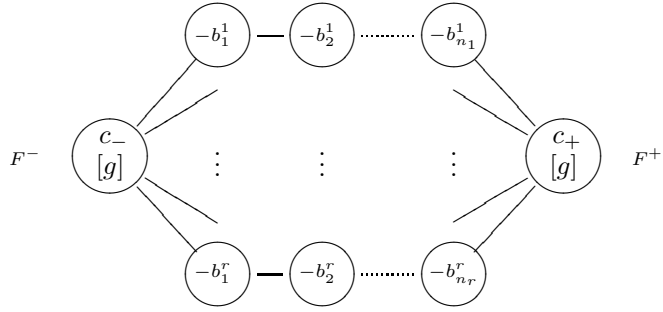
\[
\entrymodifiers={++[o][F-]}
\def\objectstyle{\scriptstyle}
\xymatrix@-1.2pc{
*{}         
& *{}
& {-b^1_{1_{ \ }}} \ar@{-}[r]  
& {-b^1_{2_{ \ }}} \ar@{.}[rr] 
& *{}
& {-b^1_{n_1}}  
& *{}
& *{} 
\\
*{}
& *{}
& *{}
& *{}
& *{}
& *{}
& *{}
& *{}
\\
*{F^-}
&
{  \txt<2em>{$\;c_-$\\$[g]$} } 
\ar@{-}[uur] \ar@{-}[ur] \ar@{-}[ddr] \ar@{-}[dr] 
& *{\vdots}   
&*{\vdots}  
& *{}
& *{\vdots} 
& {\txt<2em>{$\;c_+$\\$[g]$}} 
\ar@{-}[ul] \ar@{-}[uul] \ar@{-}[ddl] \ar@{-}[dl] 
& *{F^+}
\\
*{}
& *{}
& *{}
& *{}
& *{}
& *{}
& *{}
& *{} 
\\
*{}
& *{}
& {-b_{1_{ \ }}^r} \ar@{-}[r]  
& {-b_{2_{ \ }}^r} \ar@{.}[rr] 
& *{}
& {-b_{n_r}^r} 
& *{}
& *{} 
\\
  }
\]
By Orlik and Wagreich \cite{0352.14016}, these graphs correspond to deformation classes of $\CC^*$-surfaces with two curves of consisting of fixed points. We indeed obtain families of surfaces having these intersection graphs by varying $Y$ together with the point configuration $P_1, \ldots, P_r$ and by adding algebraic families of degree zero divisors to $\pFan$.

We can also reverse the procedure to obtain an f-divisor out of an intersection graph. Choose any integral $v_{i1}=p_{i1}$ satisfying  $\sum_i v_{i1} = c_-$. The remaining $v_{ij}$ are then determined by the $b_j^i$ via the above continued fractions. Choosing any genus $g$ curve $Y$ and pairwise different points $P_1,\ldots,P_r\in Y$ then determines an f-divisor. The open choices of the $Y$, of $P_1, \ldots, P_r \in Y$, and of particular $v_{i1}$ reflect the fact that the intersection graph only determines the topological type of $X$.

\begin{remark}
  The Neron-Severi group of a $\CC^*$-surface is generated by the the curves $F^\pm,D_{ij}, D_0$. Here the curve $D_0$ corresponds to the vertex of a trivial coefficient of $\pFan$ and its intersection behaviour is given by $D_0^2=0$, $D_0\cdot F^\pm =1$ and $D_0 \cdot D_{ij} =0$.  The matrix of the intersection form is encoded by our graph and by the intersection behaviour of $D_0$. The relations between the generators are given as the kernel of this matrix. Now Section~\ref{sec:furtherDivisors} provides a representative of the canonical class 
\[K=-F^+ - F^- -\sum_{ij} D_{ij} + (r+2g-2)\cdot D_0.\]
\end{remark}

\begin{remark}
 In order to obtain arbitrary  $\CC^*$-surfaces, Orlik and Wagreich considered a so-called  canonical equivariant resolution of the surface. This resolution is always of the above type. In our notation this is the same as considering the minimal resolution of $\ttoric(\pFan)$ for an arbitrary  $\CC^*$-surface given as $\toric(\pFan)$. 

Recall from the downgrade procedure in Section~\ref{subsec:tDown} that  $\toric(\pFan)$ is toric if $Y=\PP^1$  and the support of $\pFan$ consists of
at most two points (i.e. $(Y,\supp(\pDiv))$ is a toric pair). By the above, such an f-divisor corresponds to a circular intersection graph with $g=0$. 
\end{remark}

\section{Deformations}
\label{sec:Deformations}

\subsection{Deformations of affine toric varieties}
\label{subsec:homogToricDef}

\newcommand{\defQ}{\sigma_r}

An \emph{equivariant deformation} of a $T$-variety $X$ is a
$T$-equivariant pullback diagram
$$
\xymatrix@R=0.7pc@C=0.7pc{
     X \ar[r] \ar[d] & \cX \ar[d]^\pi \\
     0 \ar[r] & S
}
$$
where $\pi$ is a flat map. The torus $T$ acts on $X$, $\cX$, and $S$, and the 
$T$-action on the total space $\cX$  induces the given action on the 
special fiber $X = \pi^{-1}(0)$.

A deformation of $X$ over $\CC[\keps]/(\keps^2)$ such that
$T$ acts on $\keps$ with  weight $r \in M$ is called \emph{homogeneous}
of degree $-r$. This induces an $M$-grading on the module
$T^1_X$ of first order deformations.
We fix now some primitive degree $r \in M$,  along with a cosection
$s^*\colon N \to N_r := N \cap r^\perp$. Non-primitive degrees
can also be handled at the cost of more complex notation.

Consider now an affine toric variety $X = \TV(\sigma)$.
Using the toric description of $\Omega_X$~\cite{danilov},
$\sigma$ determines a complex that allows one
to compute  $T^1_X(-r)$ and, if $X$ is non-singular in 
codimension 2, $T^2_X(-r)$~\cite{tdef_T1,tdef_T2}.

$T^1_X(-r)$ may also be described as a vector space of
Minkowski summands of  the polyhedron
$\defQ := s^*(\sigma \cap \{r = 1\})$:
the set of scalar multiples of Minkowski summands of a 
polyhedron $\Delta$ forms
a cone $C(\Delta)$, with Grothendieck group $V(\Delta) := C(\Delta)-C(\Delta)$.
For example, if $\Delta$ is a parallelogram, $C(\Delta)$ is 
two-dimensional, spanned by the edges of $\Delta$. Then $T^1_X(-r)$
may be described by augmenting
$V(\defQ)$ with information about possible non-lattice vertices of 
$\defQ$~\cite[Theorem 2.5]{tdef_flip}.

\begin{definition}
  A \emph{Minkowski decomposition}
  of a polyhedron $\Delta$ is a Minkowski sum $\Delta = \sum \Delta_i$
  of polyhedra $\Delta_i$ with common tailcone. It is
  \emph{admissible} if for each face of $\Delta$, at most one
  of the corresponding faces $\face(\Delta_i,u)$ does not
  contain lattice points.
\end{definition}

For degrees $r \in \sigma\dual$,
admissible decompositions $\defQ = \Delta_0 + \dotsm + \Delta_l$
allow one to construct \emph{toric deformations}, where
$\cX$ is a toric variety and the embedding
$X \hookrightarrow \cX$ is a morphism of toric varieties.
Similarly to the Cayley cone in (\ref{sec:toroidal-resolutions}),
we define the cone $\wt{\sigma}$ in $\wt{N} := N_r \oplus \Z^{l+1}$ 
generated by $\Delta_i \times \{e_i\}$, $0 \le i \le l$.
Taking $\cX=\TV(\wt{\sigma})$, the binomials $\chi^{[0,e^i]} - \chi^{[0,e^0]}$ define a map
$\pi\colon \cX  \to \CC^l$.

\begin{theorem}[{\cite[Theorem 3.2]{tdef_flip}}]
\
\begin{enumerate}
\item $\cX \to \CC^l$ is a toric deformation of $X$.
\item The corresponding Kodaira-Spencer map $\CC^l \to T^1_X(-r)$ maps $e_i$ to the class of the Minkowski summand $\Delta_i \in C(\sigma_r) \subset V(\defQ)$.
\end{enumerate}
\end{theorem}

\begin{example}\label{ex:defMinus4}
We consider deformations of the cone over the rational normal curve
from Example~\ref{ex:minus4} with $r = [0,1]$. The non-trivial Minkowski decompositions of $\defQ = \conv\{(-\frac{1}{2},1), (\frac{1}{2},1)\}$
correspond to the decompositions of the interval
$[-\frac{1}{2},\frac{1}{2}] = [-\frac{1}{2},0] + [0,\frac{1}{2}]
                            = \{-\frac{1}{2}\} + [0, 1]$.
For the first decomposition, we get the cones
(generated by the columns of)
\begin{align*}
\wt{\sigma} &= 
  \begin{pmatrix}
        -\frac{1}{2} & 0 & 0 & \frac{1}{2} \\
        1 & 1 & 0 & 0 \\
        0 & 0 & 1 & 1
  \end{pmatrix} &
\wt{\sigma}\dual &=
  \begin{pmatrix}-2 & 0 & 0 & 2 \\
         0 & 0 & 1 & 1 \\
         1 & 1 & 0 & 0
  \end{pmatrix},
\end{align*}
with Hilbert basis $E = \{[e,0,1] \mid -2 \le e \le 0\}
                    \cup \{[e,1,0] \mid  0 \le e \le 2\}$.
The equations for $\TV(\wt{\sigma})$ and those corresponding
to the second decomposition are
\begin{align*}
\rank \pmat{y_0 & y_1 & y_2' & y_3 \\ y_1 & y_2 & y_3 & y_4} &\le 1 &
\rank \pmat{y_0 & y_1 & y_2 \\ y_1 & y_2' & y_3 \\ y_2 & y_3 & y_4} &\le 1.
\end{align*}
yielding two one-parameter deformations with deformation parameters
$s = y_2 - y_2'$ of degree $r$. These two one-parameter deformations generate $T^1(-r)$. This is in fact Pinkham's famous example of a singularity whose versal base space consists of two irreducible components; we see curves from both components via the above toric deformations.
\end{example}

\begin{figure}
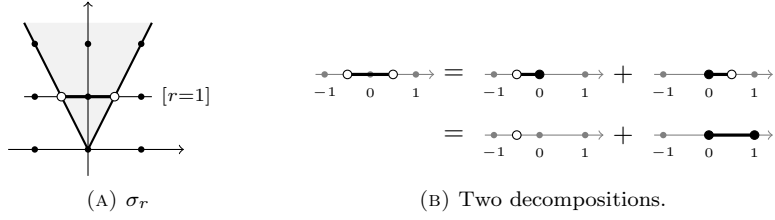

\centering
\subfloat[$\sigma_r$]{\minusVierQR}\hspace*{5ex}
\subfloat[Two decompositions.]{\minusVierDecomp}
\caption{Toric deformations of the cone over the rational normal curve.}
\label{fig:defMinusVier}
\end{figure}

This construction may be generalized to degrees $r \not\in \sigma\dual$
by deforming the toric variety associated with
$\tau = \sigma \cap \{r \ge 0\}$~\cite{tdef_flip}, but the total spaces
are no longer toric. The resulting families arise naturally in
the language of polyhedral divisors as explained
in (\ref{subsec:equivDeformLocal}), or
by Mavlyutov's approach using Cox rings~\cite{mavlyutov_cox}.

It is even possible to describe the homogeneous parts
of the versal deformation if $X$ is non-singular in codimension 2.
The \emph{universal Minkowski summand} $\wt{C}(\defQ)$ is a cone lying over $C(\defQ)$.
The associated map of toric varieties $\TV(\wt{C}(\defQ)) \to 
\TV(C(\defQ))$
allows one to obtain a family $\cX \to \ovl{\cM}$, where the scheme
$\ovl{\cM}$ is determined by $C(\defQ) \subset V(\defQ)$.

\begin{theorem}[\cite{tdef_versal},\cite{tdef_versal2}]
$\cX \to \ovl{\cM}$ is the versal deformation of $X$ in degree $-r$.
If $X$ is an isolated Gorenstein singularity, then $T^1$ is 
concentrated in the Gorenstein degree $r_0$, so $\cX$ is the versal deformation.
\end{theorem}

\subsection{Equivariant deformations of affine $T$-varieties}
\label{subsec:equivDeformLocal}

Suppose now that $X$ is some affine $T$-variety. By restricting the torus 
action to the subtorus $T_r := \ker (r\colon T \to \CC^*)$, the study of 
deformations in degree $r$ may be reduced to the study of invariant 
deformations of the $T_r$-variety $X$.

Thus, in the following, we replace $T$ by $T_r$, and consider equivariant 
deformations in degree $r = 0$. In particular, the torus $T$ acts 
trivially on the base, and acts on every fiber. Such families may be 
described by families of polyhedral divisors: the total space $\cX$ over 
$B$ corresponds to a p-divisor $\cE$ on $Z \to B$ that restricts to 
p-divisors $\cE_s = \cE_{|Z_s}$ with $\cX_s = \TV(\cE_s)$.
See Figure~\ref{fig:tvardef} for a simple example.

Note that when two prime divisors on $Z$ restrict to the same prime 
divisor on a fiber $Z_s$, their polyhedral coefficients are added via Minkowski summation.
If we wish to fix some central fiber $X = \TV(\pDiv)$, we may obtain invariant deformations 
of $X$ in two ways. First of all, we may simply move prime divisors on the base of $Y$. Secondly, for more interesting deformations, we may split up some of the polyhedral coefficients into Minkowski sums. 
Requiring 
that $X$ is the fiber $\cX_0$ implies that these Minkowski decompositions 
must be admissible.

Concretely, we say that a deformation of $(Y, \{D_i\})$ is a deformation 
$Z \to B$ of $Y$ together with a collection of prime divisors 
$\{E_{i,j}\}$ such that no $E_{i,j}$ contains fibers of $Z \to B$ and 
such that each $E_{i,j}$ restricts to $D_i$ in $Y$. Given $l_i$-parameter 
Minkowski decompositions $\pDiv_{D_i} = \sum \Delta_i^j$ of the 
coefficients $\pDiv_{D_i}$, we define a polyhedral divisor $\cE$ on $Z$ 
by setting $\cE_{E_{i,j}} = \Delta_i^j$.

\begin{figure}
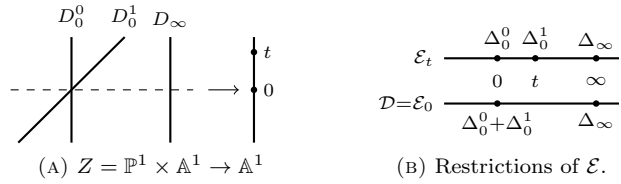

\centering
\subfloat[$Z = \PP^1 \times \AA^1 \to \AA^1$]{\hspace*{2ex}\tvardefBase\hspace*{2ex}}
\hspace*{2ex}
\subfloat[Restrictions of $\cE$.]{\hspace*{2ex}\tvardefDivisors\hspace*{2ex}}
\caption{One-parameter toric deformation as deformation of p-divisors.}
\label{fig:tvardef}
\end{figure}

\begin{theorem}[{\cite[Section 2]{defrat_tvar}}]
If $Y = \PP^1$ and the given Minkowski decompositions are admissible,
then there exists a deformation $(Z, \{E_{i,j}\})$ of $(Y, \{D_i\})$,
where $Z$ is the trivial family over some open subset of $\CC^l$,
$l = \sum l_i$,
such that $\cE$ is a p-divisor, and $\cX = \TV(\cE) \to \CC^l$ is
an invariant deformation of $X$.
The fibers of this family are $T$-varieties
$\TV(\cE_s)$, where $\cE_s$ is the restriction of $\cE$ to the fiber
$Z_s$.
\end{theorem}

\begin{example}
Suppose $X = \TV(\sigma)$ is toric, and consider some degree $r$.
By downgrading the torus action to the torus $T_r = N_r \otimes \CC^*$, 
we obtain the p-divisor $\pDiv = \defQ \otimes \{0\}
+ \sigma_{-r} \otimes  \{\infty\}$ on $\PP^1$.
A Minkowski decomposition of $\defQ$ gives rise to an 
invariant deformation $\TV(\cE)$ of the $T_r$-variety $X$. This 
deformation is equal to the corresponding deformation obtained in 
(\ref{subsec:homogToricDef}), whether $r \in \sigma\dual$ or not. 
The case of a one-parameter deformation is illustrated in
Figure~\ref{fig:tvardef}.
Note that $\sigma_{-r} = \emptyset$ if and only if
$r \not\in \sigma\dual$,
so for $r \in \sigma\dual$, $D_\infty$ has coefficient $\emptyset$,
and $\cE$ is actually defined on $\AA^{l+1}$.

In particular, the deformations of Example~\ref{ex:defMinus4} are
$T_r$-varieties with divisors
$${
\setlength{\arraycolsep}{2pt}
\begin{array}[b]{rcrcrcr}
\cE^1 &=& [-\frac{1}{2}, 0] \otimes D^0 &+& [0, \frac{1}{2}] \otimes D^1 \\
\cE^2 &=& -\frac{1}{2} \otimes D^0 &+& [0,1] \otimes D^1
\end{array}
}$$
on $\AA^2 = \Spec \CC[x,s]$, with $D^0 = V(x)$ and $D^1 = V(s-x)$. Here,
$s$ is the deformation parameter.
\end{example}

\subsection{Deformations of complete toric varieties}
\label{subsec:deformGlobal}
Consider now a rational com\-plex\-ity-one $T$-variety $X$ defined by a set $\pFan$ of p-divisors as in section \ref{subsec:nonAffineT} on $Y=\PP^1$.  It is possible to construct invariant deformations of $X$ by constructing invariant deformations of $\toric(\pDiv)$ for all $\pDiv\in\pFan$ subject to certain compatibility conditions, see \cite[Section 4]{defrat_tvar}. This construction can in particular be used to obtain deformations of non-affine toric varieties.

There is an especially nice result for the case of complete and smooth toric varieties.
Let $\Sigma$ be a complete fan in $N_\QQ$ describing a smooth toric variety $X$, and let $r$ be as before.  For any ray $\rho\in\Sigma^{(1)}$ with $\langle \rho,r\rangle=1$, let $\Gamma_{\rho}(-r)$ be the graph embedded in $N_\QQ\cap\{r=1\}$ with vertices consisting of rays $\tau\in\Sigma^{(1)}\setminus \rho$ fulfilling $\langle \tau,r\rangle>0$. Two vertices $\tau_1$, $\tau_2$ of $\Gamma_\rho(-r)$ are connected by an edge if they generate a cone in $\Sigma$. After applying any cosection $s^*:N\to N_r$ mapping $\rho$ to the origin, we can consider $\Gamma_\rho(-r)$ to actually be embedded in $N_\QQ\cap r^\perp$.  
Let $\Omega(-r)$ denote the set of all rays $\rho$ of $\Sigma$ such that $\Gamma_\rho(-r)$ is not empty.

\begin{theorem}[\cite{Nathan}, Proposition 2.4]
 There is an isomorphism $$H^1(X,\cT_X)(-r)\cong \bigoplus_{\rho} H^0(\Gamma_\rho(-r),\CC)/\CC$$ describing first-order infinitesimal deformations of $X$.
\end{theorem}

Now fix some $\rho\in\Omega(-r)$ and choose some connected component $C$ of $\Gamma_{\rho}(-r)$. For each $\sigma\in\Sigma$, define a Minkowski decomposition of $\sigma_r$ via
\begin{align*}
\sigma_r=\begin{cases}
\sigma_r+\tail(\sigma_r) & \sigma_r\cap C=\emptyset\\
\tail(\sigma_r)+\sigma_r & \sigma_r\cap C\neq\emptyset.\\
\end{cases}
\end{align*}

\begin{theorem}[\cite{defrat_tvar}, Section 4]
\
\begin{enumerate}
\item The toric deformations of $\toric(\sigma)$ for $\sigma\in\Sigma$ constructed from the above decompositions of $\sigma_r$ glue to an equivariant deformation $\pi(C,\rho,r)$ of $X=\toric(\Sigma)$.
\item Letting $C$ range over all connected components of $\Gamma_{\rho}(-r)$, the images of the $\pi(C,\rho,r)$ under the Kodaira-Spencer map span $H^0(\Gamma_\rho(-r),\CC)/\CC$.
\end{enumerate}
\end{theorem}

\begin{example}
We consider the fourth Hirzebruch surface $\cF_4=\proj_{\PP^1}(\cO\oplus\cO(4))$. This is a toric variety, and a fan $\Sigma$ with $\toric(\Sigma)=\cF_4$ is pictured in Figure \ref{fig:f4tof2}(a). Considering $r=[0,1]$ and $\rho$ the ray through $(0,1)$, we get that $\Gamma_\rho(-r)$ consists of two points. Thus, $\dim T_{\cF_4}^1(-r)=1$. A one-parameter deformation spanning this part of $T_{\cF_4}^1$ can be constructed from the Minkowski decomposition pictured in Figure \ref{fig:f4tof2}(b); the lines $\{\Delta^0\}$ and $\{\Delta^1\}$ picture the $0$th and $1$st summands dilated by a factor of two. This Minkowski decomposition comes from choosing the connected component $C=\{-1/3\}$ of $\Gamma_\rho(-r)$.

Note that the general fiber of this deformation is the second Hirzebruch surface $\cF_2=\proj_{\PP^1}(\cO\oplus\cO(2))$. Indeed, the general fiber is given by a collection of p-divisors on $\PP^1$ with slices $\{\Delta^0\}$ and $\{\Delta^1\}$. Performing a toric upgrade yields the fan in Figure \ref{fig:f4tof2}(c), which is a fan for $\cF_2$.

In these lattice coordinates, $T_{\cF_4}^1$ also has one dimensional components in degrees $-[1,1]$ and $-[2,1]$. The general fibers of the corresponding homogeneous deformations are $\PP^1\times\PP^1$ and $\cF_2$, respectively.
\end{example}

\begin{figure}[h]
  \centering
  \subfloat[Fan for $\cF_4$.]{\hspace*{2ex}\fanhirza\hspace*{1ex}}
  \subfloat[A Minkowski decomposition.]{\hirzmink}
  \subfloat[Fan for $\cF_2$.]{\hspace*{3ex}\fanhirzb\hspace*{3ex}}
  \caption{Deforming $\cF_4$ to $\cF_2$.}
\label{fig:f4tof2}
\end{figure}


\section{Related constructions}
\label{sec:relConst}

\subsection{A non-toric view on p-divisors}
\label{subsec:nonToricP}
We return to the setting of (\ref{subsec:pDiv}) and consider affine $T$-varieties. Let us assume that $\pDiv = \sum_i \Delta_i \otimes Z_i$ is a p-divisor on $Y$. Note that, in contrast to (\ref{subsec:pDiv}),
we have symmetrized the notation of the tensor factors, i.e.\ $\Delta_i\in\Pol^+_\Q(N,\sigma)$ and $Z_i\in \CaDiv_{\geq 0}(Y)$.
In the current section, we moreover adopt another convention
differing from (\ref{subsec:pDiv}).
Setting $Y := \Loc(\pDiv)$, this variety is no longer complete. Nevertheless, it is projective over $Y_0 := \Spec\kG(Y,\CO_Y)$.
Thus, we can now assume that the polyhedra $\Delta_i$ are non-empty.
\\[0.5ex]
If $\Delta := \sum_i\Delta_i$ (which equals $\deg\pDiv$ from (\ref{subsec:degPolytope}) if $Y$ is a complete curve),
then $\normal(\Delta)$ is a fan in $M_\QQ$ which refines the
polyhedral cone $\sigma\dual$. Note that we usually would not construct a toric variety from these data since the fan is given in the ``wrong'' space $M_\QQ$ instead of $N_\QQ$. However, we make an exception here and define
$$
W := \PP(\Delta) = \toric(\normal(\Delta)) \longrightarrow 
\toric(\sigma\dual)=: W_0.
$$
This is a projective map, and the polyhedra $\Delta_i$ can be interpreted as semiample divisors $E_i$ on $W$. Thus, $\pDiv=\sum_i E_i\otimes Z_i$ becomes a ``double divisor'', i.e.\
an element of $\CaDiv_\QQ(W)\otimes_\ZZ \CaDiv_\QQ(Y)$.

\subsection{Symmetrizing equivalences between p-divisors}
\label{subsec:symmRelP}
The equivalence relations on p-divisors mentioned just before Theorem \ref{thm:equivPT} in (\ref{subsec:pDiv}) can also be expressed in the symmetric language of (\ref{subsec:nonToricP}). 
First, the operation of pulling back $\pDiv$ via a modification
$\varphi: Y'\to Y$ can be contrasted with pulling back $\pDiv$ via a 
toric modification $\psi: W'\to W$.
The latter corresponds to giving a subdivision of the fan $\normal(\Delta)$; both $E_i$ and $\psi^*E_i$ correspond to the same polyhedron $\Delta_i$.
\\[0.5ex]
Second, let us recall from (\ref{subsec:pDiv}) the notion of a principal polyhedral divisor
$(a+\sigma)\otimes \div(f)$ with $a\in N$ and $f\in\CC(Y)^*$. In our new setting it is equal to $\div(\chi^{-a})\otimes\div(f)$, i.e.\ the latter can be understood as a ``double principal divisor''.

\subsection{Interpreting evaluations}
\label{subsec:intEval}
Recall from (\ref{subsec:pDiv}) that elements $u\in M$ were associated to an evaluation
$\pDiv(u) = \sum_i \min\langle \Delta_i,u\rangle\cdot Z_i\in \CaDiv_\QQ(Y)$.
To interpret this construction within the language of (\ref{subsec:nonToricP}),
we have to understand the former characters $u$ of $T$ as 
germs of curves in $W$, i.e.\ as maps $u:(\CC,0)\to W$. 
Note that the scalar $\min\langle\Delta_i,u\rangle$ equals the multiplicity of $u^*E_i$ at the origin or, likewise, the intersection number $(E_i\cdot u_*\cO_{\CC^1})$ in $W$. Hence, the sheaf $\cO_Y(\pDiv)$ on $Y$ corresponds to $\sum_i (E_i\cdot\cU) D_i$ where
$\cU := \oplus_{u}\, u_\ast \CO_{\CC^1}$ is a quasi-coherent sheaf
on $W$ with one-dimensional support.

The idea behind the preceding construction is that $\ttoric(\pDiv)\to Y$ looks like a fibration which becomes degenerate over each divisor $Z_i \subset Y$, and the polyhedron $\Delta_i$ tells us what the degeneration looks like.
Translating these polyhedra into divisors $E_i$ on a toric variety $W$ is linked to the fact that the fibers of $\ttoric(\pDiv)\to Y$ are also toric. We hope that this might serve as a general pattern to describe degenerate fibrations of a certain type in a much broader context. We conclude this section with a construction pointing into this direction.

\subsection{Abelian Galois covers of $\PP^1$}
\label{subsec:GalCovP1}
Let $A$ be a finite, abelian group. Every divisor
$\pDiv\in \Div_A^0\PP^1 := A\otimes_\Z\Div^0\PP^1$ of degree $0$ can then be understood as a linear map
$A^*\to\Div^0_{\QQ/\ZZ}\PP^1$, $u\mapsto \pDiv(u)$ where
$A^* := \gHom(A,\QQ/\ZZ)$ denotes the dual group.
Choosing arbitrary lifts $D_u\in\Div^0_{\QQ}\PP^1$
and, afterwards, rational functions $f_{u,v}=f_{v,u}\in \CC(\PP^1)$
which satisfy
$$
D_u + D_v + \div(f_{u,v}) = D_{u+v},
$$
we may, after possibly correcting the functions $f_{uv}$ by suitable constants, assume that $f_{u,v+w}\,f_{v,w} = f_{u,v}\,f_{u+v,w}$.
Then we can define a multiplication via
$$
\begin{array}{ccc}
\CO_{\PP^1}(D_u) \otimes \CO_{\PP^1}(D_v)
&\longrightarrow&
\CO_{\PP^1}(D_{u+v})
\\
f\hspace{0.6em}\otimes\hspace{0.6em} g & \mapsto&
f g\, f_{u,v}^{-1}
\end{array}
$$
which provides the sheaf
\[\CO(\pDiv) := \bigoplus_{u\in A^*} \CO_{\PP^1}(D_u).\]
with an associative and commutative $\CO_{\PP^1}$-algebra structure. Up to isomorphism, the latter does not depend on the choices we have made. Finally, the Galois covering associated to $\pDiv$ is defined
as the relative spectrum $C(\pDiv) := \Spec_{\PP^1}\CO(\pDiv)$.

\begin{proposition}[{\cite[Theorem 3.2]{coxComp1}}]
The relative spectrum $C(\pDiv)$ yields a Galois covering $\pi:C(\pDiv)\to\PP^1$ whose ramification points are contained in $\supp \pDiv$. The ramification index of $P \in C$ equals the order of the $\pDiv$-coefficient of $\pi(P)$ inside $A$.%
\end{proposition}

\begin{example}
\label{ex:cyclicCovering}
If $D\in\Div_{\Z}\PP^1$ is an effective divisor with
$n|\deg D$, then the associated well-known cyclic $n$-fold covering of $\PP^1$ is given, via the above recipe, by understanding $D$ as an element of $\Div^0_{\Z/n\Z}\PP^1$.
\end{example}

\bibliographystyle{alpha}
\bibliography{impanga}


\end{document}